\documentclass[openany]{now}
\copyrightowner{S. Bubeck}
 \volume{8}
 \issue{3-4}
 \pubyear{2015}
 \copyrightyear{2015}
 \isbn{978-1-60198-860-7}
\doi{10.1561/2200000050}
 \firstpage{231}
 \lastpage{358}

\usepackage{amsmath}
\usepackage{amssymb,amsbsy}
\usepackage{graphicx}
\usepackage{booktabs}
\usepackage{dsfont}
\usepackage{natbib}
\usepackage{pdfsync}
\usepackage{hyperref}
\usepackage{tikz}
\usetikzlibrary{arrows,decorations.pathmorphing,backgrounds,positioning,fit,petri,mindmap,shapes.geometric, shapes.misc}
\tikzset{cross/.style={cross out, draw=black, minimum size=2*(#1-\pgflinewidth), inner sep=0pt, outer sep=0pt},
cross/.default={1pt}}

\newcommand{\mI}{\mathrm{I}}
\newcommand{\mB}{\mathrm{B}}
\newcommand{\conv}{\mathrm{conv}}
\newcommand{\inte}{\mathrm{int}}
\newcommand{\tg}{\tilde{g}}
\newcommand{\tx}{\tilde{x}}
\newcommand{\ty}{\tilde{y}}
\newcommand{\tz}{\tilde{z}}

\renewcommand{\phi}{\varphi}

\newcommand{\tr}{\mathrm{Tr}}

\renewcommand{\P}{\mathbb{P}}
\newcommand{\E}{\mathbb{E}}
\newcommand{\N}{\mathbb{N}}
\newcommand{\R}{\mathbb{R}}

\newcommand{\cZ}{\mathcal{Z}}
\newcommand{\cA}{\mathcal{A}}
\newcommand{\cR}{\mathcal{R}}
\newcommand{\cB}{\mathcal{B}}
\newcommand{\cN}{\mathcal{N}}

\newcommand{\cL}{\mathcal{L}}
\newcommand{\cX}{\mathcal{X}}

\newcommand{\cY}{\mathcal{Y}}

\newcommand{\cS}{\mathcal{S}}
\newcommand{\cE}{\mathcal{E}}
\newcommand{\cK}{\mathcal{K}}
\newcommand{\cD}{\mathcal{D}}

\def\ds1{\mathds{1}}
\renewcommand{\epsilon}{\varepsilon}

\newcommand{\argmin}{\mathop{\mathrm{argmin}}}

\renewcommand{\tilde}{\widetilde}

\newlength{\minipagewidth}
\newcommand{\beq}{\begin{equation}}
\newcommand{\eeq}{\end{equation}}

\newcommand{\beqa}{\begin{eqnarray}}
\newcommand{\eeqa}{\end{eqnarray}}

\newcommand{\beqan}{\begin{eqnarray*}}
\newcommand{\eeqan}{\end{eqnarray*}}

\def\ba#1\ea{\begin{align*}#1\end{align*}} 
\def\banum#1\eanum{\begin{align}#1\end{align}} 

\title{Convex Optimization: Algorithms and Complexity}

\author{S\'ebastien Bubeck \\
Theory Group, Microsoft Research \\
sebubeck@microsoft.com}

\begin{document}

\frontmatter

\maketitle

\tableofcontents

\mainmatter

\begin{abstract}
This monograph presents the main complexity theorems in convex optimization and their corresponding algorithms.
Starting from the fundamental theory of black-box optimization, the material progresses towards recent advances in structural optimization and stochastic optimization. Our presentation of black-box optimization, strongly influenced by Nesterov's seminal book and Nemirovski's lecture notes, includes the analysis of cutting plane methods, as well as (accelerated) gradient descent schemes. We also pay special attention to non-Euclidean settings (relevant algorithms include Frank-Wolfe, mirror descent, and dual averaging) and discuss their relevance in machine learning. We provide a gentle introduction to structural optimization with FISTA (to optimize a sum of a smooth and a simple non-smooth term), saddle-point mirror prox (Nemirovski's alternative to Nesterov's smoothing), and a concise description of interior point methods. In stochastic optimization we discuss stochastic gradient descent, mini-batches, random coordinate descent, and sublinear algorithms. We also briefly touch upon convex relaxation of combinatorial problems and the use of randomness to round solutions, as well as random walks based methods.
\end{abstract}

\chapter{Introduction}
\label{intro}
The central objects of our study are convex functions and convex sets in $\R^n$.

\begin{definition}[Convex sets and convex functions]
A set $\cX \subset \R^n$ is said to be convex if it contains all of its segments, that is
$$\forall (x,y,\gamma) \in \cX \times \cX \times [0,1], \; (1-\gamma) x + \gamma y \in \mathcal{X}.$$
A function $f : \mathcal{X} \rightarrow \R$ is said to be convex if it always lies below its chords, that is
$$ \forall (x,y,\gamma) \in \cX \times \cX \times [0,1], \; f((1-\gamma) x + \gamma y) \leq (1-\gamma)f(x) + \gamma f(y) .$$
\end{definition}
We are interested in algorithms that take as input a convex set $\cX$ and a convex function $f$ and output an approximate minimum of $f$ over $\cX$. We write compactly the problem of finding the minimum of $f$ over $\cX$ as
\begin{align*}
& \mathrm{min.} \; f(x) \\
& \text{s.t.} \; x \in \cX .
\end{align*}
In the following we will make more precise how the set of constraints $\cX$ and the objective function $f$ are specified to the algorithm. Before that we proceed to give a few important examples of convex optimization problems in machine learning.

\section{Some convex optimization problems in machine learning} \label{sec:mlapps}
Many fundamental convex optimization problems in machine learning take the following form:
\begin{equation} \label{eq:veryfirst}
\underset{x \in \R^n}{\mathrm{min.}} \; \sum_{i=1}^m f_i(x) + \lambda \cR(x) ,
\end{equation}
where the functions $f_1, \hdots, f_m, \cR$ are convex and $\lambda \geq 0$ is a fixed parameter. The interpretation is that $f_i(x)$ represents the cost of using $x$ on the $i^{th}$ element of some data set, and $\cR(x)$ is a regularization term which enforces some ``simplicity'' in $x$. We discuss now major instances of \eqref{eq:veryfirst}. In all cases one has a data set of the form $(w_i, y_i) \in \R^n \times \cY, i=1, \hdots, m$ and the cost function $f_i$ depends only on the pair $(w_i, y_i)$. We refer to \cite{HTF01, SS02, SSS14} for more details on the origin of these important problems. The mere objective of this section is to expose the reader to a few concrete convex optimization problems which are routinely solved.
 
In classification one has $\cY = \{-1,1\}$. Taking $f_i(x) = \max(0, 1- y_i x^{\top} w_i)$ (the so-called hinge loss) and $\cR(x) = \|x\|_2^2$ one obtains the SVM problem. On the other hand taking $f_i(x) = \log(1 + \exp(-y_i x^{\top} w_i) )$ (the logistic loss) and again $\cR(x) = \|x\|_2^2$ one obtains the (regularized) logistic regression problem.

In regression one has $\cY = \R$. Taking $f_i(x) = (x^{\top} w_i - y_i)^2$ and $\cR(x) = 0$ one obtains the vanilla least-squares problem which can be rewritten in vector notation as
$$\underset{x \in \R^n}{\mathrm{min.}} \; \|W x - Y\|_2^2 ,$$
where $W \in \R^{m \times n}$ is the matrix with $w_i^{\top}$ on the $i^{th}$ row and $Y =(y_1, \hdots, y_n)^{\top}$.
With $\cR(x) = \|x\|_2^2$ one obtains the ridge regression problem, while with $\cR(x) = \|x\|_1$ this is the LASSO problem \cite{Tib96}.

Our last two examples are of a slightly different flavor. In particular the design variable $x$ is now best viewed as a matrix, and thus we denote it by a capital letter $X$. The sparse inverse covariance estimation problem can be written as follows, given some empirical covariance matrix $Y$,
\begin{align*}
& \mathrm{min.} \; \mathrm{Tr}(X Y) - \mathrm{logdet}(X) + \lambda \|X\|_1 \\
& \text{s.t.} \; X \in \R^{n \times n}, X^{\top} = X, X \succeq 0 .
\end{align*}
Intuitively the above problem is simply a regularized maximum likelihood estimator (under a Gaussian assumption). 

Finally we introduce the convex version of the matrix completion problem. Here our data set consists of observations of some of the entries of an unknown matrix $Y$, and we want to ``complete" the unobserved entries of $Y$ in such a way that the resulting matrix is ``simple" (in the sense that it has low rank). After some massaging (see \cite{CR09}) the (convex) matrix completion problem can be formulated as follows:
\begin{align*}
& \mathrm{min.} \; \mathrm{Tr}(X) \\
& \text{s.t.} \; X \in \R^{n \times n}, X^{\top} = X, X \succeq 0, X_{i,j} = Y_{i,j} \; \text{for} \; (i,j) \in \Omega ,
\end{align*}
where $\Omega \subset [n]^2$ and $(Y_{i,j})_{(i,j) \in \Omega}$ are given. 

\section{Basic properties of convexity}
A basic result about convex sets that we shall use extensively is the Separation Theorem.

\begin{theorem}[Separation Theorem]
Let $\mathcal{X} \subset \R^n$ be a closed convex set, and $x_0 \in \R^n \setminus \mathcal{X}$. Then, there exists $w \in \R^n$ and $t \in \R$ such that
$$w^{\top} x_0 < t, \; \text{and} \; \forall x \in \mathcal{X}, w^{\top} x \geq t.$$
\end{theorem}

Note that if $\mathcal{X}$ is not closed then one can only guarantee that $w^{\top} x_0 \leq w^{\top} x, \forall x \in \mathcal{X}$ (and $w \neq 0$). This immediately implies the Supporting Hyperplane Theorem ($\partial \cX$ denotes the boundary of $\cX$, that is the closure without the interior):

\begin{theorem}[Supporting Hyperplane Theorem]
Let $\mathcal{X} \subset \R^n$ be a convex set, and $x_0 \in \partial \mathcal{X}$. Then, there exists $w \in \R^n, w \neq 0$ such that
$$\forall x \in \mathcal{X}, w^{\top} x \geq w^{\top} x_0.$$
\end{theorem}

We introduce now the key notion of {\em subgradients}.

\begin{definition}[Subgradients]
Let $\mathcal{X} \subset \R^n$, and $f : \mathcal{X} \rightarrow \R$. Then $g \in \R^n$ is a subgradient of $f$ at $x \in \mathcal{X}$ if for any $y \in \mathcal{X}$ one has
$$f(x) - f(y) \leq g^{\top} (x - y) .$$
The set of subgradients of $f$ at $x$ is denoted $\partial f (x)$.
\end{definition}
To put it differently, for any $x \in \cX$ and $g \in \partial f(x)$, $f$ is above the linear function $y \mapsto f(x) + g^{\top} (y-x)$. The next result shows (essentially) that a convex functions always admit subgradients.

\begin{proposition}[Existence of subgradients] \label{prop:existencesubgradients}
Let $\mathcal{X} \subset \R^n$ be convex, and $f : \mathcal{X} \rightarrow \R$. If $\forall x \in \mathcal{X}, \partial f(x) \neq \emptyset$ then $f$ is convex. Conversely if $f$ is convex then for any $x \in \mathrm{int}(\mathcal{X}), \partial f(x) \neq \emptyset$. Furthermore if $f$ is convex and differentiable at $x$ then $\nabla f(x) \in \partial f(x)$. 
\end{proposition}

Before going to the proof we recall the definition of the epigraph of a function $f : \mathcal{X} \rightarrow \R$:
$$\mathrm{epi}(f) = \{(x,t) \in \mathcal{X} \times \R : t \geq f(x) \} .$$
It is obvious that a function is convex if and only if its epigraph is a convex set.

\begin{proof}
The first claim is almost trivial: let $g \in \partial f((1-\gamma) x + \gamma y)$, then by definition one has
\begin{eqnarray*}
& & f((1-\gamma) x + \gamma y) \leq f(x) + \gamma g^{\top} (y - x) , \\
& & f((1-\gamma) x + \gamma y) \leq f(y) + (1-\gamma) g^{\top} (x - y) ,
\end{eqnarray*}
which clearly shows that $f$ is convex by adding the two (appropriately rescaled) inequalities.
\newline

Now let us prove that a convex function $f$ has subgradients in the interior of $\mathcal{X}$. We build a subgradient by using a supporting hyperplane to the epigraph of the function. Let $x \in \mathcal{X}$. Then clearly $(x,f(x)) \in \partial \mathrm{epi}(f)$, and $\mathrm{epi}(f)$ is a convex set. Thus by using the Supporting Hyperplane Theorem, there exists $(a,b) \in \R^n \times \R$ such that
\begin{equation} \label{eq:supphyp}
a^{\top} x + b f(x) \geq a^{\top} y + b t, \forall (y,t) \in \mathrm{epi}(f) .
\end{equation}
Clearly, by letting $t$ tend to infinity, one can see that $b \leq 0$. Now let us assume that $x$ is in the interior of $\mathcal{X}$. Then for $\epsilon > 0$ small enough, $y=x + \epsilon a \in \mathcal{X}$, which implies that $b$ cannot be equal to $0$ (recall that if $b=0$ then necessarily $a \neq 0$ which allows to conclude by contradiction). Thus rewriting \eqref{eq:supphyp} for $t=f(y)$ one obtains
$$f(x) - f(y) \leq \frac{1}{|b|} a^{\top} (x - y) .$$
Thus $a / |b| \in \partial f(x)$ which concludes the proof of the second claim.
\newline

Finally let $f$ be a convex and differentiable function. Then by definition:
\begin{eqnarray*}
f(y) & \geq & \frac{f((1-\gamma) x + \gamma y) - (1- \gamma) f(x)}{\gamma} \\
& = & f(x) + \frac{f(x + \gamma (y - x)) - f(x)}{\gamma} \\
& \underset{\gamma \to 0}{\to} & f(x) + \nabla f(x)^{\top} (y-x),
\end{eqnarray*}
which shows that $\nabla f(x) \in \partial f(x)$.
\end{proof}

In several cases of interest the set of contraints can have an empty interior, in which case the above proposition does not yield any information. However it is easy to replace $\mathrm{int}(\cX)$ by $\mathrm{ri}(\cX)$ -the relative interior of $\cX$- which is defined as the interior of $\cX$ when we view it as subset of the affine subspace it generates. Other notions of convex analysis will prove to be useful in some parts of this text. In particular the notion of {\em closed convex functions} is convenient to exclude pathological cases: these are the convex functions with closed epigraphs. Sometimes it is also useful to consider the extension of a convex function $f: \cX \rightarrow \R$ to a function from $\R^n$ to $\overline{\R}$ by setting $f(x)= + \infty$ for $x \not\in \cX$. In convex analysis one uses the term {\em proper convex function} to denote a convex function with values in $\R \cup \{+\infty\}$ such that there exists $x \in \R^n$ with $f(x) < +\infty$. \textbf{From now on all convex functions will be closed, and if necessary we consider also their proper extension.} We refer the reader to \cite{Roc70} for an extensive discussion of these notions.

\section{Why convexity?}
The key to the algorithmic success in minimizing convex functions is that these functions exhibit a {\em local to global} phenomenon. We have already seen one instance of this in Proposition \ref{prop:existencesubgradients}, where we showed that $\nabla f(x) \in \partial f(x)$: the gradient $\nabla f(x)$ contains a priori only local information about the function $f$ around $x$ while the subdifferential $\partial f(x)$ gives a global information in the form of a linear lower bound on the entire function. Another instance of this local to global phenomenon is that local minima of convex functions are in fact global minima:

\begin{proposition}[Local minima are global minima]
Let $f$ be convex. If $x$ is a local minimum of $f$ then $x$ is a global minimum of $f$. Furthermore this happens if and only if $0 \in \partial f(x)$.
\end{proposition}

\begin{proof}
Clearly $0 \in \partial f(x)$ if and only if $x$ is a global minimum of $f$. Now assume that $x$ is local minimum of $f$. Then for $\gamma$ small enough one has for any $y$,
$$f(x) \leq f((1-\gamma) x + \gamma y) \leq (1-\gamma) f(x) + \gamma f(y) ,$$
which implies $f(x) \leq f(y)$ and thus $x$ is a global minimum of $f$.
\end{proof}

The nice behavior of convex functions will allow for very fast algorithms to optimize them. This alone would not be sufficient to justify the importance of this class of functions (after all constant functions are pretty easy to optimize). However it turns out that surprisingly many optimization problems admit a convex (re)formulation. The excellent book \cite{BV04} describes in great details the various methods that one can employ to uncover the convex aspects of an optimization problem. We will not repeat these arguments here, but we have already seen that many famous machine learning problems (SVM, ridge regression, logistic regression, LASSO, sparse covariance estimation, and matrix completion) are formulated as convex problems.

We conclude this section with a simple extension of the optimality condition ``$0 \in \partial f(x)$'' to the case of constrained optimization. We state this result in the case of a differentiable function for sake of simplicity.
\begin{proposition}[First order optimality condition] \label{prop:firstorder}
Let $f$ be convex and $\cX$ a closed convex set on which $f$ is differentiable. Then
$$x^* \in \argmin_{x \in \cX} f(x) ,$$
if and only if one has
$$\nabla f(x^*)^{\top}(x^*-y) \leq 0, \forall y \in \cX .$$
\end{proposition}

\begin{proof}
The ``if" direction is trivial by using that a gradient is also a subgradient. For the ``only if" direction it suffices to note that if $\nabla f(x)^{\top} (y-x) < 0$, then $f$ is locally decreasing around $x$ on the line to $y$ (simply consider $h(t) = f(x + t (y-x))$ and note that $h'(0) = \nabla f(x)^{\top} (y-x)$).
\end{proof}

\section{Black-box model} \label{sec:blackbox}
We now describe our first model of ``input" for the objective function and the set of constraints. In the black-box model we assume that we have unlimited computational resources, the set of constraint $\cX$ is known, and the objective function $f: \cX \rightarrow \R$ is unknown but can be accessed through queries to {\em oracles}:
\begin{itemize}
\item A zeroth order oracle takes as input a point $x \in \cX$ and outputs the value of $f$ at $x$.
\item A first order oracle takes as input a point $x \in \cX$ and outputs a subgradient of $f$ at $x$.
\end{itemize}
In this context we are interested in understanding the {\em oracle complexity} of convex optimization, that is how many queries to the oracles are necessary and sufficient to find an $\epsilon$-approximate minima of a convex function. To show an upper bound on the sample complexity we need to propose an algorithm, while lower bounds are obtained by information theoretic reasoning (we need to argue that if the number of queries is ``too small" then we don't have enough information about the function to identify an $\epsilon$-approximate solution).

From a mathematical point of view, the strength of the black-box model is that it will allow us to derive a {\em complete} theory of convex optimization, in the sense that we will obtain matching upper and lower bounds on the oracle complexity for various subclasses of interesting convex functions. While the model by itself does not limit our computational resources (for instance any operation on the constraint set $\cX$ is allowed) we will of course pay special attention to the algorithms' {\em computational complexity} (i.e., the number of elementary operations that the algorithm needs to do). We will also be interested in the situation where the set of constraint $\cX$ is unknown and can only be accessed through a {\em separation oracle}: given $x \in \R^n$, it outputs either that $x$ is in $\mathcal{X}$, or if $x \not\in \mathcal{X}$ then it outputs a separating hyperplane between $x$ and $\mathcal{X}$. 

The black-box model was essentially developed in the early days of convex optimization (in the Seventies) with \cite{NY83} being still an important reference for this theory (see also \cite{Nem95}). In the recent years this model and the corresponding algorithms have regained a lot of popularity, essentially for two reasons:
\begin{itemize}
\item It is possible to develop algorithms with dimension-free oracle complexity which is quite attractive for optimization problems in very high dimension.
\item Many algorithms developed in this model are robust to noise in the output of the oracles. This is especially interesting for stochastic optimization, and very relevant to machine learning applications. We will explore this in details in Chapter \ref{rand}.
\end{itemize}
Chapter \ref{finitedim}, Chapter \ref{dimfree} and Chapter \ref{mirror} are dedicated to the study of the black-box model (noisy oracles are discussed in Chapter \ref{rand}). We do not cover the setting where only a zeroth order oracle is available, also called derivative free optimization, and we refer to \cite{CSV09, ABM11} for further references on this.

\section{Structured optimization} \label{sec:structured}
The black-box model described in the previous section seems extremely wasteful for the applications we discussed in Section \ref{sec:mlapps}. Consider for instance the LASSO objective: $x \mapsto \|W x - y\|_2^2 + \|x\|_1$. We know this function {\em globally}, and assuming that we can only make local queries through oracles seem like an artificial constraint for the design of algorithms. Structured optimization tries to address this observation. Ultimately one would like to take into account the global structure of both $f$ and $\cX$ in order to propose the most efficient optimization procedure. An extremely powerful hammer for this task are the Interior Point Methods. We will describe this technique in Chapter \ref{beyond} alongside with other more recent techniques such as FISTA or Mirror Prox. 

We briefly describe now two classes of optimization problems for which we will be able to exploit the structure very efficiently, these are the LPs (Linear Programs) and SDPs (Semi-Definite Programs). \cite{BN01} describe a more general class of Conic Programs but we will not go in that direction here.

The class LP consists of problems where $f(x) = c^{\top} x$ for some $c \in \R^n$, and $\mathcal{X} = \{x \in \R^n : A x \leq b \}$ for some $A \in \R^{m \times n}$ and $b \in \R^m$.

The class SDP consists of problems where the optimization variable is a symmetric matrix $X \in \R^{n \times n}$. Let $\mathbb{S}^n$ be the space of $n\times n$ symmetric matrices (respectively $\mathbb{S}^n_+$ is the space of positive semi-definite matrices), and let $\langle \cdot, \cdot \rangle$ be the Frobenius inner product (recall that it can be written as $\langle A, B \rangle = \tr(A^{\top} B)$). In the class SDP the problems are of the following form: $f(x) = \langle X, C \rangle$ for some $C \in \R^{n \times n}$, and $\mathcal{X} = \{X \in \mathbb{S}^n_+ : \langle X, A_i \rangle \leq b_i, i \in \{1, \hdots, m\} \}$ for some $A_1, \hdots, A_m \in \R^{n \times n}$ and $b \in \R^m$. Note that the matrix completion problem described in Section \ref{sec:mlapps} is an example of an SDP.

\section{Overview of the results and disclaimer}
The overarching aim of this monograph is to present the main complexity theorems in convex optimization and the corresponding algorithms. We focus on five major results in convex optimization which give the overall structure of the text: the existence of efficient cutting-plane methods with optimal oracle complexity (Chapter \ref{finitedim}), a complete characterization of the relation between first order oracle complexity and curvature in the objective function (Chapter \ref{dimfree}), first order methods beyond Euclidean spaces (Chapter \ref{mirror}), non-black box methods (such as interior point methods) can give a quadratic improvement in the number of iterations with respect to optimal black-box methods (Chapter \ref{beyond}), and finally noise robustness of first order methods (Chapter \ref{rand}). Table \ref{table} can be used as a quick reference to the results proved in Chapter \ref{finitedim} to Chapter \ref{beyond}, as well as some of the results of Chapter \ref{rand} (this last chapter is the most relevant to machine learning but the results are also slightly more specific which make them harder to summarize).

An important disclaimer is that the above selection leaves out methods derived from duality arguments, as well as the two most popular research avenues in convex optimization: (i) using convex optimization in non-convex settings, and (ii) practical large-scale algorithms. Entire books have been written on these topics, and new books have yet to be written on the impressive collection of new results obtained for both (i) and (ii) in the past five years. 

A few of the blatant omissions regarding (i) include (a) the theory of submodular optimization (see \cite{Bac13}), (b) convex relaxations of combinatorial problems (a short example is given in Section \ref{sec:convexrelaxation}), and (c) methods inspired from convex optimization for non-convex problems such as low-rank matrix factorization (see e.g. \cite{JNS13} and references therein), neural networks optimization, etc. 

With respect to (ii) the most glaring omissions include (a) heuristics (the only heuristic briefly discussed here is the non-linear conjugate gradient in Section \ref{sec:CG}), (b) methods for distributed systems, and (c) adaptivity to unknown parameters. Regarding (a) we refer to \cite{NW06} where the most practical algorithms are discussed in great details (e.g., quasi-newton methods such as BFGS and L-BFGS, primal-dual interior point methods, etc.). The recent survey \cite{BPCPE11} discusses the alternating direction method of multipliers (ADMM) which is a popular method to address (b). Finally (c) is a subtle and important issue. In the entire monograph the emphasis is on presenting the algorithms and proofs in the simplest way, and thus for sake of convenience we assume that the relevant parameters describing the regularity and curvature of the objective function (Lipschitz constant, smoothness constant, strong convexity parameter) are known and can be used to tune the algorithm's own parameters. Line search is a powerful technique to replace the knowledge of these parameters and it is heavily used in practice, see again \cite{NW06}. We observe however that from a theoretical point of view (c) is only a matter of logarithmic factors as one can always run in parallel several copies of the algorithm with different guesses for the values of the parameters\footnote{Note that this trick does not work in the context of Chapter \ref{rand}.}. Overall the attitude of this text with respect to (ii) is best summarized by a quote of Thomas Cover: ``theory is the first term in the Taylor series of practice'', \cite{Cov92}.
\newline

\textbf{Notation.} We always denote by $x^*$ a point in $\cX$ such that $f(x^*) = \min_{x \in \cX} f(x)$ (note that the optimization problem under consideration will always be clear from the context). In particular we always assume that $x^*$ exists. For a vector $x \in \R^n$ we denote by $x(i)$ its $i^{th}$ coordinate. The dual of a norm $\|\cdot\|$ (defined later) will be denoted either $\|\cdot\|_*$ or $\|\cdot\|^*$ (depending on whether the norm already comes with a subscript). Other notation are standard (e.g., $\mI_n$ for the $n \times n$ identity matrix, $\succeq$ for the positive semi-definite order on matrices, etc).

\begin{center}
\begin{table}
{\footnotesize
\begin{tabular}{c|c|c|c|c}
$f$ & {Algorithm} & {Rate} & {\# Iter} & {Cost/iter} 
\\  \hline 
  {\begin{tabular}{c} non-smooth \end{tabular}} &  {\begin{tabular}{c} center of\\ gravity \end{tabular}} & $\exp\left( - \frac{t}{n} \right)$ & $n \log \left(\frac{1}{\epsilon}\right)$ &  {\begin{tabular}{c} 1 $\nabla$, \\ 1 $n$-dim $\int$ \end{tabular}}
\\  \hline 
  {\begin{tabular}{c} non-smooth \end{tabular}} &  {\begin{tabular}{c} ellipsoid \\ method \end{tabular}} & $\frac{R}{r} \exp\left( - \frac{t}{n^2}\right)$ & $n^2 \log \left(\frac{R}{r \epsilon}\right)$ &  {\begin{tabular}{c} 1 $\nabla$, \\mat-vec $\times$ \end{tabular}}
\\  \hline 
  {\begin{tabular}{c} non-smooth \end{tabular}} &  {\begin{tabular}{c} Vaidya \end{tabular}} & $\frac{R n}{r} \exp\left( - \frac{t}{n}\right)$ & $n \log \left(\frac{R n}{r \epsilon}\right)$ &  {\begin{tabular}{c} 1 $\nabla$, \\mat-mat $\times$ \end{tabular}}
\\  \hline 
  {\begin{tabular}{c} quadratic \end{tabular}} &  {\begin{tabular}{c} CG \end{tabular}} & {\begin{tabular}{c} exact \\ $\exp\left( - \frac{t}{\kappa}\right)$ \end{tabular}} & {\begin{tabular}{c} $n$ \\ $\kappa \log\left(\frac1{\epsilon}\right)$ \end{tabular}} &  {\begin{tabular}{c} 1 $\nabla$ \end{tabular}}
\\  \hline 
  {\begin{tabular}{c} non-smooth, \\ Lipschitz \end{tabular}} & {PGD} & $R L /\sqrt{t}$ & $R^2 L^2 /\epsilon^2$ &  {\begin{tabular}{c} 1 $\nabla$, \\ 1 proj. \end{tabular}}
\\  \hline 
  {smooth} & {PGD} & $\beta R^2 / t$ & $\beta R^2 /\epsilon$ &  {\begin{tabular}{c} 1 $\nabla$, \\ 1 proj. \end{tabular}}
\\  \hline 
  {smooth} & {\begin{tabular}{c} AGD \end{tabular}} & $\beta R^2 / t^2$ & $R \sqrt{\beta / \epsilon}$ & 1 $\nabla$ 
\\  \hline 
  {\begin{tabular}{c} smooth \\ (any norm) \end{tabular}}& {FW} & $\beta R^2 / t$ & $\beta R^2 /\epsilon$ &  {\begin{tabular}{c} 1 $\nabla$, \\ 1 LP \end{tabular}}
\\  \hline 
  {\begin{tabular}{c} strong. conv., \\ Lipschitz \end{tabular}} & {PGD} & $L^2 / (\alpha t)$ & $L^2 / (\alpha \epsilon)$ & {\begin{tabular}{c} 1 $\nabla$ , \\ 1 proj. \end{tabular}}
\\  \hline 
  {\begin{tabular}{c} strong. conv., \\ smooth \end{tabular}} & {PGD} & $R^2 \exp\left(-\frac{t}{\kappa}\right)$ & $\kappa \log\left(\frac{R^2}{\epsilon}\right) $ & {\begin{tabular}{c} 1 $\nabla$ , \\ 1 proj. \end{tabular}}
\\  \hline 
  {\begin{tabular}{c} strong. conv., \\ smooth \end{tabular}} & {\begin{tabular}{c} AGD \end{tabular}} & $R^2 \exp\left(-\frac{t}{\sqrt{\kappa}}\right)$ & $\sqrt{\kappa} \log\left(\frac{R^2}{\epsilon}\right) $ & 1 $\nabla$ 
\\  \hline 
  {\begin{tabular}{c} $f+g$, \\ $f$ smooth, \\ $g$ simple \end{tabular}} & FISTA & $\beta R^2 / t^2$ & $R \sqrt{\beta / \epsilon}$ & {\begin{tabular}{c} 1 $\nabla$  of $f$ \\ Prox of $g$ \end{tabular}} 
\\  \hline 
  {\begin{tabular}{c} $\underset{y \in \cY}{\max} \ \phi(x,y)$, \\ $\phi$ smooth\end{tabular}} & SP-MP & $\beta R^2 / t$ & $\beta R^2 /\epsilon$ & {\begin{tabular}{c} MD on $\cX$ \\ MD on $\cY$ \end{tabular}} 
\\  \hline 
  {\begin{tabular}{c} linear, \\ $\cX$ with $F$ \\$\nu$-self-conc. \end{tabular}} & IPM & $\nu \exp\left(- \frac{t}{\sqrt{\nu}}\right)$ & $\sqrt{\nu} \log\left(\frac{\nu}{\epsilon}\right)$ & {\begin{tabular}{c} Newton \\ step on $F$ \end{tabular}}
\\ \hline
  {\begin{tabular}{c} non-smooth \end{tabular}} & {SGD} & $B L /\sqrt{t}$ & $B^2 L^2 /\epsilon^2$ &  {\begin{tabular}{c} 1 stoch. ${\nabla}$, \\ 1 proj. \end{tabular}}
\\ \hline
  {\begin{tabular}{c} non-smooth, \\ strong. conv. \end{tabular}} & {SGD} & $B^2 / (\alpha t)$ & $B^2 / (\alpha \epsilon)$ &  {\begin{tabular}{c} 1 stoch. $\nabla$, \\ 1 proj. \end{tabular}}
\\ \hline
  {\begin{tabular}{c} $f=\frac1{m} \sum f_i$ \\ $f_i$ smooth \\ strong. conv. \end{tabular}} & {SVRG} & -- & $(m + \kappa) \log\left(\frac{1}{\epsilon}\right)$ &  {1 stoch. $\nabla$}
  \end{tabular}}
\caption{Summary of the results proved in Chapter \ref{finitedim} to Chapter \ref{beyond} and some of the results in Chapter \ref{rand}.}
\label{table}
\end{table}
\end{center}


\chapter{Convex optimization in finite dimension}
\label{finitedim}
Let $\mathcal{X} \subset \R^n$ be a convex body (that is a compact convex set with non-empty interior), and $f : \mathcal{X} \rightarrow [-B,B]$ be a continuous and convex function. Let $r, R>0$ be such that $\mathcal{X}$ is contained in an Euclidean ball of radius $R$ (respectively it contains an Euclidean ball of radius $r$). In this chapter we give several black-box algorithms to solve 
\begin{align*}
& \mathrm{min.} \; f(x) \\
& \text{s.t.} \; x \in \cX .
\end{align*}
As we will see these algorithms have an oracle complexity which is linear (or quadratic) in the dimension, hence the title of the chapter (in the next chapter the oracle complexity will be {\em independent} of the dimension). An interesting feature of the methods discussed here is that they only need a separation oracle for the constraint set $\cX$. In the literature such algorithms are often referred to as {\em cutting plane methods}. In particular these methods can be used to {\em find} a point $x \in \cX$ given only a separating oracle for $\cX$ (this is also known as the {\em feasibility problem}). 

\section{The center of gravity method} \label{sec:gravity}
We consider the following simple iterative algorithm\footnote{As a warm-up we assume in this section that $\cX$ is known. It should be clear from the arguments in the next section that in fact the same algorithm would work if initialized with $\cS_1 \supset \cX$.}: let $\cS_1= \cX$, and for $t \geq 1$ do the following:
\begin{enumerate}
\item Compute
\begin{equation}
c_t = \frac{1}{\mathrm{vol}(\cS_t)} \int_{x \in \cS_t} x dx .
\end{equation}
\item Query the first order oracle at $c_t$ and obtain $w_t \in \partial f (c_t)$. Let
$$\cS_{t+1} = \cS_t \cap \{x \in \R^n : (x-c_t)^{\top} w_t \leq 0\} .$$
\end{enumerate}
If stopped after $t$ queries to the first order oracle then we use $t$ queries to a zeroth order oracle to output
$$x_t\in \argmin_{1 \leq r \leq t} f(c_r) .$$
This procedure is known as the {\em center of gravity method}, it was discovered independently on both sides of the Wall by \cite{Lev65} and \cite{New65}.

\begin{theorem}
\label{th:centerofgravity}
The center of gravity method satisfies
$$f(x_t) - \min_{x \in \mathcal{X}} f(x) \leq 2 B \left(1 - \frac1{e} \right)^{t/n} .$$
\end{theorem}
Before proving this result a few comments are in order. 

To attain an $\epsilon$-optimal point the center of gravity method requires $O( n \log (2 B / \epsilon))$ queries to both the first and zeroth order oracles. It can be shown that this is the best one can hope for, in the sense that for $\epsilon$ small enough one needs $\Omega(n \log(1/ \epsilon))$ calls to the oracle in order to find an $\epsilon$-optimal point, see \cite{NY83} for a formal proof.

The rate of convergence given by Theorem \ref{th:centerofgravity} is exponentially fast. In the optimization literature this is called a {\em linear rate} as the (estimated) error at iteration $t+1$ is linearly related to the error at iteration $t$.

The last and most important comment concerns the computational complexity of the method. It turns out that finding the center of gravity $c_t$ is a very difficult problem by itself, and we do not have computationally efficient procedure to carry out this computation in general. In Section \ref{sec:rwmethod} we will discuss a relatively recent (compared to the 50 years old center of gravity method!) randomized algorithm to approximately compute the center of gravity. This will in turn give a randomized center of gravity method which we will describe in detail.

We now turn to the proof of Theorem \ref{th:centerofgravity}. We will use the following elementary result from convex geometry:

\begin{lemma}[\cite{Gru60}] \label{lem:Gru60}
Let $\cK$ be a centered convex set, i.e., $\int_{x \in \cK} x dx = 0$, then for any $w \in \R^n, w \neq 0$, one has
$$\mathrm{Vol} \left( \cK \cap \{x \in \R^n : x^{\top} w \geq 0\} \right) \geq \frac{1}{e} \mathrm{Vol} (\cK) .$$
\end{lemma}
We now prove Theorem \ref{th:centerofgravity}.

\begin{proof}
Let $x^*$ be such that $f(x^*) = \min_{x \in \mathcal{X}} f(x)$. Since $w_t \in \partial f(c_t)$ one has
$$f(c_t) - f(x) \leq w_t^{\top} (c_t - x) .$$
and thus
\begin{equation} \label{eq:centerofgravity1}
\cS_{t} \setminus \cS_{t+1} \subset \{x \in \cX : (x-c_t)^{\top} w_t > 0\} \subset \{x \in \cX : f(x) > f(c_t)\} ,
\end{equation}
which clearly implies that one can never remove the optimal point from our sets in consideration, that is $x^* \in \cS_t$ for any $t$.
Without loss of generality we can assume that we always have $w_t \neq 0$, for otherwise one would have $f(c_t) = f(x^*)$ which immediately conludes the proof. Now using that $w_t \neq 0$ for any $t$ and Lemma \ref{lem:Gru60} one clearly obtains
$$\mathrm{vol}(\cS_{t+1}) \leq \left(1 - \frac1{e} \right)^t \mathrm{vol}(\cX) .$$
For $\epsilon \in [0,1]$, let $\mathcal{X}_{\epsilon} = \{(1-\epsilon) x^* + \epsilon x, x \in \mathcal{X}\}$. Note that $\mathrm{vol}(\mathcal{X}_{\epsilon}) = \epsilon^n \mathrm{vol}(\mathcal{X})$. These volume computations show that for $\epsilon > \left(1 - \frac1{e} \right)^{t/n}$ one has $\mathrm{vol}(\mathcal{X}_{\epsilon}) > \mathrm{vol}(\cS_{t+1})$. In particular this implies that for $\epsilon > \left(1 - \frac1{e} \right)^{t/n}$, there must exist a time $r \in \{1,\hdots, t\}$, and $x_{\epsilon} \in \mathcal{X}_{\epsilon}$, such that $x_{\epsilon} \in \cS_{r}$ and $x_{\epsilon} \not\in \cS_{r+1}$. In particular by \eqref{eq:centerofgravity1} one has $f(c_r) < f(x_{\epsilon})$. On the other hand by convexity of $f$ one clearly has $f(x_{\epsilon}) \leq f(x^*) + 2 \epsilon B$. This concludes the proof.
\end{proof}

\section{The ellipsoid method} \label{sec:ellipsoid}
Recall that an ellipsoid is a convex set of the form
$$\mathcal{E} = \{x \in \R^n : (x - c)^{\top} H^{-1} (x-c) \leq 1 \} ,$$
where $c \in \R^n$, and $H$ is a symmetric positive definite matrix. Geometrically $c$ is the center of the ellipsoid, and the semi-axes of $\cE$ are given by the eigenvectors of $H$, with lengths given by the square root of the corresponding eigenvalues.

We give now a simple geometric lemma, which is at the heart of the ellipsoid method.
\begin{lemma} \label{lem:geomellipsoid}
Let $\mathcal{E}_0 = \{x \in \R^n : (x - c_0)^{\top} H_0^{-1} (x-c_0) \leq 1 \}$. For any $w \in \R^n$, $w \neq 0$, there exists an ellipsoid $\mathcal{E}$ such that
\begin{equation}
\mathcal{E} \supset \{x \in \mathcal{E}_0 : w^{\top} (x-c_0) \leq 0\} , \label{eq:ellipsoidlemma1}
\end{equation}
and 
\begin{equation}
\mathrm{vol}(\mathcal{E}) \leq \exp \left(- \frac{1}{2 n} \right) \mathrm{vol}(\mathcal{E}_0) . \label{eq:ellipsoidlemma2}
\end{equation}
Furthermore for $n \geq 2$ one can take $\cE = \{x \in \R^n : (x - c)^{\top} H^{-1} (x-c) \leq 1 \}$ where
\begin{align}
& c = c_0 - \frac{1}{n+1} \frac{H_0 w}{\sqrt{w^{\top} H_0 w}} , \label{eq:ellipsoidlemma3}\\
& H = \frac{n^2}{n^2-1} \left(H_0 - \frac{2}{n+1} \frac{H_0 w w^{\top} H_0}{w^{\top} H_0 w} \right) . \label{eq:ellipsoidlemma4}
\end{align}
\end{lemma}

\begin{proof}
For $n=1$ the result is obvious, in fact we even have $\mathrm{vol}(\mathcal{E}) \leq \frac12 \mathrm{vol}(\mathcal{E}_0) .$

For $n \geq 2$ one can simply verify that the ellipsoid given by \eqref{eq:ellipsoidlemma3} and \eqref{eq:ellipsoidlemma4} satisfy the required properties \eqref{eq:ellipsoidlemma1} and \eqref{eq:ellipsoidlemma2}. Rather than bluntly doing these computations we will show how to derive \eqref{eq:ellipsoidlemma3} and \eqref{eq:ellipsoidlemma4}. As a by-product this will also show that the ellipsoid defined by \eqref{eq:ellipsoidlemma3} and \eqref{eq:ellipsoidlemma4} is the unique ellipsoid of minimal volume that satisfy \eqref{eq:ellipsoidlemma1}. Let us first focus on the case where $\mathcal{E}_0$ is the Euclidean ball $\cB = \{x \in \R^n : x^{\top} x \leq 1\}$. We momentarily assume that $w$ is a unit norm vector. 

By doing a quick picture, one can see that it makes sense to look for an ellipsoid $\mathcal{E}$ that would be centered at $c= - t w$, with $t \in [0,1]$ (presumably $t$ will be small), and such that one principal direction is $w$ (with inverse squared semi-axis $a>0$), and the other principal directions are all orthogonal to $w$ (with the same inverse squared semi-axes $b>0$). In other words we are looking for $\mathcal{E}= \{x: (x - c)^{\top} H^{-1} (x-c) \leq 1 \}$ with
$$c = - t w, \; \text{and} \; H^{-1} = a w w^{\top} + b (\mI_n - w w^{\top} ) .$$
Now we have to express our constraints on the fact that $\mathcal{E}$ should contain the half Euclidean ball $\{x \in \cB : x^{\top} w \leq 0\}$. Since we are also looking for $\mathcal{E}$ to be as small as possible, it makes sense to ask for $\mathcal{E}$ to "touch" the Euclidean ball, both at $x = - w$, and at the equator $\partial \cB \cap w^{\perp}$. The former condition can be written as:
$$(- w - c)^{\top} H^{-1} (- w - c) = 1 \Leftrightarrow (t-1)^2 a = 1 ,$$
while the latter is expressed as:
$$\forall y \in \partial \cB \cap w^{\perp}, (y - c)^{\top} H^{-1} (y - c) = 1 \Leftrightarrow b + t^2 a = 1 .$$
As one can see from the above two equations, we are still free to choose any value for $t \in [0,1/2)$ (the fact that we need $t<1/2$ comes from $b=1 - \left(\frac{t}{t-1}\right)^2>0$). Quite naturally we take the value that minimizes the volume of the resulting ellipsoid. Note that
$$\frac{\mathrm{vol}(\mathcal{E})}{\mathrm{vol}(\cB)} = \frac{1}{\sqrt{a}} \left(\frac{1}{\sqrt{b}}\right)^{n-1} 
= \frac{1}{\sqrt{\frac{1}{(1-t)^2}\left (1 - \left(\frac{t}{1-t}\right)^2\right)^{n-1}}} \\= \frac{1}{\sqrt{f\left(\frac{1}{1-t}\right)}} ,$$
where $f(h) = h^2 (2 h - h^2)^{n-1}$. Elementary computations show that the maximum of $f$ (on $[1,2]$) is attained at $h = 1+ \frac{1}{n}$ (which corresponds to $t=\frac{1}{n+1}$), and the value is 
$$\left(1+\frac{1}{n}\right)^2 \left(1 - \frac{1}{n^2} \right)^{n-1} \geq \exp \left(\frac{1}{n} \right),$$
where the lower bound follows again from elementary computations. Thus we showed that, for $\cE_0 = \cB$, \eqref{eq:ellipsoidlemma1} and \eqref{eq:ellipsoidlemma2} are satisfied with the ellipsoid given by the set of points $x$ satisfying:
\begin{equation} \label{eq:ellipsoidlemma5}
\left(x + \frac{w/\|w\|_2}{n+1}\right)^{\top} \left(\frac{n^2-1}{n^2} \mI_n + \frac{2(n+1)}{n^2} \frac{w w^{\top}}{\|w\|_2^2} \right) \left(x + \frac{w/\|w\|_2}{n+1} \right) \leq 1 .
\end{equation}

We consider now an arbitrary ellipsoid $\cE_0 = \{x \in \R^n : (x - c_0)^{\top} H_0^{-1} (x-c_0) \leq 1 \}$. Let $\Phi(x) = c_0 + H_0^{1/2} x$, then clearly $\cE_0 = \Phi(\cB)$ and $\{x : w^{\top}(x - c_0) \leq 0\} = \Phi(\{x : (H_0^{1/2} w)^{\top} x \leq 0\})$. Thus in this case the image by $\Phi$ of the ellipsoid given in \eqref{eq:ellipsoidlemma5} with $w$ replaced by $H_0^{1/2} w$ will satisfy \eqref{eq:ellipsoidlemma1} and \eqref{eq:ellipsoidlemma2}. It is easy to see that this corresponds to an ellipsoid defined by
\begin{align}
& c = c_0 - \frac{1}{n+1} \frac{H_0 w}{\sqrt{w^{\top} H_0 w}} , \notag \\
& H^{-1} = \left(1 - \frac{1}{n^2}\right) H_0^{-1} + \frac{2(n+1)}{n^2} \frac{w w^{\top}}{w^{\top} H_0 w} . \label{eq:ellipsoidlemma6}
\end{align}
Applying Sherman-Morrison formula to \eqref{eq:ellipsoidlemma6} one can recover \eqref{eq:ellipsoidlemma4} which concludes the proof.
\end{proof}

We describe now the ellipsoid method, which only assumes a separation oracle for the constraint set $\cX$ (in particular it can be used to solve the feasibility problem mentioned at the beginning of the chapter). 
Let $\cE_0$ be the Euclidean ball of radius $R$ that contains $\cX$, and let $c_0$ be its center. Denote also $H_0=R^2 \mI_n$. For $t \geq 0$ do the following:
\begin{enumerate}
\item If $c_t \not\in \cX$ then call the separation oracle to obtain a separating hyperplane $w_t \in \R^n$ such that $\cX \subset \{x : (x- c_t)^{\top} w_t \leq 0\}$, otherwise call the first order oracle at $c_t$ to obtain $w_t \in \partial f (c_t)$. 
\item Let $\cE_{t+1} = \{x : (x - c_{t+1})^{\top} H_{t+1}^{-1} (x-c_{t+1}) \leq 1 \}$ be the ellipsoid given in Lemma \ref{lem:geomellipsoid} that contains $\{x \in \mathcal{E}_t : (x- c_t)^{\top} w_t \leq 0\}$, that is
\begin{align*}
& c_{t+1} = c_{t} - \frac{1}{n+1} \frac{H_t w}{\sqrt{w^{\top} H_t w}} ,\\
& H_{t+1} = \frac{n^2}{n^2-1} \left(H_t - \frac{2}{n+1} \frac{H_t w w^{\top} H_t}{w^{\top} H_t w} \right) .
\end{align*}
\end{enumerate}
If stopped after $t$ iterations and if $\{c_1, \hdots, c_t\} \cap \cX \neq \emptyset$, then we use the zeroth order oracle to output
$$x_t\in \argmin_{c \in \{c_1, \hdots, c_t\} \cap \cX} f(c_r) .$$
The following rate of convergence can be proved with the exact same argument than for Theorem \ref{th:centerofgravity} (observe that at step $t$ one can remove a point in $\cX$ from the current ellipsoid only if $c_t \in \cX$).
\begin{theorem}
For $t \geq 2n^2 \log(R/r)$ the ellipsoid method satisfies $\{c_1, \hdots, c_t\} \cap \cX \neq \emptyset$ and
$$f(x_t) - \min_{x \in \mathcal{X}} f(x) \leq \frac{2 B R}{r} \exp\left( - \frac{t}{2 n^2}\right) .$$
\end{theorem}
We observe that the oracle complexity of the ellipsoid method is much worse than the one of the center gravity method, indeed the former needs $O(n^2 \log(1/\epsilon))$ calls to the oracles while the latter requires only $O(n \log(1/\epsilon))$ calls. However from a computational point of view the situation is much better: in many cases one can derive an efficient separation oracle, while the center of gravity method is basically always intractable. This is for instance the case in the context of LPs and SDPs: with the notation of Section \ref{sec:structured} the computational complexity of the separation oracle for LPs is $O(m n)$ while for SDPs it is $O(\max(m,n) n^2)$ (we use the fact that the spectral decomposition of a matrix can be done in $O(n^3)$ operations). This gives an overall complexity of $O(\max(m,n) n^3 \log(1/\epsilon))$ for LPs and $O(\max(m,n^2) n^6 \log(1/\epsilon))$ for SDPs. We note however that the ellipsoid method is almost never used in practice, essentially because the method is too rigid to exploit the potential easiness of real problems (e.g., the volume decrease given by \eqref{eq:ellipsoidlemma2} is essentially always tight).


\section{Vaidya's cutting plane method}
We focus here on the feasibility problem (it should be clear from the previous sections how to adapt the argument for optimization). We have seen that for the feasibility problem the center of gravity has a $O(n)$ oracle complexity and unclear computational complexity (see Section \ref{sec:rwmethod} for more on this), while the ellipsoid method has oracle complexity $O(n^2)$ and computational complexity $O(n^4)$. We describe here the beautiful algorithm of \cite{Vai89, Vai96} which has oracle complexity $O(n \log(n))$ and computational complexity $O(n^4)$, thus getting the best of both the center of gravity and the ellipsoid method. In fact the computational complexity can even be improved further, and the recent breakthrough \cite{LSW15} shows that it can essentially (up to logarithmic factors) be brought down to $O(n^3)$.

This section, while giving a fundamental algorithm, should probably be skipped on a first reading. In particular we use several concepts from the theory of interior point methods which are described in Section \ref{sec:IPM}.

\subsection{The volumetric barrier}
Let $A \in \mathbb{R}^{m \times n}$ where the $i^{th}$ row is $a_i \in \mathbb{R}^n$, and let $b \in \mathbb{R}^m$. We consider the logarithmic barrier $F$ for the polytope $\{x \in \mathbb{R}^n : A x > b\}$ defined by
$$F(x) = - \sum_{i=1}^m \log(a_i^{\top} x - b_i) .$$
We also consider the volumetric barrier $v$ defined by
$$v(x) = \frac{1}{2} \mathrm{logdet}(\nabla^2 F(x) ) .$$
The intuition is clear: $v(x)$ is equal to the logarithm of the inverse volume of the Dikin ellipsoid (for the logarithmic barrier) at $x$. It will be useful to spell out the hessian of the logarithmic barrier:
$$\nabla^2 F(x) = \sum_{i=1}^m \frac{a_i a_i^{\top}}{(a_i^{\top} x - b_i)^2} .$$
Introducing the leverage score
$$\sigma_i(x) = \frac{(\nabla^2 F(x) )^{-1}[a_i, a_i]}{(a_i^{\top} x - b_i)^2} ,$$
one can easily verify that
\begin{equation} \label{eq:gradvol}
\nabla v(x) = - \sum_{i=1}^m \sigma_i(x) \frac{a_i}{a_i^{\top} x - b_i} ,
\end{equation}
and 
\begin{equation} \label{eq:hessianvol}
\nabla^2 v(x) \succeq \sum_{i=1}^m \sigma_i(x) \frac{a_i a_i^{\top}}{(a_i^{\top} x - b_i)^2} =: Q(x) .
\end{equation}

\subsection{Vaidya's algorithm}
We fix $\epsilon \leq 0.006$ a small constant to be specified later. Vaidya's algorithm produces a sequence of pairs $(A^{(t)}, b^{(t)}) \in \mathbb{R}^{m_t \times n} \times \mathbb{R}^{m_t}$ such that the corresponding polytope contains the convex set of interest. The initial polytope defined by $(A^{(0)},b^{(0)})$ is a simplex (in particular $m_0=n+1$). For $t\geq0$ we let $x_t$ be the minimizer of the volumetric barrier $v_t$ of the polytope given by $(A^{(t)}, b^{(t)})$, and $(\sigma_i^{(t)})_{i \in [m_t]}$ the leverage scores (associated to $v_t$) at the point $x_t$. We also denote $F_t$ for the logarithmic barrier given by $(A^{(t)}, b^{(t)})$. The next polytope $(A^{(t+1)}, b^{(t+1)})$ is defined by either adding or removing a constraint to the current polytope:
\begin{enumerate}
\item If for some $i \in [m_t]$ one has $\sigma_i^{(t)} = \min_{j \in [m_t]} \sigma_j^{(t)} < \epsilon$, then $(A^{(t+1)}, b^{(t+1)})$ is defined by removing the $i^{th}$ row in $(A^{(t)}, b^{(t)})$ (in particular $m_{t+1} = m_t - 1$).
\item Otherwise let $c^{(t)}$ be the vector given by the separation oracle queried at $x_t$, and $\beta^{(t)} \in \mathbb{R}$ be chosen so that 
$$\frac{(\nabla^2 F_t(x_t) )^{-1}[c^{(t)}, c^{(t)}]}{(x_t^{\top} c^{(t)} - \beta^{(t)})^2} = \frac{1}{5} \sqrt{\epsilon} .$$
Then we define $(A^{(t+1)}, b^{(t+1)})$ by adding to $(A^{(t)}, b^{(t)})$ the row given by $(c^{(t)}, \beta^{(t)})$ (in particular $m_{t+1} = m_t + 1$).
\end{enumerate}
It can be shown that the volumetric barrier is a self-concordant barrier, and thus it can be efficiently minimized with Newton's method. In fact it is enough to do {\em one step} of Newton's method on $v_t$ initialized at $x_{t-1}$, see \cite{Vai89, Vai96} for more details on this.

\subsection{Analysis of Vaidya's method} \label{sec:analysis}
The construction of Vaidya's method is based on a precise understanding of how the volumetric barrier changes when one adds or removes a constraint to the polytope. This understanding is derived in Section \ref{sec:constraintsvolumetric}. In particular we obtain the following two key inequalities: If case 1 happens at iteration $t$ then
\begin{equation} \label{eq:analysis1}
v_{t+1}(x_{t+1}) - v_t(x_t) \geq - \epsilon ,
\end{equation}
while if case 2 happens then 
\begin{equation} \label{eq:analysis2}
v_{t+1}(x_{t+1}) - v_t(x_t) \geq \frac{1}{20} \sqrt{\epsilon} .
\end{equation}
We show now how these inequalities imply that Vaidya's method stops after $O(n \log(n R/r))$ steps. First we claim that after $2t$ iterations, case 2 must have happened at least $t-1$ times. Indeed suppose that at iteration $2t-1$, case 2 has happened $t-2$ times; then $\nabla^2 F(x)$ is singular and the leverage scores are infinite, so case 2 must happen at iteration $2t$. Combining this claim with the two inequalities above we obtain:
$$v_{2t}(x_{2t}) \geq v_0(x_0) + \frac{t-1}{20} \sqrt{\epsilon} - (t+1) \epsilon \geq \frac{t}{50} \epsilon - 1 +v_0(x_0) . $$
The key point now is to recall that by definition one has $v(x) = - \log \mathrm{vol}(\cE(x,1))$ where $\cE(x,r) = \{y : \nabla F^2(x)[y-x,y-x] \leq r^2\}$ is the Dikin ellipsoid centered at $x$ and of radius $r$. Moreover the logarithmic barrier $F$ of a polytope with $m$ constraints is $m$-self-concordant, which implies that the polytope is included in the Dikin ellipsoid $\cE(z, 2m)$ where $z$ is the minimizer of $F$ (see [Theorem 4.2.6., \cite{Nes04}]). The volume of $\cE(z, 2m)$ is equal to $(2m)^n \exp(-v(z))$, which is thus always an upper bound on the volume of the polytope. Combining this with the above display we just proved that at iteration $2k$ the volume of the current polytope is at most
$$\exp \left(n \log(2m_{2t}) + 1 - v_0(x_0) - \frac{t}{50} \epsilon \right) .$$
Since $\cE(x,1)$ is always included in the polytope we have that $- v_0(x_0)$ is at most the logarithm of the volume of the initial polytope which is $O(n \log(R))$. This clearly concludes the proof as the procedure will necessarily stop when the volume is below $\exp(n \log(r))$ (we also used the trivial bound $m_t \leq n+1+t$).

\subsection{Constraints and the volumetric barrier} \label{sec:constraintsvolumetric}
We want to understand the effect on the volumetric barrier of addition/deletion of constraints to the polytope. Let $c \in \mathbb{R}^n$, $\beta \in \mathbb{R}$, and consider the logarithmic barrier $\tilde{F}$ and the volumetric barrier $\tilde{v}$ corresponding to the matrix $\tilde{A}\in \mathbb{R}^{(m+1) \times n}$ and the vector $\tilde{b} \in \mathbb{R}^{m+1}$ which are respectively the concatenation of $A$ and $c$, and the concatenation of $b$ and $\beta$. Let $x^*$ and $\tilde{x}^*$ be the minimizer of respectively $v$ and $\tilde{v}$. We recall the definition of leverage scores, for $i \in [m+1]$, where $a_{m+1}=c$ and $b_{m+1}=\beta$,
$$\sigma_i(x) = \frac{(\nabla^2 F(x) )^{-1}[a_i, a_i]}{(a_i^{\top} x - b_i)^2}, \ \text{and} \ \tilde{\sigma}_i(x) = \frac{(\nabla^2 \tilde{F}(x) )^{-1}[a_i, a_i]}{(a_i^{\top} x - b_i)^2}.$$
The leverage scores $\sigma_i$ and $\tilde{\sigma}_i$ are closely related:
\begin{lemma} \label{lem:V1}
One has for any $i \in [m+1]$,
$$\frac{\tilde{\sigma}_{m+1}(x)}{1 - \tilde{\sigma}_{m+1}(x)} \geq \sigma_i(x) \geq \tilde{\sigma}_i(x) \geq (1-\sigma_{m+1}(x)) \sigma_i(x) .$$
\end{lemma}

\begin{proof}
First we observe that by Sherman-Morrison's formula $(A+uv^{\top})^{-1} = A^{-1} - \frac{A^{-1} u v^{\top} A^{-1}}{1+A^{-1}[u,v]}$ one has
\begin{equation} \label{eq:SM}
(\nabla^2 \tilde{F}(x))^{-1} = (\nabla^2 F(x))^{-1} - \frac{(\nabla^2 F(x))^{-1} c c^{\top} (\nabla^2 F(x))^{-1}}{(c^{\top} x - \beta)^2 + (\nabla^2 F(x))^{-1}[c,c]} ,
\end{equation}
This immediately proves $\tilde{\sigma}_i(x) \leq \sigma_i(x)$. It also implies the inequality $\tilde{\sigma}_i(x) \geq (1-\sigma_{m+1}(x)) \sigma_i(x)$ thanks the following fact: $A - \frac{A u u^{\top} A}{1+A[u,u]} \succeq (1-A[u,u]) A$. For the last inequality we use that $A + \frac{A u u^{\top} A}{1+A[u,u]} \preceq \frac{1}{1-A[u,u]} A$ together with
$$(\nabla^2 {F}(x))^{-1} = (\nabla^2 \tilde{F}(x))^{-1} + \frac{(\nabla^2 \tilde{F}(x))^{-1} c c^{\top} (\nabla^2 \tilde{F}(x))^{-1}}{(c^{\top} x - \beta)^2 - (\nabla^2 \tilde{F}(x))^{-1}[c,c]} .$$
\end{proof}

We now assume the following key result, which was first proven by Vaidya. To put the statement in context recall that for a self-concordant barrier $f$ the suboptimality gap $f(x) - \min f$ is intimately related to the Newton decrement $\|\nabla f(x) \|_{(\nabla^2 f(x))^{-1}}$. Vaidya's inequality gives a similar claim for the volumetric barrier. We use the version given in [Theorem 2.6, \cite{Ans98}] which has slightly better numerical constants than the original bound. Recall also the definition of $Q$ from \eqref{eq:hessianvol}.

\begin{theorem} \label{th:V0}
Let $\lambda(x) = \|\nabla v(x) \|_{Q(x)^{-1}}$ be an approximate Newton decrement, $\epsilon = \min_{i \in [m]} \sigma_i(x)$, and assume that $\lambda(x)^2 \leq \frac{2 \sqrt{\epsilon} - \epsilon}{36}$. Then
$$v(x) - v(x^*) \leq 2 \lambda(x)^2 . $$
\end{theorem}
We also denote $\tilde{\lambda}$ for the approximate Newton decrement of $\tilde{v}$. The goal for the rest of the section is to prove the following theorem which gives the precise understanding of the volumetric barrier we were looking for.

\begin{theorem} \label{th:V1}
Let $\epsilon := \min_{i \in [m]} \sigma_i(x^*)$, $\delta := \sigma_{m+1}(x^*) / \sqrt{\epsilon}$ and assume that $\frac{\left(\delta \sqrt{\epsilon} + \sqrt{\delta^{3} \sqrt{\epsilon}}\right)^2}{1- \delta \sqrt{\epsilon}} < \frac{2 \sqrt{\epsilon} - \epsilon}{36}$. Then one has
\begin{equation} \label{eq:thV11}
\tilde{v}(\tilde{x}^*) - v(x^*) \geq \frac{1}{2} \log(1+\delta \sqrt{\epsilon}) - 2  \frac{\left(\delta \sqrt{\epsilon} + \sqrt{\delta^{3} \sqrt{\epsilon}}\right)^2}{1- \delta \sqrt{\epsilon}}  .
\end{equation}
On the other hand assuming that $\tilde{\sigma}_{m+1}(\tilde{x}^*) = \min_{i \in [m+1]} \tilde{\sigma}_{i}(\tilde{x}^*) =: \epsilon$ and that $\epsilon \leq 1/4$, one has 
\begin{equation} \label{eq:thV12}
\tilde{v}(\tilde{x}^*) - v(x^*) \leq - \frac{1}{2} \log(1 - \epsilon) + \frac{8 \epsilon^2}{(1-\epsilon)^2}.
\end{equation}
\end{theorem}

Before going into the proof let us see briefly how Theorem \ref{th:V1} give the two inequalities stated at the beginning of Section \ref{sec:analysis}. To prove \eqref{eq:analysis2} we use \eqref{eq:thV11} with $\delta=1/5$ and $\epsilon \leq 0.006$, and we observe that in this case the right hand side of \eqref{eq:thV11} is lower bounded by $\frac{1}{20} \sqrt{\epsilon}$. On the other hand to prove \eqref{eq:analysis1} we use \eqref{eq:thV12}, and we observe that for $\epsilon \leq 0.006$ the right hand side of \eqref{eq:thV12} is upper bounded by $\epsilon$.

\begin{proof}
We start with the proof of \eqref{eq:thV11}. First observe that by factoring $(\nabla^2 F(x))^{1/2}$ on the left and on the right of $\nabla^2 \tilde{F}(x)$ one obtains
\begin{align*}
& \mathrm{det}(\nabla^2 \tilde{F}(x)) \\
& = \mathrm{det}\left(\nabla^2 {F}(x) + \frac{cc^{\top}}{(c^{\top} x- \beta)^2} \right) \\
& = \mathrm{det}(\nabla^2 {F}(x)) \mathrm{det}\left(\mathrm{I}_n + \frac{(\nabla^2 {F}(x))^{-1/2} c c^{\top} (\nabla^2 {F}(x))^{-1/2}}{(c^{\top} x- \beta)^2}\right) \\
& = \mathrm{det}(\nabla^2 {F}(x)) (1+\sigma_{m+1}(x)) ,
\end{align*}
and thus
$$\tilde{v}(x) = v(x) + \frac{1}{2} \log(1+ \sigma_{m+1}(x)) .$$
In particular we have
$$\tilde{v}(\tilde{x}^*) - v(x^*) = \frac{1}{2} \log(1+ \sigma_{m+1}(x^*)) - (\tilde{v}(x^*) - \tilde{v}(\tilde{x}^*)) .$$
To bound the suboptimality gap of $x^*$ in $\tilde{v}$ we will invoke Theorem \ref{th:V0} and thus we have to upper bound the approximate Newton decrement $\tilde{\lambda}$.
Using [\eqref{eq:V21}, Lemma \ref{lem:V2}] below one has 
$$\tilde{\lambda} (x^*)^2 \leq \frac{\left(\sigma_{m+1}(x^*) + \sqrt{\frac{\sigma_{m+1}^3(x^*)}{\min_{i \in [m]} \sigma_i(x^*)}}\right)^2}{1-\sigma_{m+1}(x^*)} = \frac{\left(\delta \sqrt{\epsilon} + \sqrt{\delta^{3} \sqrt{\epsilon}}\right)^2}{1- \delta \sqrt{\epsilon}}  .$$
This concludes the proof of \eqref{eq:thV11}.
\newline

We now turn to the proof of \eqref{eq:thV12}. Following the same steps as above we immediately obtain
\begin{eqnarray*}
\tilde{v}(\tilde{x}^*) - v(x^*) & = & \tilde{v}(\tilde{x}^*) - v(\tilde{x}^*)+v(\tilde{x}^*)- v(x^*)  \\
& = & - \frac{1}{2} \log(1 - \tilde{\sigma}_{m+1}(\tilde{x}^*)) + v(\tilde{x}^*)- v(x^*).
\end{eqnarray*}
To invoke Theorem \ref{th:V0} it remains to upper bound $\lambda(\tilde{x}^*)$. Using [\eqref{eq:V22}, Lemma \ref{lem:V2}] below one has
$$ \lambda(\tilde{x}^*) \leq \frac{2 \ \tilde{\sigma}_{m+1}(\tilde{x}^*)}{1 - \tilde{\sigma}_{m+1}(\tilde{x}^*)} .$$
We can apply Theorem \ref{th:V0} since the assumption $\epsilon \leq 1/4$ implies that $\left(\frac{2 \epsilon}{1-\epsilon}\right)^2 \leq \frac{2 \sqrt{\epsilon} - \epsilon}{36}$. This concludes the proof of \eqref{eq:thV12}.
\end{proof}

\begin{lemma} \label{lem:V2}
One has
\begin{equation} \label{eq:V21}
\sqrt{1- \sigma_{m+1}(x)} \ \tilde{\lambda} (x) \leq \|\nabla {v}(x)\|_{Q(x)^{-1}} + \sigma_{m+1}(x) + \sqrt{\frac{\sigma_{m+1}^3(x)}{\min_{i \in [m]} \sigma_i(x)}} .
\end{equation}
Furthermore if $\tilde{\sigma}_{m+1}(x) = \min_{i \in [m+1]} \tilde{\sigma}_{i}(x)$ then one also has
\begin{equation} \label{eq:V22}
\lambda(x) \leq  \|\nabla \tilde{v}(x)\|_{Q(x)^{-1}} + \frac{2 \ \tilde{\sigma}_{m+1}(x)}{1 - \tilde{\sigma}_{m+1}(x)} .
\end{equation}
\end{lemma}

\begin{proof}
We start with the proof of \eqref{eq:V21}. First observe that by Lemma \ref{lem:V1} one has $\tilde{Q}(x) \succeq (1-\sigma_{m+1}(x)) Q(x)$ and thus by definition of the Newton decrement
$$\tilde{\lambda} (x) = \|\nabla \tilde{v}(x)\|_{\tilde{Q}(x)^{-1}} \leq \frac{\|\nabla \tilde{v}(x)\|_{Q(x)^{-1}}}{\sqrt{1-\sigma_{m+1}(x)}} .$$
Next observe that (recall \eqref{eq:gradvol})
$$ \nabla \tilde{v}(x) = \nabla v(x) + \sum_{i=1}^m ({\sigma}_i(x) - \tilde{\sigma}_i(x)) \frac{a_i}{a_i^{\top} x - b_i} - \tilde{\sigma}_{m+1}(x) \frac{c}{c^{\top} x - \beta} .$$
We now use that $Q(x) \succeq (\min_{i \in [m]} \sigma_i(x)) \nabla^2 F(x)$ to obtain 
$$\left \| \tilde{\sigma}_{m+1}(x) \frac{c}{c^{\top} x - \beta} \right\|_{Q(x)^{-1}}^2 \leq \frac{\tilde{\sigma}_{m+1}^2(x) \sigma_{m+1}(x)}{\min_{i \in [m]} \sigma_i(x)} .$$
By Lemma \ref{lem:V1} one has $\tilde{\sigma}_{m+1}(x) \leq {\sigma}_{m+1}(x)$ and thus we see that it only remains to prove 
$$\left\|\sum_{i=1}^m ({\sigma}_i(x) - \tilde{\sigma}_i(x)) \frac{a_i}{a_i^{\top}x - b_i} \right\|_{Q(x)^{-1}}^2 \leq \sigma_{m+1}^2(x) .$$
The above inequality follows from a beautiful calculation of Vaidya (see [Lemma 12, \cite{Vai96}]), starting from the identity
$$\sigma_i(x) - \tilde{\sigma}_i(x) = \frac{((\nabla^2 F(x))^{-1}[a_i,c])^2}{((c^{\top} x - \beta)^2 + (\nabla^2 F(x))^{-1}[c,c])(a_i^{\top} x - b_i)^2} ,$$
which itself follows from \eqref{eq:SM}.
\newline

We now turn to the proof of \eqref{eq:V22}. Following the same steps as above we immediately obtain
$$\lambda(x) = \|\nabla v(x)\|_{Q(x)^{-1}} \leq \|\nabla \tilde{v}(x)\|_{Q(x)^{-1}} + \sigma_{m+1}(x) + \sqrt{\frac{\tilde{\sigma}_{m+1}^2(x) \sigma_{m+1}(x)}{\min_{i \in [m]} \sigma_i(x)}} .$$
Using Lemma \ref{lem:V1} together with the assumption $\tilde{\sigma}_{m+1}(x) = \min_{i \in [m+1]} \tilde{\sigma}_{i}(x)$ yields \eqref{eq:V22}, thus concluding the proof.
\end{proof}

\section{Conjugate gradient} \label{sec:CG}
We conclude this chapter with the special case of unconstrained optimization of a convex quadratic function $f(x) = \frac12 x^{\top} A x - b^{\top} x$, where $A \in \R^{n \times n}$ is a positive definite matrix and $b \in \R^n$. This problem, of paramount importance in practice (it is equivalent to solving the linear system $Ax = b$), admits a simple first-order black-box procedure which attains the {\em exact} optimum $x^*$ in at most $n$ steps. This method, called the {\em conjugate gradient}, is described and analyzed below. What is written below is taken from [Chapter 5, \cite{NW06}].

Let $\langle \cdot , \cdot\rangle_A$ be the inner product on $\R^n$ defined by the positive definite matrix $A$, that is $\langle x, y\rangle_A = x^{\top} A y$ (we also denote by $\|\cdot\|_A$ the corresponding norm). For sake of clarity we denote here $\langle \cdot , \cdot\rangle$ for the standard inner product in $\R^n$. Given an orthogonal set $\{p_0, \hdots, p_{n-1}\}$ for $\langle \cdot , \cdot \rangle_A$ we will minimize $f$ by sequentially minimizing it along the directions given by this orthogonal set. That is, given $x_0 \in \R^n$, for $t \geq 0$ let
\begin{equation} \label{eq:CG1}
x_{t+1} := \argmin_{x \in \{x_t + \lambda p_t, \ \lambda \in \R\}} f(x) .
\end{equation}
Equivalently one can write
\begin{equation} \label{eq:CG2}
x_{t+1} = x_t - \langle \nabla f(x_t) , p_t \rangle \frac{p_t}{\|p_t\|_A^2} .
\end{equation}
The latter identity follows by differentiating $\lambda \mapsto f(x + \lambda p_t)$, and using that $\nabla f(x) = A x - b$. We also make an observation that will be useful later, namely that $x_{t+1}$ is the minimizer of $f$ on $x_0 + \mathrm{span}\{p_0, \hdots, p_t\}$, or equivalently
\begin{equation} \label{eq:CG3prime}
\langle \nabla f(x_{t+1}), p_i \rangle = 0, \forall \ 0 \leq i \leq t.
\end{equation}
Equation \eqref{eq:CG3prime} is true by construction for $i=t$, and for $i \leq t-1$ it follows by induction, assuming \eqref{eq:CG3prime} at $t=1$ and using the following formula:
\begin{equation} \label{eq:CG3}
\nabla f(x_{t+1}) = \nabla f(x_{t}) - \langle \nabla f(x_{t}) , p_{t} \rangle \frac{A p_{t}}{\|p_t\|_A^2} .
\end{equation}

We now claim that $x_n = x^* = \argmin_{x \in \R^n} f(x)$. It suffices to show that $\langle x_n -x_0 , p_t \rangle_A = \langle x^*-x_0 , p_t \rangle_A$ for any $t\in \{0,\hdots,n-1\}$. Note that $x_n - x_0 = - \sum_{t=0}^{n-1} \langle \nabla f(x_t), p_t \rangle \frac{p_t}{\|p_t\|_A^2}$, and thus using that $x^* = A^{-1} b$,
\begin{eqnarray*}
\langle x_n -x_0 , p_t \rangle_A = - \langle \nabla f(x_t) , p_t \rangle = \langle b - A x_t , p_t \rangle & = & \langle x^* - x_t, p_t \rangle_A \\
& = &  \langle x^* - x_0, p_t \rangle_A ,
\end{eqnarray*}
which concludes the proof of $x_n = x^*$.
\newline

In order to have a proper black-box method it remains to describe how to build iteratively the orthogonal set $\{p_0, \hdots, p_{n-1}\}$ based only on gradient evaluations of $f$. A natural guess to obtain a set of orthogonal directions (w.r.t. $\langle \cdot , \cdot \rangle_A$) is to take $p_0 = \nabla f(x_0)$ and for $t \geq 1$,
\begin{equation} \label{eq:CG4}
p_t = \nabla f(x_t) - \langle \nabla f(x_t), p_{t-1} \rangle_A \ \frac{p_{t-1}}{\|p_{t-1}\|^2_A} .
\end{equation}
Let us first verify by induction on $t \in [n-1]$ that for any $i \in \{0,\hdots,t-2\}$, $\langle p_t, p_{i}\rangle_A = 0$ (observe that for $i=t-1$ this is true by construction of $p_t$). Using the induction hypothesis one can see that it is enough to show $\langle \nabla f(x_t), p_i \rangle_A = 0$ for any $i \in \{0, \hdots, t-2\}$, which we prove now. First observe that by induction one easily obtains $A p_i \in \mathrm{span}\{p_0, \hdots, p_{i+1}\}$ from \eqref{eq:CG3} and \eqref{eq:CG4}. Using this fact together with $\langle \nabla f(x_t), p_i \rangle_A = \langle \nabla f(x_t), A p_i \rangle$ and \eqref{eq:CG3prime} thus concludes the proof of orthogonality of the set $\{p_0, \hdots, p_{n-1}\}$.

We still have to show that \eqref{eq:CG4} can be written by making only reference to the gradients of $f$ at previous points. Recall that $x_{t+1}$ is the minimizer of $f$ on $x_0 + \mathrm{span}\{p_0, \hdots, p_t\}$, and thus given the form of $p_t$ we also have that $x_{t+1}$ is the minimizer of $f$ on $x_0 + \mathrm{span}\{\nabla f(x_0), \hdots, \nabla f(x_t)\}$ (in some sense the conjugate gradient is the {\em optimal} first order method for convex quadratic functions). In particular one has $\langle \nabla f(x_{t+1}) , \nabla f(x_t) \rangle = 0$. This fact, together with the orthogonality of the set $\{p_t\}$ and \eqref{eq:CG3}, imply that
$$\frac{\langle \nabla f(x_{t+1}) , p_{t} \rangle_A}{\|p_t\|_A^2} = \langle \nabla f(x_{t+1}) , \frac{A p_{t}}{\|p_t\|_A^2}  \rangle = - \frac{\langle \nabla f(x_{t+1})  , \nabla f(x_{t+1}) \rangle}{\langle \nabla f(x_{t})  , p_t \rangle} .$$
Furthermore using the definition \eqref{eq:CG4} and $\langle \nabla f(x_t) , p_{t-1} \rangle = 0$ one also has
$$\langle \nabla f(x_t), p_t \rangle = \langle \nabla f(x_t) , \nabla f(x_t) \rangle .$$
Thus we arrive at the following rewriting of the (linear) conjugate gradient algorithm, where we recall that $x_0$ is some fixed starting point and $p_0 = \nabla f(x_0)$,
\begin{eqnarray}
x_{t+1} & = & \argmin_{x \in \left\{x_t + \lambda p_t, \ \lambda \in \R \right\}} f(x) , \label{eq:CG5} \\
p_{t+1} & = & \nabla f(x_{t+1}) + \frac{\langle \nabla f(x_{t+1})  , \nabla f(x_{t+1}) \rangle}{\langle \nabla f(x_{t})  , \nabla f(x_t) \rangle} p_t . \label{eq:CG6}
\end{eqnarray}
Observe that the algorithm defined by \eqref{eq:CG5} and \eqref{eq:CG6} makes sense for an arbitary convex function, in which case it is called the {\em non-linear conjugate gradient}. There are many variants of the non-linear conjugate gradient, and the above form is known as the Fletcher-–Reeves method. Another popular version in practice is the Polak-Ribi{\`e}re method which is based on the fact that for the general non-quadratic case one does not necessarily have $\langle \nabla f(x_{t+1}) , \nabla f(x_t) \rangle = 0$, and thus one replaces \eqref{eq:CG6} by
$$p_{t+1} = \nabla f(x_{t+1}) + \frac{\langle \nabla f(x_{t+1})  - \nabla f(x_t), \nabla f(x_{t+1}) \rangle}{\langle \nabla f(x_{t})  , \nabla f(x_t) \rangle} p_t .$$
We refer to \cite{NW06} for more details about these algorithms, as well as for advices on how to deal with the line search in \eqref{eq:CG5}.

Finally we also note that the linear conjugate gradient method can often attain an approximate solution in much fewer than $n$ steps. More precisely, denoting $\kappa$ for the condition number of $A$ (that is the ratio of the largest eigenvalue to the smallest eigenvalue of $A$), one can show that linear conjugate gradient attains an $\epsilon$ optimal point in a number of iterations of order $\sqrt{\kappa} \log(1/\epsilon)$. The next chapter will demistify this convergence rate, and in particular we will see that (i) this is the optimal rate among first order methods, and (ii) there is a way to generalize this rate to non-quadratic convex functions (though the algorithm will have to be modified).

\chapter{Dimension-free convex optimization}
\label{dimfree}
We investigate here variants of the {\em gradient descent} scheme. This iterative algorithm, which can be traced back to \cite{Cau47}, is the simplest strategy to minimize a differentiable function $f$ on $\R^n$. Starting at some initial point $x_1 \in \R^n$ it iterates the following equation:
\begin{equation} \label{eq:Cau47}
x_{t+1} = x_t - \eta \nabla f(x_t) ,
\end{equation}
where $\eta > 0$ is a fixed step-size parameter. The rationale behind \eqref{eq:Cau47} is to make a small step in the direction that minimizes the local first order Taylor approximation of $f$ (also known as the steepest descent direction). 

As we shall see, methods of the type \eqref{eq:Cau47} can obtain an oracle complexity {\em independent of the dimension}\footnote{Of course the computational complexity remains at least linear in the dimension since one needs to manipulate gradients.}. This feature makes them particularly attractive for optimization in very high dimension.

Apart from Section \ref{sec:FW}, in this chapter $\|\cdot\|$ denotes the Euclidean norm. The set of constraints $\cX \subset \R^n$ is assumed to be compact and convex. 
We define the projection operator $\Pi_{\cX}$ on $\cX$ by
$$\Pi_{\cX}(x) = \argmin_{y \in \mathcal{X}} \|x - y\| .$$
The following lemma will prove to be useful in our study. It is an easy corollary of Proposition \ref{prop:firstorder}, see also Figure \ref{fig:pythagore}.

\begin{lemma} \label{lem:todonow}
Let $x \in \cX$ and $y \in \R^n$, then
$$(\Pi_{\cX}(y) - x)^{\top} (\Pi_{\cX}(y) - y) \leq 0 ,$$
which also implies $\|\Pi_{\cX}(y) - x\|^2 + \|y - \Pi_{\cX}(y)\|^2 \leq \|y - x\|^2$.
\end{lemma}

\begin{figure}
\begin{center}
\begin{tikzpicture}[scale=4]
\clip (0.4,-0.4) rectangle (-1.2,1.2);
\draw[rotate=30, very thick] (0,0) ellipse (0.5 and 0.7);
\node [tokens=1] (noeud1) at (-0.45,-0.1) [label=right:{$x$}] {};
\node [tokens=1] (noeud2) at (-0.6, 0.8) [label=above left:{$y$}] {};
\draw[<->, thick] (noeud1) -- (noeud2) node[midway, below left] {$\|y - x \|$};
\node [tokens=1] (noeud3) at (-60:-0.7) [label=below right:{$\Pi_{\cX}(y)$}] {};
\draw[<->, thick] (noeud2) -- (noeud3) node[midway, above right] {$\|y - \Pi_{\cX}(y) \|$};
\draw[<->, thick] (noeud1) -- (noeud3) node[midway, right] {$\|\Pi_{\cX}(y) - x\|$};
\node at (0.15,-0.2) {$\cX$};
\end{tikzpicture}
\end{center}
\caption{Illustration of Lemma \ref{lem:todonow}.}
\label{fig:pythagore}
\end{figure}
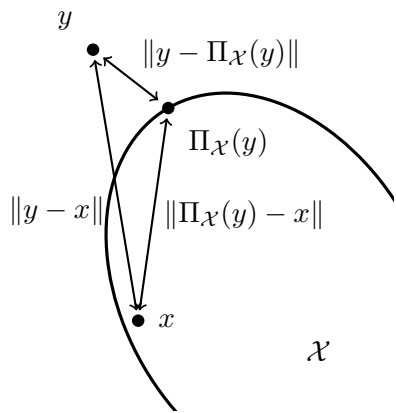

Unless specified otherwise all the proofs in this chapter are taken from \cite{Nes04} (with slight simplification in some cases).

\section{Projected subgradient descent for Lipschitz functions} \label{sec:psgd}
In this section we assume that $\cX$ is contained in an Euclidean ball centered at $x_1 \in \cX$ and of radius $R$. Furthermore we assume that $f$ is such that for any $x \in \cX$ and any $g \in \partial f(x)$ (we assume $\partial f(x) \neq \emptyset$), one has $\|g\| \leq L$. Note that by the subgradient inequality and Cauchy-Schwarz this implies that $f$ is $L$-Lipschitz on $\cX$, that is $|f(x) - f(y)| \leq L \|x-y\|$. 

In this context we make two modifications to the basic gradient descent \eqref{eq:Cau47}. First, obviously, we replace the gradient $\nabla f(x)$ (which may not exist) by a subgradient $g \in \partial f(x)$. Secondly, and more importantly, we make sure that the updated point lies in $\cX$ by projecting back (if necessary) onto it. This gives the {\em projected subgradient descent} algorithm\footnote{In the optimization literature the term ``descent" is reserved for methods such that $f(x_{t+1}) \leq f(x_t)$. In that sense the projected subgradient descent is not a descent method.} which iterates the following equations for $t \geq 1$:
\begin{align}
& y_{t+1} = x_t - \eta g_t , \ \text{where} \ g_t \in \partial f(x_t) , \label{eq:PGD1}\\
& x_{t+1} = \Pi_{\cX}(y_{t+1}) . \label{eq:PGD2}
\end{align}
This procedure is illustrated in Figure \ref{fig:pgd}. We prove now a rate of convergence for this method under the above assumptions.

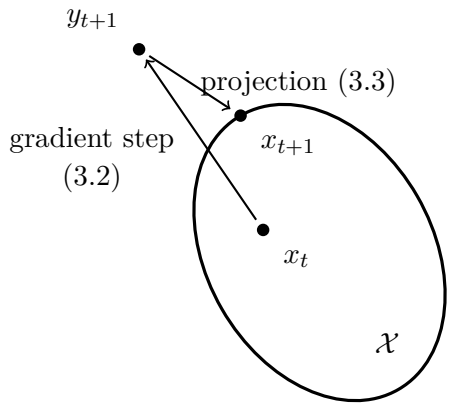
\begin{figure}
\begin{center}
\begin{tikzpicture}[scale=3]
\draw[rotate=30, very thick] (0,0) ellipse (0.5 and 0.7);
\node [tokens=1] (noeud1) at (-0.25,0.1) [label=below right:{$x_t$}] {};
\node [tokens=1] (noeud2) at (-0.8, 0.9) [label=above left:{$y_{t+1}$}] {};
\draw[->, thick] (noeud1) -- (noeud2) node[midway, left] {\begin{tabular}{c} \\ gradient step \\ \eqref{eq:PGD1} \end{tabular}};
\node [tokens=1] (noeud3) at (-60:-0.7) [label=below right:{$x_{t+1}$}] {};
\draw[->, thick] (noeud2) -- (noeud3) node[midway, right] {projection \eqref{eq:PGD2}};
\node at (0.3,-0.4) {$\cX$};
\end{tikzpicture}
\end{center}
\caption{Illustration of the projected subgradient descent method.}
\label{fig:pgd}
\end{figure}

\begin{theorem} \label{th:pgd}
The projected subgradient descent method with $\eta = \frac{R}{L \sqrt{t}}$ satisfies 
$$f\left(\frac{1}{t} \sum_{s=1}^t x_s\right) - f(x^*) \leq \frac{R L}{\sqrt{t}} .$$
\end{theorem}


\begin{proof}
Using the definition of subgradients, the definition of the method, and the elementary identity $2 a^{\top} b = \|a\|^2 + \|b\|^2 - \|a-b\|^2$, one obtains
\begin{eqnarray*}
f(x_s) - f(x^*) & \leq & g_s^{\top} (x_s - x^*) \\
& = & \frac{1}{\eta} (x_s - y_{s+1})^{\top} (x_s - x^*) \\
& = & \frac{1}{2 \eta} \left(\|x_s - x^*\|^2 + \|x_s - y_{s+1}\|^2 - \|y_{s+1} - x^*\|^2\right) \\
& = & \frac{1}{2 \eta} \left(\|x_s - x^*\|^2 - \|y_{s+1} - x^*\|^2\right) + \frac{\eta}{2} \|g_s\|^2.
\end{eqnarray*}
Now note that $\|g_s\| \leq L$, and furthermore by Lemma \ref{lem:todonow}
$$\|y_{s+1} - x^*\| \geq \|x_{s+1} - x^*\| .$$
Summing the resulting inequality over $s$, and using that $\|x_1 - x^*\| \leq R$ yield
$$\sum_{s=1}^t \left( f(x_s) - f(x^*) \right) \leq \frac{R^2}{2 \eta} + \frac{\eta L^2 t}{2} .$$
Plugging in the value of $\eta$ directly gives the statement (recall that by convexity $f((1/t) \sum_{s=1}^t x_s) \leq \frac1{t} \sum_{s=1}^t f(x_s)$).
\end{proof}

We will show in Section \ref{sec:chap3LB} that the rate given in Theorem \ref{th:pgd} is unimprovable from a black-box perspective. Thus to reach an $\epsilon$-optimal point one needs $\Theta(1/\epsilon^2)$ calls to the oracle. In some sense this is an astonishing result as this complexity is independent\footnote{Observe however that the quantities $R$ and $L$ may dependent on the dimension, see Chapter \ref{mirror} for more on this.} of the ambient dimension $n$. On the other hand this is also quite disappointing compared to the scaling in $\log(1/\epsilon)$ of the center of gravity and ellipsoid method of Chapter \ref{finitedim}. To put it differently with gradient descent one could hope to reach a reasonable accuracy in very high dimension, while with the ellipsoid method one can reach very high accuracy in reasonably small dimension. A major task in the following sections will be to explore more restrictive assumptions on the function to be optimized in order to have the best of both worlds, that is an oracle complexity independent of the dimension and with a scaling in $\log(1/\epsilon)$.

The computational bottleneck of the projected subgradient descent is often the projection step \eqref{eq:PGD2} which is a convex optimization problem by itself. In some cases this problem may admit an analytical solution (think of $\cX$ being an Euclidean ball), or an easy and fast combinatorial algorithm to solve it (this is the case for $\cX$ being an $\ell_1$-ball, see \cite{MP89}). We will see in Section \ref{sec:FW} a projection-free algorithm which operates under an extra assumption of smoothness on the function to be optimized.

Finally we observe that the step-size recommended by Theorem \ref{th:pgd} depends on the number of iterations to be performed. In practice this may be an undesirable feature. However using a time-varying step size of the form $\eta_s = \frac{R}{L \sqrt{s}}$ one can prove the same rate up to a $\log t$ factor. In any case these step sizes are very small, which is the reason for the slow convergence. In the next section we will see that by assuming {\em smoothness} in the function $f$ one can afford to be much more aggressive. Indeed in this case, as one approaches the optimum the size of the gradients themselves will go to $0$, resulting in a sort of ``auto-tuning" of the step sizes which does not happen for an arbitrary convex function.

\section{Gradient descent for smooth functions} \label{sec:gdsmooth}
We say that a continuously differentiable function $f$ is $\beta$-smooth if the gradient $\nabla f$ is $\beta$-Lipschitz, that is 
$$\|\nabla f(x) - \nabla f(y) \| \leq \beta \|x-y\| .$$
Note that if $f$ is twice differentiable then this is equivalent to the eigenvalues of the Hessians being smaller than $\beta$.
In this section we explore potential improvements in the rate of convergence under such a smoothness assumption.
In order to avoid technicalities we consider first the unconstrained situation, where $f$ is a convex and $\beta$-smooth function on $\R^n$. 
The next theorem shows that {\em gradient descent}, which iterates $x_{t+1} = x_t - \eta \nabla f(x_t)$, attains a much faster rate in this situation than in the non-smooth case of the previous section.

\begin{theorem} \label{th:gdsmooth}
Let $f$ be convex and $\beta$-smooth on $\R^n$. 
Then gradient descent with $\eta = \frac{1}{\beta}$ satisfies
$$f(x_t) - f(x^*) \leq \frac{2 \beta \|x_1 - x^*\|^2}{t-1} .$$
\end{theorem}

Before embarking on the proof we state a few properties of smooth convex functions.
\begin{lemma} \label{lem:sand}
Let $f$ be a $\beta$-smooth function on $\R^n$. Then for any $x, y \in \R^n$, one has
$$|f(x) - f(y) - \nabla f(y)^{\top} (x - y)| \leq \frac{\beta}{2} \|x - y\|^2 .$$
\end{lemma}

\begin{proof}
We represent $f(x) - f(y)$ as an integral, apply Cauchy-Schwarz and then $\beta$-smoothness:
\begin{align*}
& |f(x) - f(y) - \nabla f(y)^{\top} (x - y)| \\
& = \left|\int_0^1 \nabla f(y + t(x-y))^{\top} (x-y) dt -  \nabla f(y)^{\top} (x - y) \right| \\
& \leq \int_0^1 \|\nabla f(y + t(x-y)) -  \nabla f(y)\| \cdot \|x - y\| dt \\
& \leq \int_0^1 \beta t \|x-y\|^2 dt \\
& = \frac{\beta}{2} \|x-y\|^2 .
\end{align*}
\end{proof}

In particular this lemma shows that if $f$ is convex and $\beta$-smooth, then for any $x, y \in \R^n$, one has
\begin{equation} \label{eq:defaltsmooth}
0 \leq f(x) - f(y) - \nabla f(y)^{\top} (x - y) \leq \frac{\beta}{2} \|x - y\|^2 .
\end{equation}
This gives in particular the following important inequality to evaluate the improvement in one step of gradient descent:
\begin{equation} \label{eq:onestepofgd}
f\left(x - \frac{1}{\beta} \nabla f(x)\right) - f(x) \leq - \frac{1}{2 \beta} \|\nabla f(x)\|^2 .
\end{equation}
The next lemma, which improves the basic inequality for subgradients under the smoothness assumption, shows that in fact $f$ is convex and $\beta$-smooth if and only if \eqref{eq:defaltsmooth} holds true. In the literature \eqref{eq:defaltsmooth} is often used as a definition of smooth convex functions.

\begin{lemma} \label{lem:2}
Let $f$ be such that \eqref{eq:defaltsmooth} holds true. Then for any $x, y \in \R^n$, one has
$$f(x) - f(y) \leq \nabla f(x)^{\top} (x - y) - \frac{1}{2 \beta} \|\nabla f(x) - \nabla f(y)\|^2 .$$
\end{lemma}

\begin{proof}
Let $z = y - \frac{1}{\beta} (\nabla f(y) - \nabla f(x))$. Then one has
\begin{align*}
& f(x) - f(y) \\
& = f(x) - f(z) + f(z) - f(y) \\
& \leq \nabla f(x)^{\top} (x-z) + \nabla f(y)^{\top} (z-y) + \frac{\beta}{2} \|z - y\|^2 \\
& = \nabla f(x)^{\top}(x-y) + (\nabla f(x) - \nabla f(y))^{\top} (y-z) + \frac{1}{2 \beta} \|\nabla f(x) - \nabla f(y)\|^2 \\
& = \nabla f(x)^{\top} (x - y) - \frac{1}{2 \beta} \|\nabla f(x) - \nabla f(y)\|^2 .
\end{align*}
\end{proof}

We can now prove Theorem \ref{th:gdsmooth}

\begin{proof}
Using \eqref{eq:onestepofgd} and the definition of the method one has
$$f(x_{s+1}) - f(x_s) \leq - \frac{1}{2 \beta} \|\nabla f(x_s)\|^2.$$
In particular, denoting $\delta_s = f(x_s) - f(x^*)$, this shows:
$$\delta_{s+1} \leq \delta_s  - \frac{1}{2 \beta} \|\nabla f(x_s)\|^2.$$
One also has by convexity
$$\delta_s \leq \nabla f(x_s)^{\top} (x_s - x^*) \leq \|x_s - x^*\| \cdot \|\nabla f(x_s)\| .$$
We will prove that $\|x_s - x^*\|$ is decreasing with $s$, which with the two above displays will imply
$$\delta_{s+1} \leq \delta_s  - \frac{1}{2 \beta \|x_1 - x^*\|^2} \delta_s^2.$$
Let us see how to use this last inequality to conclude the proof. Let $\omega = \frac{1}{2 \beta \|x_1 - x^*\|^2}$, then\footnote{The last step in the sequence of implications can be improved by taking $\delta_1$ into account. Indeed one can easily show with \eqref{eq:defaltsmooth} that $\delta_1 \leq \frac{1}{4 \omega}$. This improves the rate of Theorem \ref{th:gdsmooth} from $\frac{2 \beta \|x_1 - x^*\|^2}{t-1}$ to $\frac{2 \beta \|x_1 - x^*\|^2}{t+3}$.}
$$\omega \delta_s^2 + \delta_{s+1} \leq \delta_s \Leftrightarrow \omega \frac{\delta_s}{\delta_{s+1}} + \frac{1}{\delta_{s}} \leq \frac{1}{\delta_{s+1}} \Rightarrow \frac{1}{\delta_{s+1}} - \frac{1}{\delta_{s}} \geq \omega \Rightarrow \frac{1}{\delta_t} \geq \omega (t-1) .$$
Thus it only remains to show that $\|x_s - x^*\|$ is decreasing with $s$. Using Lemma \ref{lem:2} one immediately gets
\begin{equation} \label{eq:coercive1}
(\nabla f(x) - \nabla f(y))^{\top} (x - y) \geq \frac{1}{\beta} \|\nabla f(x) - \nabla f(y)\|^2 .
\end{equation}
We use this as follows (together with $\nabla f(x^*) = 0$)
\begin{eqnarray*}
\|x_{s+1} - x^*\|^2& = & \|x_{s} - \frac{1}{\beta} \nabla f(x_s) - x^*\|^2 \\
& = & \|x_{s} - x^*\|^2 - \frac{2}{\beta} \nabla f(x_s)^{\top} (x_s - x^*) + \frac{1}{\beta^2} \|\nabla f(x_s)\|^2 \\
& \leq & \|x_{s} - x^*\|^2 - \frac{1}{\beta^2} \|\nabla f(x_s)\|^2 \\
& \leq & \|x_{s} - x^*\|^2 ,
\end{eqnarray*}
which concludes the proof.
\end{proof}

\subsection*{The constrained case}
We now come back to the constrained problem
\begin{align*}
& \mathrm{min.} \; f(x) \\
& \text{s.t.} \; x \in \cX .
\end{align*}
Similarly to what we did in Section \ref{sec:psgd} we consider the projected gradient descent algorithm, which iterates $x_{t+1} = \Pi_{\cX}(x_t - \eta \nabla f(x_t))$. 

The key point in the analysis of gradient descent for unconstrained smooth optimization is that a step of gradient descent started at $x$ will decrease the function value by at least $\frac{1}{2\beta} \|\nabla f(x)\|^2$, see \eqref{eq:onestepofgd}. In the constrained case we cannot expect that this would still hold true as a step may be cut short by the projection. The next lemma defines the ``right" quantity to measure progress in the constrained case.

\begin{lemma} \label{lem:smoothconst}
Let $x, y \in \cX$, $x^+ = \Pi_{\cX}\left(x - \frac{1}{\beta} \nabla f(x)\right)$, and $g_{\cX}(x) = \beta(x - x^+)$. Then the following holds true:
$$f(x^+) - f(y) \leq g_{\cX}(x)^{\top}(x-y) - \frac{1}{2 \beta} \|g_{\cX}(x)\|^2 .$$
\end{lemma}

\begin{proof}
We first observe that 
\begin{equation} \label{eq:chap3eq1}
\nabla f(x)^{\top} (x^+ - y) \leq g_{\cX}(x)^{\top}(x^+ - y) .
\end{equation}
Indeed the above inequality is equivalent to
$$\left(x^+- \left(x - \frac{1}{\beta} \nabla f(x) \right)\right)^{\top} (x^+ - y) \leq 0, $$
which follows from Lemma \ref{lem:todonow}. Now we use \eqref{eq:chap3eq1} as follows to prove the lemma (we also use \eqref{eq:defaltsmooth} which still holds true in the constrained case)
\begin{align*}
& f(x^+) - f(y) \\
& = f(x^+) - f(x) + f(x) - f(y) \\
& \leq \nabla f(x)^{\top} (x^+-x) + \frac{\beta}{2} \|x^+-x\|^2 + \nabla f(x)^{\top} (x-y) \\
& = \nabla f(x)^{\top} (x^+ - y) + \frac{1}{2 \beta} \|g_{\cX}(x)\|^2 \\
& \leq g_{\cX}(x)^{\top}(x^+ - y) + \frac{1}{2 \beta} \|g_{\cX}(x)\|^2 \\
& = g_{\cX}(x)^{\top}(x - y) - \frac{1}{2 \beta} \|g_{\cX}(x)\|^2 .
\end{align*}
\end{proof}
We can now prove the following result.
\begin{theorem} \label{th:gdsmoothconstrained}
Let $f$ be convex and $\beta$-smooth on $\cX$. Then projected gradient descent with $\eta = \frac{1}{\beta}$ satisfies
$$f(x_t) - f(x^*) \leq \frac{3 \beta \|x_1 - x^*\|^2 + f(x_1) - f(x^*)}{t} .$$
\end{theorem}

\begin{proof}
Lemma \ref{lem:smoothconst} immediately gives
$$f(x_{s+1}) - f(x_s) \leq - \frac{1}{2 \beta} \|g_{\cX}(x_s)\|^2 ,$$
and
$$f(x_{s+1}) - f(x^*) \leq \|g_{\cX}(x_s)\| \cdot \|x_s - x^*\| .$$
We will prove that $\|x_s - x^*\|$ is decreasing with $s$, which with the two above displays will imply
$$\delta_{s+1} \leq \delta_s  - \frac{1}{2 \beta \|x_1 - x^*\|^2} \delta_{s+1}^2.$$
An easy induction shows that
$$\delta_s \leq \frac{3 \beta \|x_1 - x^*\|^2 + f(x_1) - f(x^*)}{s}.$$
Thus it only remains to show that $\|x_s - x^*\|$ is decreasing with $s$. Using Lemma \ref{lem:smoothconst} one can see that $g_{\cX}(x_s)^{\top} (x_s - x^*) \geq \frac{1}{2 \beta} \|g_{\cX}(x_s)\|^2$ which implies
\begin{eqnarray*}
\|x_{s+1} - x^*\|^2& = & \|x_{s} - \frac{1}{\beta} g_{\cX}(x_s) - x^*\|^2 \\
& = & \|x_{s} - x^*\|^2 - \frac{2}{\beta} g_{\cX}(x_s)^{\top} (x_s - x^*) + \frac{1}{\beta^2} \|g_{\cX}(x_s)\|^2 \\
& \leq & \|x_{s} - x^*\|^2 .
\end{eqnarray*}
\end{proof}

\section{Conditional gradient descent, aka Frank-Wolfe} \label{sec:FW}
%
We describe now an alternative algorithm to minimize a smooth convex function $f$ over a compact convex set $\mathcal{X}$. The {\em conditional gradient descent}, introduced in \cite{FW56}, performs the following update for $t \geq 1$, where $(\gamma_s)_{s \geq 1}$ is a fixed sequence,
\begin{align}
&y_{t} \in \mathrm{argmin}_{y \in \mathcal{X}} \nabla f(x_t)^{\top} y \label{eq:FW1} \\
& x_{t+1} = (1 - \gamma_t) x_t + \gamma_t y_t . \label{eq:FW2}
\end{align}
In words conditional gradient descent makes a step in the steepest descent direction {\em given the constraint set $\cX$}, see Figure \ref{fig:FW} for an illustration. From a computational perspective, a key property of this scheme is that it replaces the projection step of projected gradient descent by a linear optimization over $\cX$, which in some cases can be a much simpler problem. 

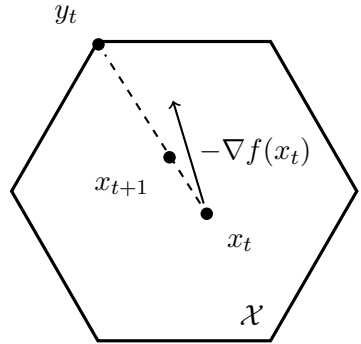
\begin{figure}
\begin{center}
\begin{tikzpicture}[scale=3]
\node [tokens=1] (noeud1) at (0.1,-0.1) [label=below right:{$x_t$}] {};
\node [tokens=1] (noeud2) at (-0.38,0.65) [label=above left:{$y_{t}$}] {};
\draw[->, thick] (noeud1) -- (-0.05, 0.4) node[midway, right] {$-\nabla f(x_t)$};
\node [tokens=1] (noeud3) at (-0.065,0.15) [label=below left:{$x_{t+1}$}] {};
\draw[thick, dashed] (noeud1) -- (noeud2) {};
\node at (0.3,-0.55) {$\cX$};
\node (S) [very thick, regular polygon, regular polygon sides=6, draw,
inner sep=40] at (0,0) {};
\end{tikzpicture}
\end{center}
\caption{Illustration of conditional gradient descent.}
\label{fig:FW}
\end{figure}

We now turn to the analysis of this method. A major advantage of conditional gradient descent over projected gradient descent is that the former can adapt to smoothness in an arbitrary norm. Precisely let $f$ be $\beta$-smooth in some norm $\|\cdot\|$, that is $\|\nabla f(x) - \nabla f(y) \|_* \leq \beta \|x-y\|$ where the dual norm $\|\cdot\|_*$ is defined as $\|g\|_* = \sup_{x \in \mathbb{R}^n : \|x\| \leq 1} g^{\top} x$. The following result is extracted from \cite{Jag13} (see also \cite{DH78}).
%

\begin{theorem}
Let $f$ be a convex and $\beta$-smooth function w.r.t. some norm $\|\cdot\|$, $R = \sup_{x, y \in \mathcal{X}} \|x - y\|$, and $\gamma_s = \frac{2}{s+1}$ for $s \geq 1$. Then for any $t \geq 2$, one has
$$f(x_t) - f(x^*) \leq \frac{2 \beta R^2}{t+1} .$$
\end{theorem}

\begin{proof}
The following inequalities hold true, using respectively $\beta$-smoothness (it can easily be seen that \eqref{eq:defaltsmooth} holds true for smoothness in an arbitrary norm), the definition of $x_{s+1}$, the definition of $y_s$, and the convexity of $f$:
\begin{eqnarray*}
f(x_{s+1}) - f(x_s) & \leq & \nabla f(x_s)^{\top} (x_{s+1} - x_s) + \frac{\beta}{2} \|x_{s+1} - x_s\|^2 \\
& \leq & \gamma_s \nabla f(x_s)^{\top} (y_{s} - x_s) + \frac{\beta}{2} \gamma_s^2 R^2 \\
& \leq & \gamma_s \nabla f(x_s)^{\top} (x^* - x_s) + \frac{\beta}{2} \gamma_s^2 R^2 \\
& \leq & \gamma_s (f(x^*) - f(x_s)) + \frac{\beta}{2} \gamma_s^2 R^2 .
\end{eqnarray*}
Rewriting this inequality in terms of $\delta_s = f(x_s) - f(x^*)$ one obtains
$$\delta_{s+1} \leq (1 - \gamma_s) \delta_s + \frac{\beta}{2} \gamma_s^2 R^2 .$$
A simple induction using that $\gamma_s = \frac{2}{s+1}$ finishes the proof (note that the initialization is done at step $2$ with the above inequality yielding $\delta_2 \leq \frac{\beta}{2} R^2$).
\end{proof}

In addition to being projection-free and ``norm-free", the conditional gradient descent satisfies a perhaps even more important property: it produces {\em sparse iterates}. More precisely consider the situation where $\cX \subset \R^n$ is a polytope, that is the convex hull of a finite set of points (these points are called the vertices of $\cX$). Then Carath\'eodory's theorem states that any point $x \in \cX$ can be written as a convex combination of at most $n+1$ vertices of $\cX$. On the other hand, by definition of the conditional gradient descent, one knows that the $t^{th}$ iterate $x_t$ can be written as a convex combination of $t$ vertices (assuming that $x_1$ is a vertex). Thanks to the dimension-free rate of convergence one is usually interested in the regime where $t \ll n$, and thus we see that the iterates of conditional gradient descent are very sparse in their vertex representation.

We note an interesting corollary of the sparsity property together with the rate of convergence we proved: smooth functions on the simplex $\{x \in \mathbb{R}_+^n : \sum_{i=1}^n x_i = 1\}$ always admit sparse approximate minimizers. More precisely there must exist a point $x$ with only $t$ non-zero coordinates and such that $f(x) - f(x^*) = O(1/t)$. Clearly this is the best one can hope for in general, as it can be seen with the function $f(x) = \|x\|^2_2$ since by Cauchy-Schwarz one has $\|x\|_1 \leq \sqrt{\|x\|_0} \|x\|_2$ which implies on the simplex $\|x\|_2^2 \geq 1 / \|x\|_0$.
%

Next we describe an application where the three properties of conditional gradient descent (projection-free, norm-free, and sparse iterates) are critical to develop a computationally efficient procedure.

\subsection*{An application of conditional gradient descent: Least-squares regression with structured sparsity}
This example is inspired by \cite{Lug10} (see also \cite{Jon92}). Consider the problem of approximating a signal $Y \in \mathbb{R}^n$ by a ``small" combination of dictionary elements $d_1, \hdots, d_N \in \mathbb{R}^n$. One way to do this is to consider a LASSO type problem in dimension $N$ of the following form (with $\lambda \in \R$ fixed)
$$\min_{x \in \mathbb{R}^N} \big\| Y - \sum_{i=1}^N x(i) d_i \big\|_2^2 + \lambda \|x\|_1 .$$
Let $D \in \mathbb{R}^{n \times N}$ be the dictionary matrix with $i^{th}$ column given by $d_i$. Instead of considering the penalized version of the problem one could look at the following constrained problem (with $s \in \R$ fixed) on which we will now focus, see e.g. \cite{FT07},
\begin{eqnarray}
\min_{x \in \mathbb{R}^N} \| Y - D x \|_2^2 
& \qquad \Leftrightarrow \qquad & \min_{x \in \mathbb{R}^N} \| Y / s - D x \|_2^2 \label{eq:structuredsparsity} \\
\text{subject to}  \; \|x\|_1 \leq s
& & \text{subject to} \; \|x\|_1 \leq 1 . \notag
\end{eqnarray}
We make some assumptions on the dictionary. We are interested in situations where the size of the dictionary $N$ can be very large, potentially exponential in the ambient dimension $n$. Nonetheless we want to restrict our attention to algorithms that run in reasonable time with respect to the ambient dimension $n$, that is we want polynomial time algorithms in $n$. Of course in general this is impossible, and we need to assume that the dictionary has some structure that can be exploited. Here we make the assumption that one can do {\em linear optimization} over the dictionary in polynomial time in $n$. More precisely we assume that one can solve in time $p(n)$ (where $p$ is polynomial) the following problem for any $y \in \mathbb{R}^n$:
$$\min_{1 \leq i \leq N} y^{\top} d_i .$$
This assumption is met for many {\em combinatorial} dictionaries. For instance the dic­tio­nary ele­ments could be vec­tor of inci­dence of span­ning trees in some fixed graph, in which case the lin­ear opti­miza­tion prob­lem can be solved with a greedy algorithm.

Finally, for normalization issues, we assume that the $\ell_2$-norm of the dictionary elements are controlled by some $m>0$, that is $\|d_i\|_2 \leq m, \forall i \in [N]$.

Our problem of interest \eqref{eq:structuredsparsity} corresponds to minimizing the function $f(x) = \frac{1}{2} \| Y - D x \|^2_2$ on the $\ell_1$-ball of $\mathbb{R}^N$ in polynomial time in $n$. At first sight this task may seem completely impossible, indeed one is not even allowed to write down entirely a vector $x \in \mathbb{R}^N$ (since this would take time linear in $N$). The key property that will save us is that this function admits {\em sparse minimizers} as we discussed in the previous section, and this will be exploited by the conditional gradient descent method. 

First let us study the computational complexity of the $t^{th}$ step of conditional gradient descent. Observe that
$$\nabla f(x) = D^{\top} (D x - Y).$$
Now assume that $z_t = D x_t - Y \in \mathbb{R}^n$ is already computed, then to compute \eqref{eq:FW1} one needs to find the coordinate $i_t \in [N]$ that maximizes $|[\nabla f(x_t)](i)|$ which can be done by maximizing $d_i^{\top} z_t$ and $- d_i^{\top} z_t$. Thus \eqref{eq:FW1} takes time $O(p(n))$. Computing $x_{t+1}$ from $x_t$ and $i_{t}$ takes time $O(t)$ since $\|x_t\|_0 \leq t$, and computing $z_{t+1}$ from $z_t$ and $i_t$ takes time $O(n)$. Thus the overall time complexity of running $t$ steps is (we assume $p(n) = \Omega(n)$) 
\begin{equation}
O(t p(n) + t^2). \label{eq:structuredsparsity2}
\end{equation} 

To derive a rate of convergence it remains to study the smoothness of $f$. This can be done as follows:
\begin{eqnarray*}
\| \nabla f(x) - \nabla f(y) \|_{\infty} & = & \|D^{\top} D (x-y) \|_{\infty} \\
& = & \max_{1 \leq i \leq N} \bigg| d_i^{\top} \left(\sum_{j=1}^N d_j (x(j) - y(j))\right) \bigg| \\
& \leq & m^2 \|x-y\|_1 ,
\end{eqnarray*}
which means that $f$ is $m^2$-smooth with respect to the $\ell_1$-norm. Thus we get the following rate of convergence:
\begin{equation}
f(x_t) - f(x^*) \leq \frac{8 m^2}{t+1} . \label{eq:structuredsparsity3}
\end{equation}
Putting together \eqref{eq:structuredsparsity2} and \eqref{eq:structuredsparsity3} we proved that one can get an $\epsilon$-optimal solution to \eqref{eq:structuredsparsity} with a computational effort of $O(m^2 p(n)/\epsilon + m^4/\epsilon^2)$ using the conditional gradient descent.

\section{Strong convexity}
We will now discuss another property of convex functions that can significantly speed-up the convergence of first order methods: strong convexity. We say that $f: \cX \rightarrow \mathbb{R}$ is $\alpha$-{\em strongly convex} if it satisfies the following improved subgradient inequality:
\begin{equation} \label{eq:defstrongconv}
f(x) - f(y) \leq \nabla f(x)^{\top} (x - y) - \frac{\alpha}{2} \|x - y \|^2 .
\end{equation}
Of course this definition does not require differentiability of the function $f$, and one can replace $\nabla f(x)$ in the inequality above by $g \in \partial f(x)$. It is immediate to verify that a function $f$ is $\alpha$-strongly convex if and only if $x \mapsto f(x) - \frac{\alpha}{2} \|x\|^2$ is convex (in particular if $f$ is twice differentiable then the eigenvalues of the Hessians of $f$ have to be larger than $\alpha$). The strong convexity parameter $\alpha$ is a measure of the {\em curvature} of $f$. For instance a linear function has no curvature and hence $\alpha = 0$. On the other hand one can clearly see why a large value of $\alpha$ would lead to a faster rate: in this case a point far from the optimum will have a large gradient, and thus gradient descent will make very big steps when far from the optimum. Of course if the function is non-smooth one still has to be careful and tune the step-sizes to be relatively small, but nonetheless we will be able to improve the oracle complexity from $O(1/\epsilon^2)$ to $O(1/(\alpha \epsilon))$. On the other hand with the additional assumption of $\beta$-smoothness we will prove that gradient descent with a constant step-size achieves a {\em linear rate of convergence}, precisely the oracle complexity will be $O(\frac{\beta}{\alpha} \log(1/\epsilon))$. This achieves the objective we had set after Theorem \ref{th:pgd}: strongly-convex and smooth functions can be optimized in very large dimension and up to very high accuracy.

Before going into the proofs let us discuss another interpretation of strong-convexity and its relation to smoothness. Equation \eqref{eq:defstrongconv} can be read as follows: at any point $x$ one can find a (convex) quadratic lower bound $q_x^-(y) = f(x) + \nabla f(x)^{\top} (y - x) + \frac{\alpha}{2} \|x - y \|^2$ to the function $f$, i.e. $q_x^-(y) \leq f(y), \forall y \in \cX$ (and $q_x^-(x) = f(x)$). On the other hand for $\beta$-smoothness \eqref{eq:defaltsmooth}
%
implies that at any point $y$ one can find a (convex) quadratic upper bound $q_y^+(x) = f(y) + \nabla f(y)^{\top} (x - y) + \frac{\beta}{2} \|x - y \|^2$ to the function $f$, i.e. $q_y^+(x) \geq f(x), \forall x \in \cX$ (and $q_y^+(y) = f(y)$). 
Thus in some sense strong convexity is a {\em dual} assumption to smoothness, and in fact this can be made precise within the framework of Fenchel duality. Also remark that clearly one always has $\beta \geq \alpha$.

\subsection{Strongly convex and Lipschitz functions}
We consider here the projected subgradient descent algorithm with time-varying step size $(\eta_t)_{t \geq 1}$, that is
\begin{align*}
& y_{t+1} = x_t - \eta_t g_t , \ \text{where} \ g_t \in \partial f(x_t) \\
& x_{t+1} = \Pi_{\cX}(y_{t+1}) .
\end{align*}
The following result is extracted from \cite{LJSB12}.

\begin{theorem} \label{th:LJSB12}
Let $f$ be $\alpha$-strongly convex and $L$-Lipschitz on $\cX$. Then projected subgradient descent with $\eta_s = \frac{2}{\alpha (s+1)}$ satisfies
$$f \left(\sum_{s=1}^t \frac{2 s}{t(t+1)} x_s \right) - f(x^*) \leq \frac{2 L^2}{\alpha (t+1)} .$$
\end{theorem}

\begin{proof}
Coming back to our original analysis of projected subgradient descent in Section \ref{sec:psgd} and using the strong convexity assumption one immediately obtains
$$f(x_s) - f(x^*) \leq \frac{\eta_s}{2} L^2 + \left( \frac{1}{2 \eta_s} - \frac{\alpha}{2} \right) \|x_s - x^*\|^2 - \frac{1}{2 \eta_s} \|x_{s+1} - x^*\|^2 .$$
Multiplying this inequality by $s$ yields
$$s( f(x_s) - f(x^*) ) \leq \frac{L^2}{\alpha} + \frac{\alpha}{4} \bigg( s(s-1) \|x_s - x^*\|^2 - s (s+1) \|x_{s+1} - x^*\|^2 \bigg),$$
Now sum the resulting inequality over $s=1$ to $s=t$, and apply Jensen's inequality to obtain the claimed statement.
\end{proof}

\subsection{Strongly convex and smooth functions}
As we will see now, having both strong convexity and smoothness allows for a drastic improvement in the convergence rate. We denote $\kappa= \frac{\beta}{\alpha}$ for the {\em condition number} of $f$. The key observation is that Lemma \ref{lem:smoothconst} can be improved to (with the notation of the lemma):
\begin{equation} \label{eq:improvedstrongsmooth}
f(x^+) - f(y) \leq g_{\cX}(x)^{\top}(x-y) - \frac{1}{2 \beta} \|g_{\cX}(x)\|^2 - \frac{\alpha}{2} \|x-y\|^2 .
\end{equation}

\begin{theorem} \label{th:gdssc}
Let $f$ be $\alpha$-strongly convex and $\beta$-smooth on $\cX$. Then projected gradient descent with $\eta = \frac{1}{\beta}$ satisfies for $t \geq 0$,
$$\|x_{t+1} - x^*\|^2 \leq \exp\left( - \frac{t}{\kappa} \right) \|x_1 - x^*\|^2 .$$
\end{theorem}

\begin{proof}
Using \eqref{eq:improvedstrongsmooth} with $y=x^*$ one directly obtains
\begin{eqnarray*}
\|x_{t+1} - x^*\|^2& = & \|x_{t} - \frac{1}{\beta} g_{\cX}(x_t) - x^*\|^2 \\
& = & \|x_{t} - x^*\|^2 - \frac{2}{\beta} g_{\cX}(x_t)^{\top} (x_t - x^*) + \frac{1}{\beta^2} \|g_{\cX}(x_t)\|^2 \\
& \leq & \left(1 - \frac{\alpha}{\beta} \right) \|x_{t} - x^*\|^2 \\
& \leq & \left(1 - \frac{\alpha}{\beta} \right)^t \|x_{1} - x^*\|^2 \\
& \leq & \exp\left( - \frac{t}{\kappa} \right) \|x_1 - x^*\|^2 ,
\end{eqnarray*}
which concludes the proof.
\end{proof}

We now show that in the unconstrained case one can improve the rate by a constant factor, precisely one can replace $\kappa$ by $(\kappa+1) / 4$ in the oracle complexity bound by using a larger step size. This is not a spectacular gain but the reasoning is based on an improvement of \eqref{eq:coercive1} which can be of interest by itself. Note that \eqref{eq:coercive1} and the lemma to follow are sometimes referred to as {\em coercivity} of the gradient.

\begin{lemma} \label{lem:coercive2}
Let $f$ be $\beta$-smooth and $\alpha$-strongly convex on $\R^n$. Then for all $x, y \in \mathbb{R}^n$, one has
$$(\nabla f(x) - \nabla f(y))^{\top} (x - y) \geq \frac{\alpha \beta}{\beta + \alpha} \|x-y\|^2 + \frac{1}{\beta + \alpha} \|\nabla f(x) - \nabla f(y)\|^2 .$$
\end{lemma}

\begin{proof}
Let $\phi(x) = f(x) - \frac{\alpha}{2} \|x\|^2$. By definition of $\alpha$-strong convexity one has that $\phi$ is convex. Furthermore one can show that $\phi$ is $(\beta-\alpha)$-smooth by proving \eqref{eq:defaltsmooth} (and using that it implies smoothness). Thus using \eqref{eq:coercive1} one gets
$$(\nabla \phi(x) - \nabla \phi(y))^{\top} (x - y) \geq \frac{1}{\beta - \alpha} \|\nabla \phi(x) - \nabla \phi(y)\|^2 ,$$
which gives the claimed result with straightforward computations. (Note that if $\alpha = \beta$ the smoothness of $\phi$ directly implies that $\nabla f(x) - \nabla f(y) = \alpha (x-y)$ which proves the lemma in this case.)
\end{proof}

\begin{theorem}
Let $f$ be $\beta$-smooth and $\alpha$-strongly convex on $\R^n$. Then gradient descent with $\eta = \frac{2}{\alpha + \beta}$ satisfies
$$f(x_{t+1}) - f(x^*) \leq \frac{\beta}{2} \exp\left( - \frac{4 t}{\kappa+1} \right) \|x_1 - x^*\|^2 .$$
\end{theorem}

\begin{proof}
First note that by $\beta$-smoothness (since $\nabla f(x^*) = 0$) one has
$$f(x_t) - f(x^*) \leq \frac{\beta}{2} \|x_t - x^*\|^2 .$$
Now using Lemma \ref{lem:coercive2} one obtains
\begin{eqnarray*}
\|x_{t+1} - x^*\|^2& = & \|x_{t} - \eta \nabla f(x_{t}) - x^*\|^2 \\
& = & \|x_{t} - x^*\|^2 - 2 \eta \nabla f(x_{t})^{\top} (x_{t} - x^*) + \eta^2 \|\nabla f(x_{t})\|^2 \\
& \leq & \left(1 - 2 \frac{\eta \alpha \beta}{\beta + \alpha}\right)\|x_{t} - x^*\|^2 + \left(\eta^2 - 2 \frac{\eta}{\beta + \alpha}\right) \|\nabla f(x_{t})\|^2 \\
& = & \left(\frac{\kappa - 1}{\kappa+1}\right)^2 \|x_{t} - x^*\|^2 \\
& \leq & \exp\left( - \frac{4 t}{\kappa+1} \right) \|x_1 - x^*\|^2 ,
\end{eqnarray*}
which concludes the proof.
\end{proof}

\section{Lower bounds} \label{sec:chap3LB}
We prove here various oracle complexity lower bounds. These results first appeared in \cite{NY83} but we follow here the simplified presentation of \cite{Nes04}. In general a black-box procedure is a mapping from ``history" to the next query point, that is it maps $(x_1, g_1, \hdots, x_t, g_t)$ (with $g_s \in \partial f (x_s)$) to $x_{t+1}$. In order to simplify the notation and the argument, throughout the section we make the following assumption on the black-box procedure: $x_1=0$ and for any $t \geq 0$, $x_{t+1}$ is in the linear span of $g_1, \hdots, g_t$, that is
\begin{equation} \label{eq:ass1}
x_{t+1} \in \mathrm{Span}(g_1, \hdots, g_t) .
\end{equation}
Let $e_1, \hdots, e_n$ be the canonical basis of $\mathbb{R}^n$, and $\mB_2(R) = \{x \in \R^n : \|x\| \leq R\}$. We start with a theorem for the two non-smooth cases (convex and strongly convex).

\begin{theorem} \label{th:lb1}
Let $t \leq n$, $L, R >0$. There exists a convex and $L$-Lipschitz function $f$ such that for any black-box procedure satisfying \eqref{eq:ass1},
$$\min_{1 \leq s \leq t} f(x_s) - \min_{x \in \mB_2(R)} f(x) \geq  \frac{R L}{2 (1 + \sqrt{t})} .$$
There also exists an $\alpha$-strongly convex and $L$-lipschitz function $f$ such that for any black-box procedure satisfying \eqref{eq:ass1},
$$\min_{1 \leq s \leq t} f(x_s) - \min_{x \in \mB_2\left(\frac{L}{2 \alpha}\right)} f(x) \geq  \frac{L^2}{8 \alpha t} .$$
\end{theorem}

Note that the above result is restricted to a number of iterations smaller than the dimension, that is $t \leq n$. This restriction is of course necessary to obtain lower bounds polynomial in $1/t$: as we saw in Chapter \ref{finitedim} one can always obtain an exponential rate of convergence when the number of calls to the oracle is larger than the dimension. 

\begin{proof}
We consider the following $\alpha$-strongly convex function:
$$f(x) = \gamma \max_{1 \leq i \leq t} x(i) + \frac{\alpha}{2} \|x\|^2 .$$
It is easy to see that
$$\partial f(x) = \alpha x + \gamma \conv\left(e_i , i : x(i) = \max_{1 \leq j \leq t} x(j) \right).$$
In particular if $\|x\| \leq R$ then for any $g \in \partial f(x)$ one has $\|g\| \leq \alpha R + \gamma$. In other words $f$ is $(\alpha R + \gamma)$-Lipschitz on $\mB_2(R)$.

Next we describe the first order oracle for this function: when asked for a subgradient at $x$, it returns $\alpha x + \gamma e_{i}$ where $i$ is the {\em first} coordinate that satisfies $x(i) = \max_{1 \leq j \leq t} x(j)$. In particular when asked for a subgradient at $x_1=0$ it returns $e_1$. Thus $x_2$ must lie on the line generated by $e_1$. It is easy to see by induction that in fact $x_s$ must lie in the linear span of $e_1, \hdots, e_{s-1}$. In particular for $s \leq t$ we necessarily have $x_s(t) = 0$ and thus $f(x_s) \geq 0$.

It remains to compute the minimal value of $f$. Let $y$ be such that $y(i) = - \frac{\gamma}{\alpha t}$ for $1 \leq i \leq t$ and $y(i) = 0$ for $t+1 \leq i \leq n$. It is clear that $0 \in \partial f(y)$ and thus the minimal value of $f$ is
$$f(y) = - \frac{\gamma^2}{\alpha t} + \frac{\alpha}{2} \frac{\gamma^2}{\alpha^2 t} = - \frac{\gamma^2}{2 \alpha t} .$$ 

Wrapping up, we proved that for any $s \leq t$ one must have
$$f(x_s) - f(x^*) \geq \frac{\gamma^2}{2 \alpha t} .$$
Taking $\gamma = L/2$ and $R= \frac{L}{2 \alpha}$ we proved the lower bound for $\alpha$-strongly convex functions (note in particular that $\|y\|^2 = \frac{\gamma^2}{\alpha^2 t} = \frac{L^2}{4 \alpha^2 t} \leq R^2$ with these parameters). On the other taking $\alpha = \frac{L}{R} \frac{1}{1 + \sqrt{t}}$ and $\gamma = L \frac{\sqrt{t}}{1 + \sqrt{t}}$ concludes the proof for convex functions (note in particular that $\|y\|^2 = \frac{\gamma^2}{\alpha^2 t} = R^2$ with these parameters).
\end{proof}

We proceed now to the smooth case. As we will see in the following proofs we restrict our attention to quadratic functions, and it might be useful to recall that in this case one can attain the exact optimum in $n$ calls to the oracle (see Section \ref{sec:CG}). We also recall that for a twice differentiable function $f$, $\beta$-smoothness is equivalent to the largest eigenvalue of the Hessians of $f$ being smaller than $\beta$ at any point, which we write
$$\nabla^2 f(x) \preceq \beta \mI_n , \forall x .$$
Furthermore $\alpha$-strong convexity is equivalent to 
$$\nabla^2 f(x) \succeq \alpha \mI_n , \forall x .$$

\begin{theorem} \label{th:lb2}
Let $t \leq (n-1)/2$, $\beta >0$. There exists a $\beta$-smooth convex function $f$ such that for any black-box procedure satisfying \eqref{eq:ass1},
$$\min_{1 \leq s \leq t} f(x_s) - f(x^*) \geq  \frac{3 \beta}{32} \frac{\|x_1 - x^*\|^2}{(t+1)^2} .$$
\end{theorem}

\begin{proof} In this proof for $h: \R^n \rightarrow \R$ we denote $h^* = \inf_{x \in \R^n} h(x)$.
For $k \leq n$ let $A_k \in \R^{n \times n}$ be the symmetric and tridiagonal matrix defined by
$$(A_k)_{i,j}  = \left\{\begin{array}{ll} 
2, & i = j, i \leq k \\
-1, & j \in \{i-1, i+1\}, i \leq k, j \neq k+1\\
0, & \text{otherwise}.
\end{array}\right.$$
It is easy to verify that $0 \preceq A_k \preceq 4 \mI_n$ since
$$x^{\top} A_k x = 2 \sum_{i=1}^k x(i)^2 - 2 \sum_{i=1}^{k-1} x(i) x(i+1) = x(1)^2 + x(k)^2 + \sum_{i=1}^{k-1} (x(i) - x(i+1))^2 .$$

We consider now the following $\beta$-smooth convex function:
$$f(x) = \frac{\beta}{8} x^{\top} A_{2 t + 1} x - \frac{\beta}{4} x^{\top} e_1 .$$
Similarly to what happened in the proof Theorem \ref{th:lb1}, one can see here too that $x_s$ must lie in the linear span of $e_1, \hdots, e_{s-1}$ (because of our assumption on the black-box procedure). In particular for $s \leq t$ we necessarily have $x_s(i) = 0$ for $i=s, \hdots, n$, which implies $x_s^{\top} A_{2 t+1} x_s = x_s^{\top} A_{s} x_s$. In other words, if we denote
$$f_k(x) = \frac{\beta}{8} x^{\top} A_{k} x - \frac{\beta}{4} x^{\top} e_1 ,$$
then we just proved that
$$f(x_s) - f^* = f_s(x_s) - f_{2t+1}^* \geq f_{s}^* - f_{2 t + 1}^* \geq f_{t}^* - f_{2 t + 1}^* .$$
Thus it simply remains to compute the minimizer $x^*_k$ of $f_k$, its norm, and the corresponding function value $f_k^*$.

The point $x^*_k$ is the unique solution in the span of $e_1, \hdots, e_k$ of $A_k x = e_1$. It is easy to verify that it is defined by $x^*_k(i) = 1 - \frac{i}{k+1}$ for $i=1, \hdots, k$. Thus we immediately have:
$$f^*_k = \frac{\beta}{8} (x^*_k)^{\top} A_{k} x^*_k - \frac{\beta}{4} (x^*_k)^{\top} e_1 = - \frac{\beta}{8} (x^*_k)^{\top} e_1 = - \frac{\beta}{8} \left(1 - \frac{1}{k+1}\right) .$$
Furthermore note that
$$\|x^*_k\|^2 = \sum_{i=1}^k \left(1 - \frac{i}{k+1}\right)^2 = \sum_{i=1}^k \left( \frac{i}{k+1}\right)^2 \leq \frac{k+1}{3} .$$
Thus one obtains:
$$f_{t}^* - f_{2 t+1}^* = \frac{\beta}{8} \left(\frac{1}{t+1} - \frac{1}{2 t + 2} \right) \geq \frac{3 \beta}{32} \frac{\|x^*_{2 t + 1}\|^2}{(t+1)^2},$$
which concludes the proof.
\end{proof}

To simplify the proof of the next theorem we will consider the limiting situation $n \to +\infty$. More precisely we assume now that we are working in $\ell_2 = \{ x = (x(n))_{n \in \mathbb{N}} : \sum_{i=1}^{+\infty} x(i)^2 < + \infty\}$ rather than in $\mathbb{R}^n$. Note that all the theorems we proved in this chapter are in fact valid in an arbitrary Hilbert space $\mathcal{H}$. We chose to work in $\mathbb{R}^n$ only for clarity of the exposition.

\begin{theorem} \label{th:lb3}
Let $\kappa > 1$. There exists a $\beta$-smooth and $\alpha$-strongly convex function $f: \ell_2 \rightarrow \mathbb{R}$ with $\kappa = \beta / \alpha$ such that for any $t \geq 1$ and any black-box procedure satisfying \eqref{eq:ass1} one has
$$f(x_t) - f(x^*) \geq  \frac{\alpha}{2}  \left(\frac{\sqrt{\kappa} - 1}{\sqrt{\kappa}+1}\right)^{2 (t-1)} \|x_1 - x^*\|^2 .$$
\end{theorem}

Note that for large values of the condition number $\kappa$ one has 
$$\left(\frac{\sqrt{\kappa} - 1}{\sqrt{\kappa}+1}\right)^{2 (t-1)} \approx \exp\left(- \frac{4 (t-1)}{\sqrt{\kappa}} \right) .$$

\begin{proof}
The overall argument is similar to the proof of Theorem \ref{th:lb2}. Let $A : \ell_2 \rightarrow \ell_2$ be the linear operator that corresponds to the infinite tridiagonal matrix with $2$ on the diagonal and $-1$ on the upper and lower diagonals. We consider now the following function:
$$f(x) = \frac{\alpha (\kappa-1)}{8} \left(\langle Ax, x\rangle - 2 \langle e_1, x \rangle \right) + \frac{\alpha}{2} \|x\|^2 .$$
We already proved that $0 \preceq A \preceq 4 \mI$ which easily implies that $f$ is $\alpha$-strongly convex and $\beta$-smooth. Now as always the key observation is that for this function, thanks to our assumption on the black-box procedure, one necessarily has $x_t(i) = 0, \forall i \geq t$. This implies in particular:
$$\|x_t - x^*\|^2 \geq \sum_{i=t}^{+\infty} x^*(i)^2 .$$
Furthermore since $f$ is $\alpha$-strongly convex, one has
$$f(x_t) - f(x^*) \geq \frac{\alpha}{2} \|x_t - x^*\|^2 .$$
Thus it only remains to compute $x^*$. This can be done by differentiating $f$ and setting the gradient to $0$, which gives the following infinite set of equations
\begin{align*}
& 1 - 2 \frac{\kappa+1}{\kappa-1} x^*(1) + x^*(2) = 0 , \\
& x^*(k-1) - 2 \frac{\kappa+1}{\kappa-1} x^*(k) + x^*(k+1) = 0, \forall k \geq 2 .
\end{align*}
It is easy to verify that $x^*$ defined by $x^*(i) = \left(\frac{\sqrt{\kappa} - 1}{\sqrt{\kappa} + 1}\right)^i$ satisfy this infinite set of equations, and the conclusion of the theorem then follows by straightforward computations.
\end{proof}

\section{Geometric descent} \label{sec:GeoD}
So far our results leave a gap in the case of smooth optimization: gradient descent achieves an oracle complexity of $O(1/\epsilon)$ (respectively $O(\kappa \log(1/\epsilon))$ in the strongly convex case) while we proved a lower bound of $\Omega(1/\sqrt{\epsilon})$ (respectively $\Omega(\sqrt{\kappa} \log(1/\epsilon))$). In this section we close these gaps with the geometric descent method which was recently introduced in \cite{BLS15}. Historically the first method with optimal oracle complexity was proposed in \cite{NY83}. This method, inspired by the conjugate gradient (see Section \ref{sec:CG}), assumes an oracle to compute {\em plane searches}. In \cite{Nem82} this assumption was relaxed to a line search oracle (the geometric descent method also requires a line search oracle). Finally in \cite{Nes83} an optimal method requiring only a first order oracle was introduced. The latter algorithm, called Nesterov's accelerated gradient descent, has been the most influential optimal method for smooth optimization up to this day. We describe and analyze this method in Section \ref{sec:AGD}. As we shall see the intuition behind Nesterov's accelerated gradient descent (both for the derivation of the algorithm and its analysis) is not quite transparent, which motivates the present section as geometric descent has a simple geometric interpretation loosely inspired from the ellipsoid method (see Section \ref{sec:ellipsoid}).

We focus here on the unconstrained optimization of a smooth and strongly convex function, and we prove that geometric descent achieves the oracle complexity of $O(\sqrt{\kappa} \log(1/\epsilon))$, thus reducing the complexity of the basic gradient descent by a factor $\sqrt{\kappa}$. We note that this improvement is quite relevant for machine learning applications. Consider for example the logistic regression problem described in Section \ref{sec:mlapps}: this is a smooth and strongly convex problem, with a smoothness of order of a numerical constant, but with strong convexity equal to the regularization parameter whose inverse can be as large as the sample size. Thus in this case $\kappa$ can be of order of the sample size, and a faster rate by a factor of $\sqrt{\kappa}$ is quite significant. We also observe that this improved rate for smooth and strongly convex objectives also implies an almost optimal rate of $O(\log(1/\epsilon) / \sqrt{\epsilon})$ for the smooth case, as one can simply run geometric descent on the function $x \mapsto f(x) + \epsilon \|x\|^2$. 

In Section \ref{sec:warmup} we describe the basic idea of geometric descent, and we show how to obtain effortlessly a geometric method with an oracle complexity of $O(\kappa \log(1/\epsilon))$ (i.e., similar to gradient descent). Then we explain why one should expect to be able to accelerate this method in Section \ref{sec:accafterwarmup}. The geometric descent method is described precisely and analyzed in Section \ref{sec:GeoDmethod}.

\subsection{Warm-up: a geometric alternative to gradient descent} \label{sec:warmup}
\begin{figure}
\begin{center}
\begin{tikzpicture}[scale=0.7, every node/.style={transform shape}]

\draw  (0,0) ellipse (2 and 2);
\draw[dashed]  (4,0) ellipse (4 and 4);
\draw  (4,0) ellipse (3.85 and 3.85);
\draw [|-|] (4,0) -- (4,4) node[pos=0.5, right] {$|g|$};
\draw [|-|] (4,0) -- (7.85,0) node[pos=0.5, above] {$\sqrt{1-\epsilon}\ |g|$};

\begin{scope}
  \clip (0,0) ellipse (2 and 2);
  \fill[lightgray] (4,0) ellipse (3.85 and 3.85);
\end{scope}

\draw (0,0) node[cross=5pt] {};

\draw [|-|] (0,0) -- (-2,0) node[pos=0.4, above] {$1$};
\draw [|-|] (0.64719,0) -- (2.53958,0) node[pos=0.35, above] {$\sqrt{1-\epsilon}$};
\draw[thick]  (0.64719,0) ellipse (1.8924 and 1.8924);

\end{tikzpicture}
\end{center}
\caption{One ball shrinks.}
\label{fig:one_ball}
\end{figure}
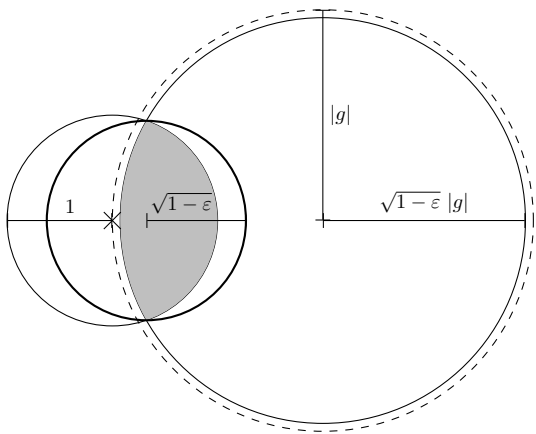

We start with some notation. Let $\mB(x,r^2) := \{y \in \R^n : \|y-x\|^2 \leq r^2 \}$ (note that the second argument is the radius squared), and
$$x^+ = x - \frac{1}{\beta} \nabla f(x), \ \text{and} \ x^{++} = x - \frac{1}{\alpha} \nabla f(x) . $$
Rewriting the definition of strong convexity \eqref{eq:defstrongconv} as
\begin{eqnarray*}
& f(y) \geq f(x) + \nabla f(x)^{\top} (y-x) + \frac{\alpha}{2} \|y-x\|^2 \\
& \Leftrightarrow \ \frac{\alpha}{2} \|y - x + \frac{1}{\alpha} \nabla f(x) \|^2 \leq \frac{\|\nabla f(x)\|^2}{2 \alpha} - (f(x) - f(y)),
\end{eqnarray*}
one obtains an enclosing ball for the minimizer of $f$ with the $0^{th}$ and $1^{st}$ order information at $x$:
$$x^* \in \mB\left(x^{++}, \frac{\|\nabla f(x)\|^2}{\alpha^2} - \frac{2}{\alpha} (f(x) - f(x^*)) \right) .$$
Furthermore recall that by smoothness (see \eqref{eq:onestepofgd}) one has $f(x^+) \leq f(x) - \frac{1}{2 \beta} \|\nabla f(x)\|^2$ which allows to \emph{shrink} the above ball by a factor of $1-\frac{1}{\kappa}$ and obtain the following:
\begin{equation} \label{eq:ball2}
x^* \in \mB\left(x^{++}, \frac{\|\nabla f(x)\|^2}{\alpha^2} \left(1 - \frac{1}{\kappa}\right) - \frac{2}{\alpha} (f(x^+) - f(x^*)) \right) 
\end{equation}
This suggests a natural strategy: assuming that one has an enclosing ball $A:=\mB(x,R^2)$ for $x^*$ (obtained from previous steps of the strategy), one can then enclose $x^*$ in a ball $B$ containing the intersection of $\mB(x,R^2)$ and the ball $\mB\left(x^{++}, \frac{\|\nabla f(x)\|^2}{\alpha^2} \left(1 - \frac{1}{\kappa}\right)\right)$ obtained by \eqref{eq:ball2}. Provided that the radius of $B$ is a fraction of the radius of $A$, one can then iterate the procedure by replacing $A$ by $B$, leading to a linear convergence rate. Evaluating  the rate at which the radius shrinks is an elementary calculation: for any $g \in \R^n$, $\epsilon \in (0,1)$, there exists $x \in \R^n$ such that
$$\mB(0,1) \cap \mB(g, \|g\|^2 (1- \epsilon)) \subset \mB(x, 1-\epsilon) . \quad \quad \text{(Figure \ref{fig:one_ball})}$$
Thus we see that in the strategy described above, the radius squared of the enclosing ball for $x^*$ shrinks by a factor $1 - \frac{1}{\kappa}$ at each iteration, thus matching the rate of convergence of gradient descent (see Theorem \ref{th:gdssc}).

\subsection{Acceleration} \label{sec:accafterwarmup}
\begin{figure}
\begin{center}
\begin{tikzpicture}[scale=0.7, every node/.style={transform shape}]

\draw[dashed]  (0,0) ellipse (2 and 2);
\draw  (0,0) ellipse (1.68 and 1.68);
\draw[dashed]  (4,0) ellipse (4 and 4);
\draw  (4,0) ellipse (3.85 and 3.85);
\draw [|-|] (4,0) -- (7.85,0) node[pos=0.5, above] {$\sqrt{1-\epsilon}\ |g|$};

\begin{scope}
  \clip (0,0) ellipse (1.68 and 1.68);
  \fill[lightgray] (4,0) ellipse (3.85 and 3.85);
\end{scope}

\draw (0,0) node[cross=4pt] {};

\draw [|-|] (0,-3) -- (-1.68,-3) node[pos=0.5, above] {$\sqrt{1-\epsilon |g|^2}$};
\draw [|-|] (0.5,0) -- (2.1044,0) node[pos=0.3, above] {\scriptsize $\sqrt{1-\sqrt{\epsilon}}$};

\draw[thick]  (0.5,0) ellipse (1.6044 and 1.6044);

\end{tikzpicture}
\end{center}
\caption{Two balls shrink.}
\label{fig:two_ball}

\end{figure}
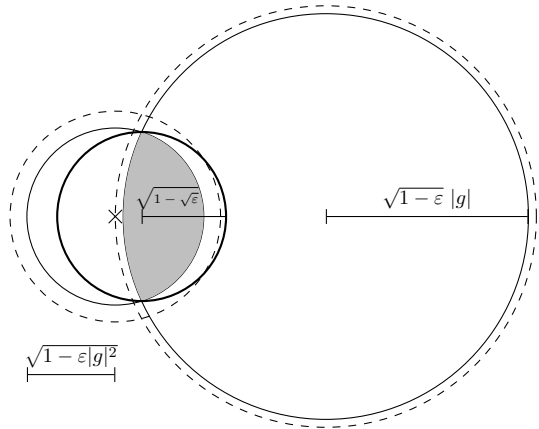

In the argument from the previous section we missed the following opportunity: observe that the ball $A=\mB(x,R^2)$ was obtained by intersections of previous balls of the form given by \eqref{eq:ball2}, and thus the new value $f(x)$ could be used to reduce the radius of those previous balls too (an important caveat is that the value $f(x)$ should be smaller than the values used to build those previous balls). Potentially this could show that the optimum is in fact contained in the ball $\mB\left(x, R^2 - \frac{1}{\kappa} \|\nabla f(x)\|^2\right)$. By taking the intersection with the ball $\mB\left(x^{++}, \frac{\|\nabla f(x)\|^2}{\alpha^2} \left(1 - \frac{1}{\kappa}\right)\right)$ this would allow to obtain a new ball with radius shrunk by a factor $1- \frac{1}{\sqrt{\kappa}}$ (instead of $1 - \frac{1}{\kappa}$): indeed for any $g \in \R^n$, $\epsilon \in (0,1)$, there exists $x \in \R^n$ such that
$$\mB(0,1 - \epsilon \|g\|^2) \cap \mB(g, \|g\|^2 (1- \epsilon)) \subset \mB(x, 1-\sqrt{\epsilon}) . \quad \quad \text{(Figure \ref{fig:two_ball})}$$
Thus it only remains to deal with the caveat noted above, which we do via a line search. In turns this line search might shift the new ball \eqref{eq:ball2}, and to deal with this we shall need the following strengthening of the above set inclusion (we refer to \cite{BLS15} for a simple proof of this result):
\begin{lemma} \label{lem:geom}
Let $a \in \R^n$ and $\epsilon \in (0,1), g \in \R_+$. Assume that $\|a\| \geq g$. Then there exists $c \in \R^n$ such that for any $\delta \geq 0$,
$$\mB(0,1 - \epsilon g^2 - \delta) \cap \mB(a, g^2(1-\epsilon) - \delta) \subset \mB\left(c, 1 - \sqrt{\epsilon} - \delta \right) .$$
\end{lemma}

\subsection{The geometric descent method} \label{sec:GeoDmethod}
Let $x_0 \in \R^n$, $c_0 = x_0^{++}$, and $R_0^2 = \left(1 - \frac{1}{\kappa}\right)\frac{\|\nabla f(x_0)\|^2}{\alpha^2}$. For any $t \geq 0$ let
$$x_{t+1} = \argmin_{x \in \left\{(1-\lambda) c_t + \lambda x_t^+, \ \lambda \in \R \right\}} f(x) ,$$
and $c_{t+1}$ (respectively $R^2_{t+1}$) be the center (respectively the squared radius) of the ball given by (the proof of) Lemma \ref{lem:geom} which contains
$$\mB\left(c_t, R_t^2 - \frac{\|\nabla f(x_{t+1})\|^2}{\alpha^2 \kappa}\right) \cap \mB\left(x_{t+1}^{++}, \frac{\|\nabla f(x_{t+1})\|^2}{\alpha^2} \left(1 - \frac{1}{\kappa}\right) \right).$$
Formulas for $c_{t+1}$ and $R^2_{t+1}$ are given at the end of this section.

\begin{theorem}\label{thm:main}
For any $t \geq 0$, one has $x^* \in \mB(c_t, R_t^2)$, $R_{t+1}^2 \leq \left(1 - \frac{1}{\sqrt{\kappa}}\right) R_t^2$, and thus
$$\|x^* - c_t\|^2 \leq \left(1 - \frac{1}{\sqrt{\kappa}}\right)^t R_0^2 .$$
\end{theorem}

\begin{proof} 
We will prove a stronger claim by induction that for each $t\geq 0$, one has
$$x^* \in \mB\left(c_t, R_t^2 - \frac{2}{\alpha} \left(f(x_t^+) - f(x^*)\right)\right) .$$
The case $t=0$ follows immediately by \eqref{eq:ball2}. Let us assume that the above display is true for some $t \geq 0$. Then using $f(x_{t+1}^+) \leq f(x_{t+1}) - \frac{1}{2\beta} \|\nabla f(x_{t+1})\|^2 \leq f(x_t^+) - \frac{1}{2\beta} \|\nabla f(x_{t+1})\|^2 ,$
one gets
$$x^* \in \mB\left(c_t, R_t^2 - \frac{\|\nabla f(x_{t+1})\|^2}{\alpha^2 \kappa} - \frac{2}{\alpha} \left(f(x_{t+1}^+) - f(x^*)\right) \right) .$$
Furthermore by \eqref{eq:ball2} one also has
$$\mB\left(x_{t+1}^{++}, \frac{\|\nabla f(x_{t+1})\|^2}{\alpha^2} \left(1 - \frac{1}{\kappa}\right) - \frac{2}{\alpha} \left(f(x_{t+1}^+) - f(x^*)\right) \right).$$
Thus it only remains to observe that the squared radius of the ball given by Lemma \ref{lem:geom} which encloses the intersection of the two above balls is smaller than $\left(1 - \frac{1}{\sqrt{\kappa}}\right) R_t^2 - \frac{2}{\alpha} (f(x_{t+1}^+) - f(x^*))$.
We apply Lemma~\ref{lem:geom} after moving $c_t$ to the origin and scaling distances by $R_t$. We set $\epsilon =\frac{1}{\kappa}$, $g=\frac{\|\nabla f(x_{t+1})\|}{\alpha}$, $\delta=\frac{2}{\alpha}\left(f(x_{t+1}^+)-f(x^*)\right)$ and $a={x_{t+1}^{++}-c_t}$.  The line search step of the algorithm implies that $\nabla f(x_{t+1})^{\top} (x_{t+1} - c_t) = 0$ and therefore, $\|a\|=\|x_{t+1}^{++} - c_t\| \geq \|\nabla f(x_{t+1})\|/\alpha=g$ and Lemma~\ref{lem:geom} applies to give the result.
\end{proof}

One can use the following formulas for $c_{t+1}$ and $R^2_{t+1}$ (they are derived from the proof of Lemma \ref{lem:geom}). If $|\nabla f(x_{t+1})|^2 / \alpha^2 < R_t^2 / 2$ then one can tate $c_{t+1} = x_{t+1}^{++}$ and $R_{t+1}^2 = \frac{|\nabla f(x_{t+1})|^2}{\alpha^2} \left(1 - \frac{1}{\kappa}\right)$. On the other hand if $|\nabla f(x_{t+1})|^2 / \alpha^2 \geq R_t^2 / 2$ then one can tate
\begin{eqnarray*}
c_{t+1} & = & c_t + \frac{R_t^2 + |x_{t+1} - c_t|^2}{2 |x_{t+1}^{++} - c_t|^2} (x_{t+1}^{++} - c_t) , \\
R_{t+1}^2 & = & R_t^2 - \frac{|\nabla f(x_{t+1})|^2}{\alpha^2 \kappa} - \left( \frac{R_t^2 + \|x_{t+1} - c_t\|^2}{2 \|x_{t+1}^{++} - c_t\|}  \right)^2.
\end{eqnarray*}

\section{Nesterov's accelerated gradient descent} \label{sec:AGD}


We describe here the original Nesterov's method which attains the optimal oracle complexity for smooth convex optimization. We give the details of the method both for the strongly convex and non-strongly convex case. We refer to \cite{SBC14} for a recent interpretation of the method in terms of differential equations, and to \cite{AO14} for its relation to mirror descent (see Chapter \ref{mirror}).


\subsection{The smooth and strongly convex case}

Nesterov's accelerated gradient descent, illustrated in Figure \ref{fig:nesterovacc}, can be described as follows: Start at an arbitrary initial point 
$x_1 = y_1$ and then iterate the following equations for $t \geq 1$,
\begin{eqnarray*}
y_{t+1} & = & x_t  - \frac{1}{\beta} \nabla f(x_t) , \\
x_{t+1} & = & \left(1 + \frac{\sqrt{\kappa}-1}{\sqrt{\kappa}+1} \right) y_{t+1} - \frac{\sqrt{\kappa}-1}{\sqrt{\kappa}+1} y_t .
\end{eqnarray*}

\begin{figure}
\begin{center}
\begin{tikzpicture}[scale=1]
\node [tokens=1] (noeud1) at (0.5,1) [label=below right:{$x_s$}] {};
\node [tokens=1] (noeud2) at (1.5,-1) [label=below right:{$y_{s}$}] {};
\node [tokens=1] (noeud3) at (2.5,2) [label=below right:{$y_{s+1}$}] {};
\node [tokens=1] (noeud4) at (2.8,3) [label=above left:{$x_{s+1}$}] {};
\draw[->, thick] (noeud1) -- (noeud3) node[midway, left] {$-\frac{1}{\beta}\nabla f(x_s)$};
\draw[thick, dashed] (noeud2) -- (noeud4) {};
\node [tokens=1] (noeud5) at (4.5,3.3) [label=below right:{$y_{s+2}$}] {};
\node [tokens=1] (noeud6) at (5.3,3.8) [label=right:{$x_{s+2}$}] {};
\draw[->, thick] (noeud4) -- (noeud5) {};
\draw[thick, dashed] (noeud3) -- (noeud6) {};
\end{tikzpicture}
\end{center}
\caption{Illustration of Nesterov's accelerated gradient descent.}
\label{fig:nesterovacc}
\end{figure}
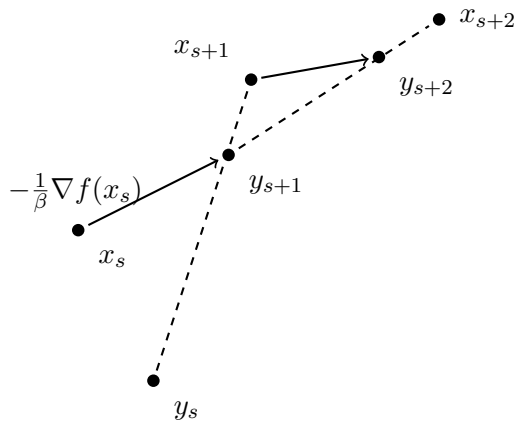

\begin{theorem}
Let $f$ be $\alpha$-strongly convex and $\beta$-smooth, then Nesterov's accelerated gradient descent satisfies
$$f(y_t) - f(x^*) \leq \frac{\alpha + \beta}{2} \|x_1 - x^*\|^2 \exp\left(- \frac{t-1}{\sqrt{\kappa}} \right).$$
\end{theorem}

\begin{proof}
We define $\alpha$-strongly convex quadratic functions $\Phi_s, s \geq 1$ by induction as follows:
\begin{align}
& \Phi_1(x) = f(x_1) + \frac{\alpha}{2} \|x-x_1\|^2 , \notag \\
& \Phi_{s+1}(x) = \left(1 - \frac{1}{\sqrt{\kappa}}\right) \Phi_s(x) \notag \\
& \qquad + \frac{1}{\sqrt{\kappa}} \left(f(x_s) + \nabla f(x_s)^{\top} (x-x_s) + \frac{\alpha}{2} \|x-x_s\|^2 \right). \label{eq:AGD0}
\end{align}
Intuitively $\Phi_s$ becomes a finer and finer approximation (from below) to $f$ in the following sense:
\begin{equation} \label{eq:AGD1}
\Phi_{s+1}(x) \leq f(x) + \left(1 - \frac{1}{\sqrt{\kappa}}\right)^s (\Phi_1(x) - f(x)). 
\end{equation}
The above inequality can be proved immediately by induction, using the fact that by $\alpha$-strong convexity one has
$$f(x_s) + \nabla f(x_s)^{\top} (x-x_s) + \frac{\alpha}{2} \|x-x_s\|^2 \leq f(x) .$$
Equation \eqref{eq:AGD1} by itself does not say much, for it to be useful one needs to understand how ``far" below $f$ is $\Phi_s$. The following inequality answers this question:
\begin{equation} \label{eq:AGD2}
f(y_s) \leq \min_{x \in \mathbb{R}^n} \Phi_s(x) . 
\end{equation}
The rest of the proof is devoted to showing that \eqref{eq:AGD2} holds true, but first let us see how to combine \eqref{eq:AGD1} and \eqref{eq:AGD2} to obtain the rate given by the theorem (we use that by $\beta$-smoothness one has $f(x) - f(x^*) \leq \frac{\beta}{2} \|x-x^*\|^2$):
\begin{eqnarray*}
f(y_t) - f(x^*) & \leq & \Phi_t(x^*) - f(x^*) \\
& \leq & \left(1 - \frac{1}{\sqrt{\kappa}}\right)^{t-1} (\Phi_1(x^*) - f(x^*)) \\
& \leq & \frac{\alpha + \beta}{2} \|x_1-x^*\|^2 \left(1 - \frac{1}{\sqrt{\kappa}}\right)^{t-1} .
\end{eqnarray*}
We now prove \eqref{eq:AGD2} by induction (note that it is true at $s=1$ since $x_1=y_1$). Let $\Phi_s^* = \min_{x \in \mathbb{R}^n} \Phi_s(x)$. Using the definition of $y_{s+1}$ (and $\beta$-smoothness), convexity, and the induction hypothesis, one gets
\begin{eqnarray*}
f(y_{s+1}) & \leq & f(x_s) - \frac{1}{2 \beta} \| \nabla f(x_s) \|^2 \\
& = & \left(1 - \frac{1}{\sqrt{\kappa}}\right) f(y_s) + \left(1 - \frac{1}{\sqrt{\kappa}}\right)(f(x_s) - f(y_s)) \\
& & + \frac{1}{\sqrt{\kappa}} f(x_s) - \frac{1}{2 \beta} \| \nabla f(x_s) \|^2 \\
& \leq & \left(1 - \frac{1}{\sqrt{\kappa}}\right) \Phi_s^* + \left(1 - \frac{1}{\sqrt{\kappa}}\right) \nabla f(x_s)^{\top} (x_s - y_s) \\
& & + \frac{1}{\sqrt{\kappa}} f(x_s) - \frac{1}{2 \beta} \| \nabla f(x_s) \|^2 .
\end{eqnarray*}
Thus we now have to show that
\begin{eqnarray} 
\Phi_{s+1}^* & \geq & \left(1 - \frac{1}{\sqrt{\kappa}}\right) \Phi_s^* + \left(1 - \frac{1}{\sqrt{\kappa}}\right) \nabla f(x_s)^{\top} (x_s - y_s) \notag \\
& & + \frac{1}{\sqrt{\kappa}} f(x_s) - \frac{1}{2 \beta} \| \nabla f(x_s) \|^2 . \label{eq:AGD3}
\end{eqnarray}
To prove this inequality we have to understand better the functions $\Phi_s$. First note that $\nabla^2 \Phi_s(x) = \alpha \mathrm{I}_n$ (immediate by induction) and thus $\Phi_s$ has to be of the following form:
$$\Phi_s(x) = \Phi_s^* + \frac{\alpha}{2} \|x - v_s\|^2 ,$$
for some $v_s \in \mathbb{R}^n$. Now observe that by differentiating \eqref{eq:AGD0} and using the above form of $\Phi_s$ one obtains
$$\nabla \Phi_{s+1}(x) = \alpha \left(1 - \frac{1}{\sqrt{\kappa}}\right) (x-v_s) + \frac{1}{\sqrt{\kappa}} \nabla f(x_s) + \frac{\alpha}{\sqrt{\kappa}} (x-x_s) .$$
In particular $\Phi_{s+1}$ is by definition minimized at $v_{s+1}$ which can now be defined by induction using the above identity, precisely:
\begin{equation} \label{eq:AGD4}
v_{s+1} = \left(1 - \frac{1}{\sqrt{\kappa}}\right) v_s + \frac{1}{\sqrt{\kappa}} x_s - \frac{1}{\alpha \sqrt{\kappa}} \nabla f(x_s) .
\end{equation}
Using the form of $\Phi_s$ and $\Phi_{s+1}$, as well as the original definition \eqref{eq:AGD0} one gets the following identity by evaluating $\Phi_{s+1}$ at $x_s$:
\begin{align} 
& \Phi_{s+1}^* + \frac{\alpha}{2} \|x_s - v_{s+1}\|^2 \notag \\
& = \left(1 - \frac{1}{\sqrt{\kappa}}\right) \Phi_s^* + \frac{\alpha}{2} \left(1 - \frac{1}{\sqrt{\kappa}}\right) \|x_s - v_s\|^2 + \frac{1}{\sqrt{\kappa}} f(x_s) . \label{eq:AGD5}
\end{align}
Note that thanks to \eqref{eq:AGD4} one has
\begin{eqnarray*}
\|x_s - v_{s+1}\|^2 & = & \left(1 - \frac{1}{\sqrt{\kappa}}\right)^2 \|x_s - v_s\|^2 + \frac{1}{\alpha^2 \kappa} \|\nabla f(x_s)\|^2 \\
& & - \frac{2}{\alpha \sqrt{\kappa}} \left(1 - \frac{1}{\sqrt{\kappa}}\right) \nabla f(x_s)^{\top}(v_s-x_s) ,
\end{eqnarray*}
which combined with \eqref{eq:AGD5} yields
\begin{eqnarray*}
\Phi_{s+1}^* & = & \left(1 - \frac{1}{\sqrt{\kappa}}\right) \Phi_s^* + \frac{1}{\sqrt{\kappa}} f(x_s) + \frac{\alpha}{2 \sqrt{\kappa}} \left(1 - \frac{1}{\sqrt{\kappa}}\right) \|x_s - v_s\|^2 \\
& & \qquad - \frac{1}{2 \beta} \| \nabla f(x_s) \|^2 + \frac{1}{\sqrt{\kappa}} \left(1 - \frac{1}{\sqrt{\kappa}}\right) \nabla f(x_s)^{\top}(v_s-x_s) .
\end{eqnarray*}
Finally we show by induction that $v_s - x_s = \sqrt{\kappa}(x_s - y_s)$, which concludes the proof of \eqref{eq:AGD3} and thus also concludes the proof of the theorem:
\begin{eqnarray*}
v_{s+1} - x_{s+1} & = & \left(1 - \frac{1}{\sqrt{\kappa}}\right) v_s + \frac{1}{\sqrt{\kappa}} x_s - \frac{1}{\alpha \sqrt{\kappa}} \nabla f(x_s) - x_{s+1} \\
& = & \sqrt{\kappa} x_s - (\sqrt{\kappa}-1) y_s - \frac{\sqrt{\kappa}}{\beta} \nabla f(x_s) - x_{s+1} \\
& = & \sqrt{\kappa} y_{s+1} - (\sqrt{\kappa}-1) y_s - x_{s+1} \\
& = & \sqrt{\kappa} (x_{s+1} - y_{s+1}) ,
\end{eqnarray*}
where the first equality comes from \eqref{eq:AGD4}, the second from the induction hypothesis, the third from the definition of $y_{s+1}$ and the last one from the definition of $x_{s+1}$.
\end{proof}

\subsection{The smooth case}
In this section we show how to adapt Nesterov's accelerated gradient descent for the case $\alpha=0$, using a time-varying combination of the elements in the primary sequence $(y_t)$. First we define the following sequences:
$$\lambda_0 = 0, \ \lambda_{t} = \frac{1 + \sqrt{1+ 4 \lambda_{t-1}^2}}{2}, \ \text{and} \  \gamma_t = \frac{1 - \lambda_t}{\lambda_{t+1}}.$$
(Note that $\gamma_t \leq 0$.) Now the algorithm is simply defined by the following equations, with $x_1 = y_1$ an arbitrary initial point,
\begin{eqnarray*}
y_{t+1} & = & x_t  - \frac{1}{\beta} \nabla f(x_t) , \\
x_{t+1} & = & (1 - \gamma_s) y_{t+1} + \gamma_t y_t .
\end{eqnarray*}

\begin{theorem}
Let $f$ be a convex and $\beta$-smooth function, then Nesterov's accelerated gradient descent satisfies
$$f(y_t) - f(x^*) \leq \frac{2 \beta \|x_1 - x^*\|^2}{t^2} .$$
\end{theorem}

We follow here the proof of \cite{BT09}. We also refer to \cite{Tse08} for a proof with simpler step-sizes.

\begin{proof}
Using the unconstrained version of Lemma \ref{lem:smoothconst} one obtains
\begin{align}
& f(y_{s+1}) - f(y_s) \notag \\
& \leq \nabla f(x_s)^{\top} (x_s-y_s) - \frac{1}{2 \beta} \| \nabla f(x_s) \|^2 \notag \\
& = \beta (x_s - y_{s+1})^{\top} (x_s-y_s) - \frac{\beta}{2} \| x_s - y_{s+1} \|^2 . \label{eq:1}
\end{align}
Similarly we also get
\begin{equation} \label{eq:2}
f(y_{s+1}) - f(x^*) \leq \beta (x_s - y_{s+1})^{\top} (x_s-x^*) - \frac{\beta}{2} \| x_s - y_{s+1} \|^2 .
\end{equation}
Now multiplying \eqref{eq:1} by $(\lambda_{s}-1)$ and adding the result to \eqref{eq:2}, one obtains with $\delta_s = f(y_s) - f(x^*)$,
\begin{align*}
& \lambda_{s} \delta_{s+1} - (\lambda_{s} - 1) \delta_s \\
& \leq \beta (x_s - y_{s+1})^{\top} (\lambda_{s} x_{s} - (\lambda_{s} - 1) y_s-x^*) - \frac{\beta}{2} \lambda_{s} \| x_s - y_{s+1} \|^2.
\end{align*}
Multiplying this inequality by $\lambda_{s}$ and using that by definition $\lambda_{s-1}^2 = \lambda_{s}^2 - \lambda_{s}$, as well as the elementary identity $2 a^{\top} b -  \|a\|^2 = \|b\|^2 - \|b-a\|^2$, one obtains
\begin{align}
& \lambda_{s}^2 \delta_{s+1} - \lambda_{s-1}^2 \delta_s \notag \\
& \leq \frac{\beta}{2} \bigg( 2 \lambda_{s} (x_s - y_{s+1})^{\top} (\lambda_{s} x_{s} - (\lambda_{s} - 1) y_s-x^*) - \|\lambda_{s}( y_{s+1} - x_s  )\|^2\bigg) \notag \\
& = \frac{\beta}{2} \bigg(\| \lambda_{s} x_{s} - (\lambda_{s} - 1) y_{s}-x^* \|^2 - \| \lambda_{s} y_{s+1} - (\lambda_{s} - 1) y_{s}-x^* \|^2 \bigg). \label{eq:3}
\end{align}
Next remark that, by definition, one has 
\begin{align}
& x_{s+1} = y_{s+1} + \gamma_s (y_s - y_{s+1}) \notag \\
& \Leftrightarrow \lambda_{s+1} x_{s+1} = \lambda_{s+1} y_{s+1} + (1-\lambda_{s})(y_s - y_{s+1}) \notag \\
& \Leftrightarrow \lambda_{s+1} x_{s+1} - (\lambda_{s+1} - 1) y_{s+1}= \lambda_{s} y_{s+1} - (\lambda_{s}-1) y_{s} . \label{eq:5}
\end{align}
Putting together \eqref{eq:3} and \eqref{eq:5} one gets with $u_s = \lambda_{s} x_{s} - (\lambda_{s} - 1) y_{s} - x^*$,
$$\lambda_{s}^2 \delta_{s+1} - \lambda_{s-1}^2 \delta_s^2 \leq \frac{\beta}{2} \bigg(\|u_s\|^2 - \|u_{s+1}\|^2 \bigg) .$$
Summing these inequalities from $s=1$ to $s=t-1$ one obtains:
$$\delta_t \leq \frac{\beta}{2 \lambda_{t-1}^2} \|u_1\|^2.$$
By induction it is easy to see that $\lambda_{t-1} \geq \frac{t}{2}$ which concludes the proof.
\end{proof}

\chapter{Almost dimension-free convex optimization in non-Euclidean spaces}
\label{mirror}
In the previous chapter we showed that dimension-free oracle complexity is possible when the objective function $f$ and the constraint set $\cX$ are well-behaved in the Euclidean norm; e.g. if for all points $x \in \cX$ and all subgradients $g \in \partial f(x)$, one has that $\|x\|_2$ and $\|g\|_2$ are independent of the ambient dimension $n$. If this assumption is not met then the gradient descent techniques of Chapter \ref{dimfree} may lose their dimension-free convergence rates. For instance consider a differentiable convex function $f$ defined on the Euclidean ball $\mB_{2,n}$ and such that $\|\nabla f(x)\|_{\infty} \leq 1, \forall x \in \mB_{2,n}$. This implies that $\|\nabla f(x)\|_{2} \leq \sqrt{n}$, and thus projected gradient descent will converge to the minimum of $f$ on $\mB_{2,n}$ at a rate $\sqrt{n / t}$. In this chapter we describe the method of \cite{NY83}, known as mirror descent, which allows to find the minimum of such functions $f$ over the $\ell_1$-ball (instead of the Euclidean ball) at the much faster rate $\sqrt{\log(n) / t}$. This is only one example of the potential of mirror descent. This chapter is devoted to the description of mirror descent and some of its alternatives. The presentation is inspired from \cite{BT03}, [Chapter 11, \cite{CL06}], \cite{Rak09, Haz11, Bub11}.
\newpage

In order to describe the intuition behind the method let us abstract the situation for a moment and forget that we are doing optimization in finite dimension. We already observed that projected gradient descent works in an arbitrary Hilbert space $\mathcal{H}$. Suppose now that we are interested in the more general situation of optimization in some Banach space $\mathcal{B}$. In other words the norm that we use to measure the various quantity of interest does not derive from an inner product (think of $\mathcal{B} = \ell_1$ for example). In that case the gradient descent strategy does not even make sense: indeed the gradients (more formally the Fr\'echet derivative) $\nabla f(x)$ are elements of the dual space $\mathcal{B}^*$ and thus one cannot perform the computation $x - \eta \nabla f(x)$ (it simply does not make sense). We did not have this problem for optimization in a Hilbert space $\mathcal{H}$ since by Riesz representation theorem $\mathcal{H}^*$ is isometric to $\mathcal{H}$. The great insight of Nemirovski and Yudin is that one can still do a gradient descent by first mapping the point $x \in \mathcal{B}$ into the dual space $\mathcal{B}^*$, then performing the gradient update in the dual space, and finally mapping back the resulting point to the primal space $\mathcal{B}$. Of course the new point in the primal space might lie outside of the constraint set $\mathcal{X} \subset \mathcal{B}$ and thus we need a way to project back the point on the constraint set $\mathcal{X}$. Both the primal/dual mapping and the projection are based on the concept of a {\em mirror map} which is the key element of the scheme. Mirror maps are defined in Section \ref{sec:mm}, and the above scheme is formally described in Section \ref{sec:MD}.

In the rest of this chapter we fix an arbitrary norm $\|\cdot\|$ on $\R^n$, and a compact convex set $\cX \subset \R^n$. The dual norm $\|\cdot\|_*$ is defined as $\|g\|_* = \sup_{x \in \mathbb{R}^n : \|x\| \leq 1} g^{\top} x$. We say that a convex function $f : \cX \rightarrow \R$ is (i) $L$-Lipschitz w.r.t. $\|\cdot\|$ if $\forall x \in \cX, g \in \partial f(x), \|g\|_* \leq L$, (ii) $\beta$-smooth w.r.t. $\|\cdot\|$ if $\|\nabla f(x) - \nabla f(y) \|_* \leq \beta \|x-y\|, \forall x, y \in \cX$, and (iii) $\alpha$-strongly convex w.r.t. $\|\cdot\|$ if 
$$f(x) - f(y) \leq g^{\top} (x - y) - \frac{\alpha}{2} \|x - y \|^2 , \forall x, y \in \cX, g \in \partial f(x).$$ We also define the Bregman divergence associated to $f$ as 
$$D_{f}(x,y) = f(x) - f(y) - \nabla f(y)^{\top} (x - y) .$$
The following identity will be useful several times:
\begin{equation} \label{eq:useful1}
(\nabla f(x) - \nabla f(y))^{\top}(x-z) = D_{f}(x,y) + D_{f}(z,x) - D_{f}(z,y) .
\end{equation}

\section{Mirror maps} \label{sec:mm}
Let $\cD \subset \R^n$ be a convex open set such that $\mathcal{X}$ is included in its closure, that is $\mathcal{X} \subset \overline{\mathcal{D}}$, and $\mathcal{X} \cap \mathcal{D} \neq \emptyset$. We say that $\Phi : \cD \rightarrow \R$ is a mirror map if it safisfies the following properties\footnote{Assumption (ii) can be relaxed in some cases, see for example \cite{ABL14}.}:
\begin{enumerate}
\item[(i)] $\Phi$ is strictly convex and differentiable.
\item[(ii)] The gradient of $\Phi$ takes all possible values, that is $\nabla \Phi(\cD) = \R^n$.
\item[(iii)] The gradient of $\Phi$ diverges on the boundary of $\cD$, that is 
$$\lim_{x \rightarrow \partial \mathcal{D}} \|\nabla \Phi(x)\| = + \infty .$$
\end{enumerate}

In mirror descent the gradient of the mirror map $\Phi$ is used to map points from the ``primal" to the ``dual" (note that all points lie in $\R^n$ so the notions of primal and dual spaces only have an intuitive meaning). Precisely a point $x \in \mathcal{X} \cap \mathcal{D}$ is mapped to $\nabla \Phi(x)$, from which one takes a gradient step to get to $\nabla \Phi(x) - \eta \nabla f(x)$. Property (ii) then allows us to write the resulting point as $\nabla \Phi(y) = \nabla \Phi(x) - \eta \nabla f(x)$ for some $y \in \cD$. The primal point $y$ may lie outside of the set of constraints $\cX$, in which case one has to project back onto $\cX$. In mirror descent this projection is done via the Bregman divergence associated to $\Phi$. Precisely one defines
$$\Pi_{\cX}^{\Phi} (y) = \argmin_{x \in \mathcal{X} \cap \mathcal{D}} D_{\Phi}(x,y) .$$
Property (i) and (iii) ensures the existence and uniqueness of this projection (in particular since $x \mapsto D_{\Phi}(x,y)$ is locally increasing on the boundary of $\mathcal{D}$). The following lemma shows that the Bregman divergence essentially behaves as the Euclidean norm squared in terms of projections (recall Lemma \ref{lem:todonow}).

\begin{lemma} \label{lem:todonow2}
Let $x \in \cX \cap \cD$ and $y \in \cD$, then
$$(\nabla \Phi(\Pi_{\cX}^{\Phi}(y)) - \nabla \Phi(y))^{\top} (\Pi^{\Phi}_{\cX}(y) - x) \leq 0 ,$$
which also implies 
$$D_{\Phi}(x, \Pi^{\Phi}_{\cX}(y)) + D_{\Phi}(\Pi^{\Phi}_{\cX}(y), y) \leq D_{\Phi}(x,y) .$$
\end{lemma}

\begin{proof}
The proof is an immediate corollary of Proposition \ref{prop:firstorder} together with the fact that $\nabla_x D_{\Phi}(x,y) = \nabla \Phi(x) - \nabla \Phi(y)$.
\end{proof}

\section{Mirror descent} \label{sec:MD}
We can now describe the mirror descent strategy based on a mirror map $\Phi$. Let $x_1 \in \argmin_{x \in \mathcal{X} \cap \mathcal{D}} \Phi(x)$. Then for $t \geq 1$, let $y_{t+1} \in \mathcal{D}$ such that
\begin{equation} \label{eq:MD1}
\nabla \Phi(y_{t+1}) = \nabla \Phi(x_{t}) - \eta g_t, \ \text{where} \ g_t \in \partial f(x_t) ,
\end{equation}
and
\begin{equation} \label{eq:MD2}
x_{t+1} \in \Pi_{\cX}^{\Phi} (y_{t+1}) .
\end{equation}
See Figure \ref{fig:MD} for an illustration of this procedure.
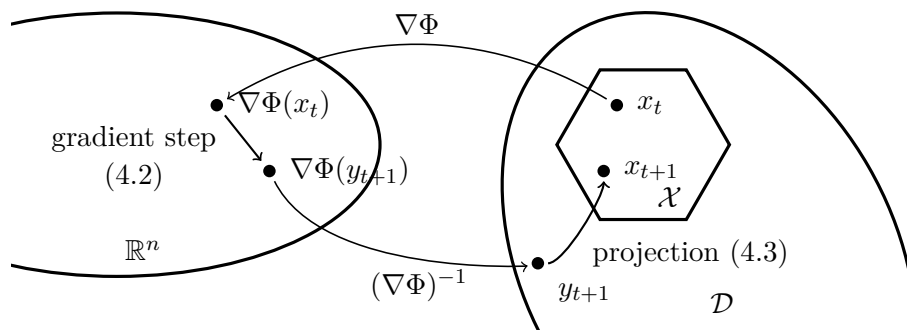
\begin{figure}
\begin{tikzpicture}[scale=3.5]
\clip (-2.4,-0.7) rectangle (1,1);
\draw[rotate=30, very thick] (0,-0.5) ellipse (0.7 and 1);
\draw[very thick] (-2,0) ellipse (1 and 0.5);
\node (S) [very thick, regular polygon, regular polygon sides=6, draw,
inner sep=20] at (0,0) {};
\node at (0.3,-0.6) {$\cD$};
\node at (-1.9,-0.4) {$\R^n$};
\node at (0.1, -0.2) {$\cX$};
\node [tokens=1] (noeudat) at (-0.1,0.15) [label=right:{$x_t$}] {};
\node [tokens=1] (noeudat1) at (-0.15,-0.1) [label=right:{$x_{t+1}$}] {};
\node [tokens=1] (noeudwt1) at (-0.4,-0.45) [label=below right:{$y_{t+1}$}] {};
\draw[->, thick] (noeudwt1) .. controls (-0.3, -0.45) and (-0.15, -0.2) .. (noeudat1) node[midway, below right] {projection \eqref{eq:MD2}};
\node [tokens=1] (noeudmat) at (-1.62,0.15) [label=right:{$\nabla \Phi(x_t)$}] {};
\node [tokens=1] (noeudmwt1) at (-1.42,-0.1) [label=right:{$\nabla \Phi(y_{t+1})$}] {};
\draw[->, thick] (noeudmat) -- (noeudmwt1) node[midway, left] {\begin{tabular}{c} \\ gradient step \\ \eqref{eq:MD1} \end{tabular}};
\draw[->, semithick] (noeudat) .. controls (-0.6,0.45) and (-1.12, 0.45) .. (noeudmat) node[midway, above] {$\nabla \Phi$}; 
\draw[->, semithick] (noeudmwt1) .. controls (-1.22,-0.5) and (-0.44, -0.46) .. (noeudwt1) node[midway, below] {$(\nabla \Phi)^{-1}$}; 
\end{tikzpicture}
\caption{Illustration of mirror descent.}
\label{fig:MD}
\end{figure}

\begin{theorem} \label{th:MD}
Let $\Phi$ be a mirror map $\rho$-strongly convex on $\mathcal{X} \cap \mathcal{D}$ w.r.t. $\|\cdot\|$.
Let $R^2 = \sup_{x \in \mathcal{X} \cap \mathcal{D}} \Phi(x) - \Phi(x_1)$, and $f$ be convex and $L$-Lipschitz w.r.t. $\|\cdot\|$. Then mirror descent with $\eta = \frac{R}{L} \sqrt{\frac{2 \rho}{t}}$ satisfies
$$f\bigg(\frac{1}{t} \sum_{s=1}^t x_s \bigg) - f(x^*) \leq RL \sqrt{\frac{2}{\rho t}} .$$
\end{theorem}

\begin{proof}
Let $x \in \mathcal{X} \cap \mathcal{D}$. The claimed bound will be obtained by taking a limit $x \rightarrow x^*$. Now by convexity of $f$, the definition of mirror descent, equation \eqref{eq:useful1}, and Lemma \ref{lem:todonow2}, one has
\begin{align*}
& f(x_s) - f(x) \\
& \leq g_s^{\top} (x_s - x) \\
& = \frac{1}{\eta} (\nabla \Phi(x_s) - \nabla \Phi(y_{s+1}))^{\top} (x_s - x) \\
& = \frac{1}{\eta} \bigg( D_{\Phi}(x, x_s) + D_{\Phi}(x_s, y_{s+1}) - D_{\Phi}(x, y_{s+1}) \bigg) \\
& \leq \frac{1}{\eta} \bigg( D_{\Phi}(x, x_s) + D_{\Phi}(x_s, y_{s+1}) - D_{\Phi}(x, x_{s+1}) - D_{\Phi}(x_{s+1}, y_{s+1}) \bigg) .
\end{align*}
The term $D_{\Phi}(x, x_s) -  D_{\Phi}(x, x_{s+1})$ will lead to a telescopic sum when summing over $s=1$ to $s=t$, and it remains to bound the other term as follows using $\rho$-strong convexity of the mirror map and $a z - b z^2 \leq \frac{a^2}{4 b}, \forall z \in \R$:
\begin{align*}
& D_{\Phi}(x_s, y_{s+1}) - D_{\Phi}(x_{s+1}, y_{s+1}) \\
& = \Phi(x_s) - \Phi(x_{s+1}) - \nabla \Phi(y_{s+1})^{\top} (x_{s} - x_{s+1}) \\
& \leq (\nabla \Phi(x_s) - \nabla \Phi(y_{s+1}))^{\top} (x_{s} - x_{s+1}) - \frac{\rho}{2} \|x_s - x_{s+1}\|^2 \\
& = \eta g_s^{\top} (x_{s} - x_{s+1}) - \frac{\rho}{2} \|x_s - x_{s+1}\|^2 \\
& \leq \eta L \|x_{s} - x_{s+1}\| - \frac{\rho}{2} \|x_s - x_{s+1}\|^2 \\
& \leq \frac{(\eta L)^2}{2 \rho}.
\end{align*}
We proved 
$$\sum_{s=1}^t \bigg(f(x_s) - f(x)\bigg) \leq \frac{D_{\Phi}(x,x_1)}{\eta} + \eta \frac{L^2 t}{2 \rho},$$
which concludes the proof up to trivial computation.
\end{proof}

We observe that one can rewrite mirror descent as follows:
\begin{eqnarray}
x_{t+1} & = & \argmin_{x \in \mathcal{X} \cap \mathcal{D}} \ D_{\Phi}(x,y_{t+1}) \notag \\
& = & \argmin_{x \in \mathcal{X} \cap \mathcal{D}} \ \Phi(x) - \nabla \Phi(y_{t+1})^{\top} x \label{eq:MD3} \\
& = & \argmin_{x \in \mathcal{X} \cap \mathcal{D}} \ \Phi(x) - (\nabla \Phi(x_{t}) - \eta g_t)^{\top} x \notag \\
& = & \argmin_{x \in \mathcal{X} \cap \mathcal{D}} \ \eta g_t^{\top} x + D_{\Phi}(x,x_t) . \label{eq:MDproxview}
\end{eqnarray}
This last expression is often taken as the definition of mirror descent (see \cite{BT03}). It gives a proximal point of view on mirror descent: the method is trying to minimize the local linearization of the function while not moving too far away from the previous point, with distances measured via the Bregman divergence of the mirror map.

\section{Standard setups for mirror descent} \label{sec:mdsetups}
\noindent \textbf{``Ball setup".} The simplest version of mirror descent is obtained by taking $\Phi(x) = \frac{1}{2} \|x\|^2_2$ on $\mathcal{D} = \mathbb{R}^n$. The function $\Phi$ is a mirror map strongly convex w.r.t. $\|\cdot\|_2$, and furthermore the associated Bregman divergence is given by $D_{\Phi}(x,y) = \frac{1}{2} \|x - y\|^2_2$. Thus in that case mirror descent is exactly equivalent to projected subgradient descent, and the rate of convergence obtained in Theorem \ref{th:MD} recovers our earlier result on projected subgradient descent.
\newline

\noindent
\textbf{``Simplex setup".} A more interesting choice of a mirror map is given by the negative entropy
$$\Phi(x) = \sum_{i=1}^n x(i) \log x(i),$$
on $\mathcal{D} = \mathbb{R}_{++}^n$. In that case the gradient update $\nabla \Phi(y_{t+1}) = \nabla \Phi(x_t) - \eta \nabla f(x_t)$ can be written equivalently as
$$y_{t+1}(i) = x_{t}(i) \exp\big(- \eta [\nabla f(x_t) ](i) \big) , \ i=1, \hdots, n.$$
The Bregman divergence of this mirror map is given by $D_{\Phi}(x,y) = \sum_{i=1}^n x(i) \log \frac{x(i)}{y(i)}$ (also known as the Kullback-Leibler divergence). It is easy to verify that the projection with respect to this Bregman divergence on the simplex $\Delta_n = \{x \in \mathbb{R}_+^n : \sum_{i=1}^n x(i) = 1\}$ amounts to a simple renormalization $y \mapsto y / \|y\|_1$. Furthermore it is also easy to verify that $\Phi$ is $1$-strongly convex w.r.t. $\|\cdot\|_1$ on $\Delta_n$ (this result is known as Pinsker's inequality). Note also that for $\mathcal{X} = \Delta_n$ one has $x_1 = (1/n, \hdots, 1/n)$ and $R^2 = \log n$.

The above observations imply that when minimizing on the simplex $\Delta_n$ a function $f$ with subgradients bounded in $\ell_{\infty}$-norm, mirror descent with the negative entropy achieves a rate of convergence of order $\sqrt{\frac{\log n}{t}}$. On the other hand the regular subgradient descent achieves only a rate of order $\sqrt{\frac{n}{t}}$ in this case!
\newline

\noindent
\textbf{``Spectrahedron setup".} We consider here functions defined on matrices, and we are interested in minimizing a function $f$ on the {\em spectrahedron} $\mathcal{S}_n$ defined as:
$$\mathcal{S}_n = \left\{X \in \mathbb{S}_+^n : \mathrm{Tr}(X) = 1 \right\} .$$
In this setting we consider the mirror map on $\mathcal{D} = \mathbb{S}_{++}^n$ given by the negative von Neumann entropy:
$$\Phi(X) = \sum_{i=1}^n \lambda_i(X) \log \lambda_i(X) ,$$
where $\lambda_1(X), \hdots, \lambda_n(X)$ are the eigenvalues of $X$. It can be shown that the gradient update $\nabla \Phi(Y_{t+1}) = \nabla \Phi(X_t) - \eta \nabla f(X_t)$ can be written equivalently as
$$Y_{t+1} = \exp\big(\log X_t - \eta \nabla f(X_t) \big) ,$$
where the matrix exponential and matrix logarithm are defined as usual. Furthermore the projection on $\mathcal{S}_n$ is a simple trace renormalization.

With highly non-trivial computation one can show that $\Phi$ is $\frac{1}{2}$-strongly convex with respect to the Schatten $1$-norm defined as
$$\|X\|_1 = \sum_{i=1}^n \lambda_i(X).$$
It is easy to see that for $\mathcal{X} = \mathcal{S}_n$ one has $x_1 = \frac{1}{n} \mI_n$ and $R^2 = \log n$. In other words the rate of convergence for optimization on the spectrahedron is the same than on the simplex!

\section{Lazy mirror descent, aka Nesterov's dual averaging}
In this section we consider a slightly more efficient version of mirror descent for which we can prove that Theorem \ref{th:MD} still holds true. This alternative algorithm can be advantageous in some situations (such as distributed settings), but the basic mirror descent scheme remains important for extensions considered later in this text (saddle points, stochastic oracles, ...). 

In lazy mirror descent, also commonly known as Nesterov's dual averaging or simply dual averaging, one replaces \eqref{eq:MD1} by
$$\nabla \Phi(y_{t+1}) = \nabla \Phi(y_{t}) - \eta g_t ,$$
and also $y_1$ is such that $\nabla \Phi(y_1) = 0$. In other words instead of going back and forth between the primal and the dual, dual averaging simply averages the gradients in the dual, and if asked for a point in the primal it simply maps the current dual point to the primal using the same methodology as mirror descent. In particular using \eqref{eq:MD3} one immediately sees that dual averaging is defined by:
\begin{equation} \label{eq:DA0}
x_t = \argmin_{x \in \mathcal{X} \cap \mathcal{D}} \ \eta \sum_{s=1}^{t-1} g_s^{\top} x + \Phi(x) .
\end{equation}
\begin{theorem}
Let $\Phi$ be a mirror map $\rho$-strongly convex on $\mathcal{X} \cap \mathcal{D}$ w.r.t. $\|\cdot\|$.
Let $R^2 = \sup_{x \in \mathcal{X} \cap \mathcal{D}} \Phi(x) - \Phi(x_1)$, and $f$ be convex and $L$-Lipschitz w.r.t. $\|\cdot\|$. Then dual averaging with $\eta = \frac{R}{L} \sqrt{\frac{\rho}{2 t}}$ satisfies
$$f\bigg(\frac{1}{t} \sum_{s=1}^t x_s \bigg) - f(x^*) \leq 2 RL \sqrt{\frac{2}{\rho t}} .$$
\end{theorem}

\begin{proof}
We define $\psi_t(x) = \eta \sum_{s=1}^{t} g_s^{\top} x + \Phi(x)$, so that $x_t \in  \argmin_{x \in \mathcal{X} \cap \mathcal{D}} \psi_{t-1}(x)$. Since $\Phi$ is $\rho$-strongly convex one clearly has that $\psi_t$ is $\rho$-strongly convex, and thus
\begin{eqnarray*}
\psi_t(x_{t+1}) - \psi_t(x_t) & \leq & \nabla \psi_t(x_{t+1})^{\top}(x_{t+1} - x_{t}) - \frac{\rho}{2} \|x_{t+1} - x_t\|^2 \\
& \leq & - \frac{\rho}{2} \|x_{t+1} - x_t\|^2 ,
\end{eqnarray*}
where the second inequality comes from the first order optimality condition for $x_{t+1}$ (see Proposition \ref{prop:firstorder}). Next observe that
\begin{eqnarray*}
\psi_t(x_{t+1}) - \psi_t(x_t) & = & \psi_{t-1}(x_{t+1}) - \psi_{t-1}(x_t) + \eta g_t^{\top} (x_{t+1} - x_t) \\
& \geq & \eta g_t^{\top} (x_{t+1} - x_t) .
\end{eqnarray*}
Putting together the two above displays and using Cauchy-Schwarz (with the assumption $\|g_t\|_* \leq L$) one obtains
$$\frac{\rho}{2} \|x_{t+1} - x_t\|^2 \leq \eta g_t^{\top} (x_t - x_{t+1}) \leq \eta L \|x_t - x_{t+1} \|.$$
In particular this shows that $\|x_{t+1} - x_t\| \leq \frac{2 \eta L}{\rho}$ and thus with the above display
\begin{equation} \label{eq:DA1}
g_t^{\top} (x_t - x_{t+1}) \leq \frac{2 \eta L^2}{\rho} .
\end{equation}
Now we claim that for any $x \in \cX \cap \cD$,
\begin{equation} \label{eq:DA2}
\sum_{s=1}^t g_s^{\top} (x_s - x) \leq \sum_{s=1}^t g_s^{\top} (x_s - x_{s+1}) + \frac{\Phi(x) - \Phi(x_1)}{\eta} ,
\end{equation}
which would clearly conclude the proof thanks to \eqref{eq:DA1} and straightforward computations. Equation \eqref{eq:DA2} is equivalent to 
$$\sum_{s=1}^t g_s^{\top} x_{s+1} + \frac{\Phi(x_1)}{\eta} \leq \sum_{s=1}^t g_s^{\top} x + \frac{\Phi(x)}{\eta} ,$$
and we now prove the latter equation by induction. At $t=0$ it is true since $x_1 \in \argmin_{x \in \cX \cap \cD} \Phi(x)$. The following inequalities prove the inductive step, where we use the induction hypothesis at $x=x_{t+1}$ for the first inequality, and the definition of $x_{t+1}$ for the second inequality:
$$\sum_{s=1}^{t} g_s^{\top} x_{s+1} + \frac{\Phi(x_1)}{\eta} \leq g_{t}^{\top}x_{t+1} + \sum_{s=1}^{t-1} g_s^{\top} x_{t+1} + \frac{\Phi(x_{t+1})}{\eta} \leq \sum_{s=1}^{t} g_s^{\top} x + \frac{\Phi(x)}{\eta} .$$
\end{proof}

\section{Mirror prox}
It can be shown that mirror descent accelerates for smooth functions to the rate $1/t$. We will prove this result in Chapter \ref{rand} (see Theorem \ref{th:SMDsmooth}). We describe here a variant of mirror descent which also attains the rate $1/t$ for smooth functions. This method is called mirror prox and it was introduced in \cite{Nem04}. 
The true power of mirror prox will reveal itself later in the text when we deal with smooth representations of non-smooth functions as well as stochastic oracles\footnote{Basically mirror prox allows for a smooth vector field point of view (see Section \ref{sec:vectorfield}), while mirror descent does not.}.

Mirror prox is described by the following equations:
\begin{align*}
& \nabla \Phi(y_{t+1}') = \nabla \Phi(x_{t}) - \eta \nabla f(x_t), \\ \\
& y_{t+1} \in \argmin_{x \in \mathcal{X} \cap \mathcal{D}} D_{\Phi}(x,y_{t+1}') , \\  \\
& \nabla \Phi(x_{t+1}') = \nabla \Phi(x_{t}) - \eta \nabla f(y_{t+1}), \\ \\
& x_{t+1} \in \argmin_{x \in \mathcal{X} \cap \mathcal{D}} D_{\Phi}(x,x_{t+1}') .
\end{align*}
In words the algorithm first makes a step of mirror descent to go from $x_t$ to $y_{t+1}$, and then it makes a similar step to obtain $x_{t+1}$, starting again from $x_t$ but this time using the gradient of $f$ evaluated at $y_{t+1}$ (instead of $x_t$), see Figure \ref{fig:mp} for an illustration. The following result justifies the procedure.

\begin{figure}
\begin{tikzpicture}[scale=4]
\clip (-2.4,-0.7) rectangle (0.5,1);
\draw[rotate=30, very thick] (0,-0.5) ellipse (0.73 and 1);
\draw[very thick] (-2,0) ellipse (1 and 0.5);
\node (S) [very thick, regular polygon, regular polygon sides=6, draw,
inner sep=22] at (0,0) {};
\node at (0.3,-0.6) {$\cD$};
\node at (-2.2,-0.4) {$\R^n$};
\node at (0.12, -0.22) {$\cX$};
\node [tokens=1] (noeudat) at (-0.1,0.15) [label=right:{$x_t$}] {};
\node [tokens=1] (noeudat1) at (-0.15,-0.13) [label=right:{$y_{t+1}$}] {};
\node [tokens=1] (noeudwt1) at (-0.4,-0.45) [label=below right:{$y_{t+1}'$}] {};
\draw[->, thick] (noeudwt1) .. controls (-0.3, -0.45) and (-0.15, -0.2) .. (noeudat1) node[midway, below right] {projection};
\node [tokens=1] (noeudat3) at (-0.2,0) [label=right:{$x_{t+1}$}] {};
\node [tokens=1] (noeudwt3) at (-0.3,-0.2) [label=left:{$x_{t+1}'$}] {};
\draw[->, thick] (noeudwt3) .. controls (-0.22, -0.12) and (-0.22, -0.1) .. (noeudat3) {};
\node [tokens=1] (noeudmat) at (-1.7,0.3) [label=right:{$\nabla \Phi(x_t)$}] {};
\node [tokens=1] (noeudmwt1) at (-2,-0.2) [label=below right:{$\nabla \Phi(y_{t+1}')$}] {};
\node [tokens=1] (noeudmwt2) at (-1.6,-0.1) [label=right:{$\nabla \Phi(x_{t+1}')$}] {};
\draw[->, thick] (noeudmat) -- (noeudmwt2) node[midway, right] {$- \eta \nabla f(y_{t+1})$};
\draw[->, thick, dashed] (noeudmat) -- (noeudmwt1) node[midway, left] {$- \eta \nabla f(x_t)$};
\draw[->, semithick] (-0.4,0.5) .. controls (-0.7,0.6) and (-1, 0.6) .. (-1.3,0.5) node[midway, above] {$\nabla \Phi$}; 
\draw[->, semithick] (-1.45,-0.5) .. controls (-1.15,-0.6) and (-0.85, -0.6) .. (-0.55,-0.5) node[midway, below] {$(\nabla \Phi)^{-1}$}; 
\end{tikzpicture}
\caption{Illustration of mirror prox.}
\label{fig:mp}
\end{figure}
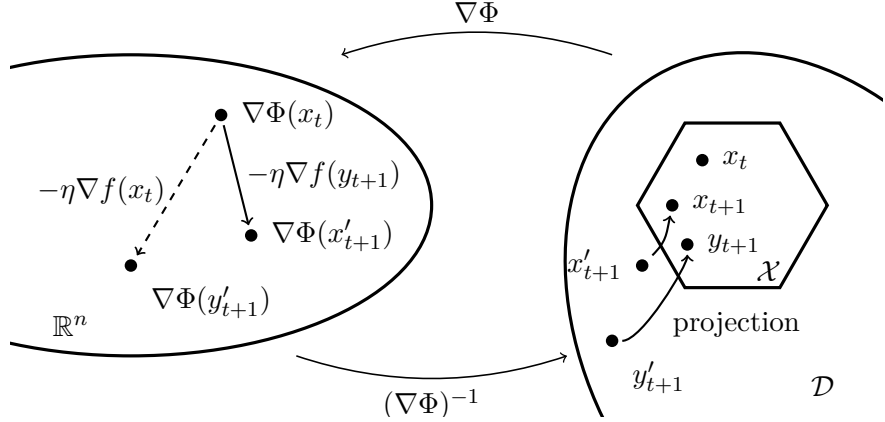

\begin{theorem}
Let $\Phi$ be a mirror map $\rho$-strongly convex on $\mathcal{X} \cap \mathcal{D}$ w.r.t. $\|\cdot\|$. Let $R^2 = \sup_{x \in \mathcal{X} \cap \mathcal{D}} \Phi(x) - \Phi(x_1)$, and $f$ be convex and $\beta$-smooth w.r.t. $\|\cdot\|$. Then mirror prox with $\eta = \frac{\rho}{\beta}$ satisfies
$$f\bigg(\frac{1}{t} \sum_{s=1}^t y_{s+1} \bigg) - f(x^*) \leq \frac{\beta R^2}{\rho t} .$$
\end{theorem}

\begin{proof}
Let $x \in \mathcal{X} \cap \mathcal{D}$. We write
\begin{eqnarray*}
f(y_{t+1}) - f(x) & \leq & \nabla f(y_{t+1})^{\top} (y_{t+1} - x) \\
& = & \nabla f(y_{t+1})^{\top} (x_{t+1} - x) + \nabla f(x_t)^{\top} (y_{t+1} - x_{t+1}) \\
& & + (\nabla f(y_{t+1}) - \nabla f(x_t))^{\top} (y_{t+1} - x_{t+1}) .
\end{eqnarray*}
We will now bound separately these three terms. For the first one, using the definition of the method, Lemma \ref{lem:todonow2}, and equation \eqref{eq:useful1}, one gets
\begin{align*}
& \eta \nabla f(y_{t+1})^{\top} (x_{t+1} - x) \\
& = ( \nabla \Phi(x_t) - \nabla \Phi(x_{t+1}'))^{\top} (x_{t+1} - x) \\
& \leq ( \nabla \Phi(x_t) - \nabla \Phi(x_{t+1}))^{\top} (x_{t+1} - x) \\
& = D_{\Phi}(x,x_t) - D_{\Phi}(x, x_{t+1}) - D_{\Phi}(x_{t+1}, x_t) .
\end{align*}
For the second term using the same properties than above and the strong-convexity of the mirror map one obtains
\begin{align}
& \eta \nabla f(x_t)^{\top} (y_{t+1} - x_{t+1}) \notag\\
& = ( \nabla \Phi(x_t) - \nabla \Phi(y_{t+1}'))^{\top} (y_{t+1} - x_{t+1}) \notag\\
& \leq ( \nabla \Phi(x_t) - \nabla \Phi(y_{t+1}))^{\top} (y_{t+1} - x_{t+1}) \notag\\
& = D_{\Phi}(x_{t+1},x_t) - D_{\Phi}(x_{t+1}, y_{t+1}) - D_{\Phi}(y_{t+1}, x_t) \label{eq:pourplustard1}\\
& \leq D_{\Phi}(x_{t+1},x_t) - \frac{\rho}{2} \|x_{t+1} - y_{t+1} \|^2 - \frac{\rho}{2} \|y_{t+1} - x_t\|^2 . \notag
\end{align}
Finally for the last term, using Cauchy-Schwarz, $\beta$-smoothness, and $2 ab \leq a^2 + b^2$ one gets
\begin{align*}
& (\nabla f(y_{t+1}) - \nabla f(x_t))^{\top} (y_{t+1} - x_{t+1}) \\
& \leq \|\nabla f(y_{t+1}) - \nabla f(x_t)\|_*  \cdot \|y_{t+1} - x_{t+1} \| \\
& \leq \beta \|y_{t+1} - x_t\| \cdot \|y_{t+1} - x_{t+1} \| \\
& \leq \frac{\beta}{2} \|y_{t+1} - x_t\|^2 + \frac{\beta}{2}  \|y_{t+1} - x_{t+1} \|^2 . 
\end{align*}
Thus summing up these three terms and using that $\eta = \frac{\rho}{\beta}$ one gets
$$f(y_{t+1}) - f(x) \leq \frac{D_{\Phi}(x,x_t) - D_{\Phi}(x,x_{t+1})}{\eta} .$$
The proof is concluded with straightforward computations.
\end{proof}

\section{The vector field point of view on MD, DA, and MP} \label{sec:vectorfield}
In this section we consider a mirror map $\Phi$ that satisfies the assumptions from Theorem \ref{th:MD}.

By inspecting the proof of Theorem \ref{th:MD} one can see that for arbitrary vectors $g_1, \hdots, g_t \in \R^n$ the mirror descent strategy described by \eqref{eq:MD1} or \eqref{eq:MD2} (or alternatively by \eqref{eq:MDproxview}) satisfies for any $x \in \cX \cap \cD$,
\begin{equation} \label{eq:vfMD}
\sum_{s=1}^t g_s^{\top} (x_s - x) \leq \frac{R^2}{\eta} + \frac{\eta}{2 \rho} \sum_{s=1}^t \|g_s\|_*^2 .
\end{equation}
The observation that the sequence of vectors $(g_s)$ does not have to come from the subgradients of a {\em fixed} function $f$ is the starting point for the theory of online learning, see \cite{Bub11} for more details. In this monograph we will use this observation to generalize mirror descent to saddle point calculations as well as stochastic settings. We note that we could also use dual averaging (defined by \eqref{eq:DA0}) which satisfies
$$\sum_{s=1}^t g_s^{\top} (x_s - x) \leq \frac{R^2}{\eta} + \frac{2 \eta}{\rho} \sum_{s=1}^t \|g_s\|_*^2 .$$
In order to generalize mirror prox we simply replace the gradient $\nabla f$ by an arbitrary vector field $g: \cX \rightarrow \R^n$ which yields the following equations:
\begin{align*}
& \nabla \Phi(y_{t+1}') = \nabla \Phi(x_{t}) - \eta g(x_t), \\
& y_{t+1} \in \argmin_{x \in \mathcal{X} \cap \mathcal{D}} D_{\Phi}(x,y_{t+1}') , \\ 
& \nabla \Phi(x_{t+1}') = \nabla \Phi(x_{t}) - \eta g(y_{t+1}), \\
& x_{t+1} \in \argmin_{x \in \mathcal{X} \cap \mathcal{D}} D_{\Phi}(x,x_{t+1}') .
\end{align*}
Under the assumption that the vector field is $\beta$-Lipschitz w.r.t. $\|\cdot\|$, i.e., $\|g(x) - g(y)\|_* \leq \beta \|x-y\|$ one obtains with $\eta = \frac{\rho}{\beta}$
\begin{equation} \label{eq:vfMP}
\sum_{s=1}^t g(y_{s+1})^{\top}(y_{s+1} - x) \leq \frac{\beta R^2}{\rho}.
\end{equation}

\chapter{Beyond the black-box model}
\label{beyond}
In the black-box model non-smoothness dramatically deteriorates the rate of convergence of first order methods from $1/t^2$ to $1/\sqrt{t}$. However, as we already pointed out in Section \ref{sec:structured}, we (almost) always know the function to be optimized {\em globally}. In particular the ``source" of non-smoothness can often be identified. For instance the LASSO objective (see Section \ref{sec:mlapps}) is non-smooth, but it is a sum of a smooth part (the least squares fit) and a {\em simple} non-smooth part (the $\ell_1$-norm). Using this specific structure we will propose in Section \ref{sec:simplenonsmooth} a first order method with a $1/t^2$ convergence rate, despite the non-smoothness. In Section \ref{sec:sprepresentation} we consider another type of non-smoothness that can effectively be overcome, where the function is the maximum of smooth functions. Finally we conclude this chapter with a concise description of interior point methods, for which the structural assumption is made on the constraint set rather than on the objective function.

\section{Sum of a smooth and a simple non-smooth term} \label{sec:simplenonsmooth}
We consider here the following problem\footnote{We restrict to unconstrained minimization for sake of simplicity. One can extend the discussion to constrained minimization by using ideas from Section \ref{sec:gdsmooth}.}:
$$\min_{x \in \R^n} f(x) + g(x) ,$$
where $f$ is convex and $\beta$-smooth, and $g$ is convex. We assume that $f$ can be accessed through a first order oracle, and that $g$ is known and ``simple". What we mean by simplicity will be clear from the description of the algorithm. For instance a separable function, that is $g(x) = \sum_{i=1}^n g_i(x(i))$, will be considered as simple. The prime example being $g(x) = \|x\|_1$. This section is inspired from \cite{BT09} (see also \cite{Nes07, WNF09}).

\section*{ISTA (Iterative Shrinkage-Thresholding Algorithm)}
Recall that gradient descent on the smooth function $f$ can be written as (see \eqref{eq:MDproxview})
$$x_{t+1} = \argmin_{x \in \mathbb{R}^n} \eta \nabla f(x_t)^{\top} x + \frac{1}{2} \|x - x_t\|^2_2 .$$
Here one wants to minimize $f+g$, and $g$ is assumed to be known and ``simple". Thus it seems quite natural to consider the following update rule, where only $f$ is locally approximated with a first order oracle:
\begin{eqnarray}
x_{t+1} & = & \argmin_{x \in \mathbb{R}^n} \eta (g(x) + \nabla f(x_t)^{\top} x) + \frac{1}{2} \|x - x_t\|^2_2 \notag \\
& = & \argmin_{x \in \mathbb{R}^n} \ g(x) + \frac{1}{2\eta} \|x - (x_t - \eta \nabla f(x_t)) \|_2^2 . \label{eq:proxoperator}
\end{eqnarray}
The algorithm described by the above iteration is known as ISTA (Iterative Shrinkage-Thresholding Algorithm). In terms of convergence rate it is easy to show that ISTA has the same convergence rate on $f+g$ as gradient descent on $f$. More precisely with $\eta=\frac{1}{\beta}$ one has
$$f(x_t) + g(x_t) - (f(x^*) + g(x^*)) \leq \frac{\beta \|x_1 - x^*\|^2_2}{2 t} .$$
This improved convergence rate over a subgradient descent directly on $f+g$ comes at a price: in general \eqref{eq:proxoperator} may be a difficult optimization problem by itself, and this is why one needs to assume that $g$ is simple. For instance if $g$ can be written as $g(x) = \sum_{i=1}^n g_i(x(i))$ then one can compute $x_{t+1}$ by solving $n$ convex problems in dimension $1$. In the case where $g(x) = \lambda \|x\|_1$ this one-dimensional problem is given by:
$$\min_{x \in \mathbb{R}} \ \lambda |x| + \frac{1}{2 \eta}(x - x_0)^2, \ \text{where} \ x_0 \in \mathbb{R} .$$
Elementary computations shows that this problem has an analytical solution given by $\tau_{\lambda \eta}(x_0)$,
where $\tau$ is the shrinkage operator (hence the name ISTA), defined by
$$\tau_{\alpha}(x) = (|x|-\alpha)_+ \mathrm{sign}(x) .$$
Much more is known about \eqref{eq:proxoperator} (which is called the {\em proximal operator} of $g$), and in fact entire monographs have been written about this equation, see e.g. \cite{PB13, BJMO12}.

\section*{FISTA (Fast ISTA)}
An obvious idea is to combine Nesterov's accelerated gradient descent (which results in a $1/t^2$ rate to optimize $f$) with ISTA. This results in FISTA (Fast ISTA) which is described as follows. Let
$$\lambda_0 = 0, \ \lambda_{t} = \frac{1 + \sqrt{1+ 4 \lambda_{t-1}^2}}{2}, \ \text{and} \  \gamma_t = \frac{1 - \lambda_t}{\lambda_{t+1}}.$$
Let $x_1 = y_1$ an arbitrary initial point, and
\begin{eqnarray*}
y_{t+1} & = & \mathrm{argmin}_{x \in \mathbb{R}^n} \ g(x) + \frac{\beta}{2} \|x - (x_t - \frac1{\beta} \nabla f(x_t)) \|_2^2 , \\
x_{t+1} & = & (1 - \gamma_t) y_{t+1} + \gamma_t y_t .
\end{eqnarray*}
Again it is easy show that the rate of convergence of FISTA on $f+g$ is similar to the one of Nesterov's accelerated gradient descent on $f$, more precisely:
$$f(y_t) + g(y_t) - (f(x^*) + g(x^*)) \leq \frac{2 \beta \|x_1 - x^*\|^2}{t^2} .$$

\section*{CMD and RDA}
ISTA and FISTA assume smoothness in the Euclidean metric. Quite naturally one can also use these ideas in a non-Euclidean setting. Starting from \eqref{eq:MDproxview} one obtains the CMD (Composite Mirror Descent) algorithm of \cite{DSSST10}, while with \eqref{eq:DA0} one obtains the RDA (Regularized Dual Averaging) of \cite{Xia10}. We refer to these papers for more details.

\section{Smooth saddle-point representation of a non-smooth function} \label{sec:sprepresentation}
Quite often the non-smoothness of a function $f$ comes from a $\max$ operation. More precisely non-smooth functions can often be represented as
\begin{equation} \label{eq:sprepresentation}
f(x) = \max_{1 \leq i \leq m} f_i(x) ,
\end{equation}
where the functions $f_i$ are smooth. This was the case for instance with the function we used to prove the black-box lower bound $1/\sqrt{t}$ for non-smooth optimization in Theorem \ref{th:lb1}. We will see now that by using this structural representation one can in fact attain a rate of $1/t$. This was first observed in \cite{Nes04b} who proposed the Nesterov's smoothing technique. Here we will present the alternative method of \cite{Nem04} which we find more transparent (yet another version is the Chambolle-Pock algorithm, see \cite{CP11}). Most of what is described in this section can be found in \cite{JN11a, JN11b}.

In the next subsection we introduce the more general problem of saddle point computation. We then proceed to apply a modified version of mirror descent to this problem, which will be useful both in Chapter \ref{rand} and also as a warm-up for the more powerful modified mirror prox that we introduce next.

\subsection{Saddle point computation} \label{sec:sp}
Let $\cX \subset \R^n$, $\cY \subset \R^m$ be compact and convex sets. Let $\phi : \cX \times \cY \rightarrow \mathbb{R}$ be a continuous function, such that $\phi(\cdot, y)$ is convex and $\phi(x, \cdot)$ is concave. We write $g_{\cX}(x,y)$ (respectively $g_{\cY}(x,y)$) for an element of $\partial_x \phi(x,y)$ (respectively $\partial_y (-\phi(x,y))$). We are interested in computing
$$\min_{x \in \cX} \max_{y \in \cY} \phi(x,y) .$$
By Sion's minimax theorem there exists a pair $(x^*, y^*) \in \cX \times \cY$ such that
$$\phi(x^*,y^*) = \min_{x \in \mathcal{X}} \max_{y \in \mathcal{Y}} \phi(x,y) = \max_{y \in \mathcal{Y}} \min_{x \in \mathcal{X}} \phi(x,y) .$$
We will explore algorithms that produce a candidate pair of solutions $(\tx, \ty) \in \cX \times \cY$. The quality of $(\tx, \ty)$ is evaluated through the so-called duality gap\footnote{Observe that the duality gap is the sum of the primal gap $\max_{y \in \mathcal{Y}} \phi(\tx,y) - \phi(x^*,y^*)$ and the dual gap $\phi(x^*,y^*) - \min_{x \in \mathcal{X}} \phi(x, \ty)$.}
$$\max_{y \in \mathcal{Y}} \phi(\tx,y) - \min_{x \in \mathcal{X}} \phi(x,\ty) .$$
The key observation is that the duality gap can be controlled similarly to the suboptimality gap $f(x) - f(x^*)$ in a simple convex optimization problem. Indeed for any $(x, y) \in \cX \times \cY$,
$$\phi(\tx,\ty) - \phi(x,\ty) \leq g_{\cX}(\tx,\ty)^{\top} (\tx-x),$$
and 
$$- \phi(\tx,\ty) - (- \phi(\tx,y)) \leq g_{\cY}(\tx,\ty)^{\top} (\ty-y) .$$
In particular, using the notation $z = (x,y) \in \cZ := \cX \times \cY$ and $g(z) = (g_{\cX}(x,y), g_{\cY}(x,y))$ we just proved
\begin{equation} \label{eq:keysp}
\max_{y \in \mathcal{Y}} \phi(\tx,y) - \min_{x \in \mathcal{X}} \phi(x,\ty) \leq g(\tz)^{\top} (\tz - z) , 
\end{equation}
for some $z \in \mathcal{Z}.$ In view of the vector field point of view developed in Section \ref{sec:vectorfield} this suggests to do a mirror descent in the $\cZ$-space with the vector field $g : \cZ \rightarrow \R^n \times \R^m$. 

We will assume in the next subsections that $\cX$ is equipped with a mirror map $\Phi_{\cX}$ (defined on $\cD_{\cX}$) which is $1$-strongly convex w.r.t. a norm $\|\cdot\|_{\cX}$ on $\cX \cap \cD_{\cX}$. We denote $R^2_{\cX} = \sup_{x \in \cX} \Phi(x) - \min_{x \in \cX} \Phi(x)$. We define similar quantities for the space $\cY$.

\subsection{Saddle Point Mirror Descent (SP-MD)} \label{sec:spmd}
We consider here mirror descent on the space $\cZ = \cX \times \cY$ with the mirror map $\Phi(z) = a \Phi_{\cX}(x) + b \Phi_{\cY}(y)$ (defined on $\cD = \cD_{\cX} \times \cD_{\cY}$), where $a, b \in \R_+$ are to be defined later, and with the vector field $g : \cZ \rightarrow \R^n \times \R^m$ defined in the previous subsection. We call the resulting algorithm SP-MD (Saddle Point Mirror Descent). It can be described succintly as follows.

Let $z_1 \in \argmin_{z \in \cZ \cap \cD} \Phi(z)$. Then for $t \geq 1$, let
$$z_{t+1} \in \argmin_{z \in \cZ \cap \cD} \ \eta g_t^{\top} z + D_{\Phi}(z,z_t) ,$$
where $g_t = (g_{\cX,t}, g_{\cY,t})$ with $g_{\cX,t} \in \partial_x \phi(x_t,y_t)$ and $g_{\cY,t} \in \partial_y (- \phi(x_t,y_t))$.

\begin{theorem} \label{th:spmd}
Assume that $\phi(\cdot, y)$ is $L_{\cX}$-Lipschitz w.r.t. $\|\cdot\|_{\cX}$, that is $\|g_{\cX}(x,y)\|_{\cX}^* \leq L_{\cX}, \forall (x, y) \in \cX \times \cY$. Similarly assume that $\phi(x, \cdot)$ is $L_{\cY}$-Lipschitz w.r.t. $\|\cdot\|_{\cY}$. Then SP-MD with $a= \frac{L_{\cX}}{R_{\cX}}$, $b=\frac{L_{\cY}}{R_{\cY}}$, and $\eta=\sqrt{\frac{2}{t}}$ satisfies
$$\max_{y \in \mathcal{Y}} \phi\left( \frac1{t} \sum_{s=1}^t x_s,y \right) - \min_{x \in \mathcal{X}} \phi\left(x, \frac1{t} \sum_{s=1}^t y_s \right) \leq (R_{\cX} L_{\cX} + R_{\cY} L_{\cY}) \sqrt{\frac{2}{t}}.$$
\end{theorem}

\begin{proof}
First we endow $\mathcal{Z}$ with the norm $\|\cdot\|_{\cZ}$ defined by
$$\|z\|_{\cZ} = \sqrt{a \|x\|_{\mathcal{X}}^2 + b \|y\|_{\mathcal{Y}}^2} .$$
It is immediate that $\Phi$ is $1$-strongly convex with respect to $\|\cdot\|_{\mathcal{Z}}$ on $\mathcal{Z} \cap \mathcal{D}$. Furthermore one can easily check that
$$\|z\|_{\mathcal{Z}}^* = \sqrt{\frac1{a} \left(\|x\|_{\mathcal{X}}^*\right)^2 + \frac1{b} \left(\|y\|_{\mathcal{Y}}^*\right)^2} ,$$
and thus the vector field $(g_t)$ used in the SP-MD satisfies:
$$\|g_t\|_{\mathcal{Z}}^* \leq \sqrt{\frac{L_{\mathcal{X}}^2}{a} + \frac{L_{\mathcal{Y}}^2}{b}} .$$
Using \eqref{eq:vfMD} together with \eqref{eq:keysp} and the values of $a, b$ and $\eta$ concludes the proof.
\end{proof}

\subsection{Saddle Point Mirror Prox (SP-MP)}
We now consider the most interesting situation in the context of this chapter, where the function $\phi$ is smooth. Precisely we say that $\phi$ is $(\beta_{11}, \beta_{12}, \beta_{22}, \beta_{21})$-smooth if for any $x, x' \in \cX, y, y' \in \cY$, 
\begin{align*}
& \|\nabla_x \phi(x,y) - \nabla_x \phi(x',y) \|_{\mathcal{X}}^* \leq \beta_{11} \|x-x'\|_{\mathcal{X}} , \\
& \|\nabla_x \phi(x,y) - \nabla_x \phi(x,y') \|_{\mathcal{X}}^* \leq \beta_{12} \|y-y'\|_{\mathcal{Y}} , \\
& \|\nabla_y \phi(x,y) - \nabla_y \phi(x,y') \|_{\mathcal{Y}}^* \leq \beta_{22} \|y-y'\|_{\mathcal{Y}} , \\
& \|\nabla_y \phi(x,y) - \nabla_y \phi(x',y) \|_{\mathcal{Y}}^* \leq \beta_{21} \|x-x'\|_{\mathcal{X}} ,
\end{align*}
This will imply the Lipschitzness of the vector field $g : \cZ \rightarrow \R^n \times \R^m$ under the appropriate norm. Thus we use here mirror prox on the space $\cZ$ with the mirror map $\Phi(z) = a \Phi_{\cX}(x) + b \Phi_{\cY}(y)$ and the vector field $g$. The resulting algorithm is called SP-MP (Saddle Point Mirror Prox) and we can describe it succintly as follows.

Let $z_1 \in \argmin_{z \in \cZ \cap \cD} \Phi(z)$. Then for $t \geq 1$, let $z_t=(x_t,y_t)$ and $w_t=(u_t, v_t)$ be defined by
\begin{eqnarray*}
w_{t+1} & = & \argmin_{z \in \cZ \cap \cD} \ \eta (\nabla_x \phi(x_t, y_t), - \nabla_y \phi(x_t,y_t))^{\top} z + D_{\Phi}(z,z_t) \\
z_{t+1} & = & \argmin_{z \in \cZ \cap \cD} \ \eta (\nabla_x \phi(u_{t+1}, v_{t+1}), - \nabla_y \phi(u_{t+1},v_{t+1}))^{\top} z + D_{\Phi}(z,z_t) .
\end{eqnarray*}

\begin{theorem} \label{th:spmp}
Assume that $\phi$ is $(\beta_{11}, \beta_{12}, \beta_{22}, \beta_{21})$-smooth. Then SP-MP with $a= \frac{1}{R_{\cX}^2}$, $b=\frac{1}{R_{\cY}^2}$, and 
$\eta= 1 / \left(2 \max \left(\beta_{11} R^2_{\cX}, \beta_{22} R^2_{\cY}, \beta_{12} R_{\cX} R_{\cY}, \beta_{21} R_{\cX} R_{\cY}\right) \right)$
satisfies
\begin{align*}
& \max_{y \in \mathcal{Y}} \phi\left( \frac1{t} \sum_{s=1}^t u_{s+1},y \right) - \min_{x \in \mathcal{X}} \phi\left(x, \frac1{t} \sum_{s=1}^t v_{s+1} \right) \\
& \leq \max \left(\beta_{11} R^2_{\cX}, \beta_{22} R^2_{\cY}, \beta_{12} R_{\cX} R_{\cY}, \beta_{21} R_{\cX} R_{\cY}\right) \frac{4}{t} .
\end{align*}
\end{theorem}

\begin{proof}
In light of the proof of Theorem \ref{th:spmd} and \eqref{eq:vfMP} it clearly suffices to show that the vector field $g(z) = (\nabla_x \phi(x,y), - \nabla_y \phi_(x,y))$ is $\beta$-Lipschitz w.r.t. $\|z\|_{\cZ} = \sqrt{\frac{1}{R_{\cX}^2} \|x\|_{\mathcal{X}}^2 + \frac{1}{R_{\cY}^2} \|y\|_{\mathcal{Y}}^2}$ with $\beta = 2 \max \left(\beta_{11} R^2_{\cX}, \beta_{22} R^2_{\cY}, \beta_{12} R_{\cX} R_{\cY}, \beta_{21} R_{\cX} R_{\cY}\right)$. In other words one needs to show that
$$\|g(z) - g(z')\|_{\cZ}^* \leq \beta \|z - z'\|_{\cZ} ,$$
which can be done with straightforward calculations (by introducing $g(x',y)$ and using the definition of smoothness for $\phi$).
\end{proof}

\subsection{Applications} \label{sec:spex}
We investigate briefly three applications for SP-MD and SP-MP.

\subsubsection{Minimizing a maximum of smooth functions} \label{sec:spex1}
The problem \eqref{eq:sprepresentation} (when $f$ has to minimized over $\cX$) can be rewritten as
$$\min_{x \in \cX} \max_{y \in \Delta_m} \vec{f}(x)^{\top} y ,$$
where $\vec{f}(x) = (f_1(x), \hdots, f_m(x)) \in \R^m$. We assume that the functions $f_i$ are $L$-Lipschtiz and $\beta$-smooth w.r.t. some norm $\|\cdot\|_{\cX}$. Let us study the smoothness of $\phi(x,y) = \vec{f}(x)^{\top} y$ when $\cX$ is equipped with $\|\cdot\|_{\cX}$ and $\Delta_m$ is equipped with $\|\cdot\|_1$. On the one hand $\nabla_y \phi(x,y) = \vec{f}(x)$, in particular one immediately has $\beta_{22}=0$, and furthermore
$$ \|\vec{f}(x)  - \vec{f}(x') \|_{\infty} \leq L \|x-x'\|_{\mathcal{X}} , $$
that is $\beta_{21}=L$. On the other hand $\nabla_x \phi(x,y) = \sum_{i=1}^m y_i \nabla f_i(x)$, and thus
\begin{align*}
& \|\sum_{i=1}^m y(i) (\nabla f_i(x) - \nabla f_i(x')) \|_{\cX}^* \leq \beta \|x-x'\|_{\cX} , \\
& \|\sum_{i=1}^m (y(i)-y'(i)) \nabla f_i(x) \|_{\cX}^* \leq L\|y-y'\|_1 ,
\end{align*}
that is $\beta_{11} = \beta$ and $\beta_{12} = L$. Thus using SP-MP with some mirror map on $\cX$ and the negentropy on $\Delta_m$ (see the ``simplex setup" in Section \ref{sec:mdsetups}), one obtains an $\epsilon$-optimal point of $f(x) = \max_{1 \leq i \leq m} f_i(x)$ in $O\left(\frac{\beta R_{\cX}^2 + L R_{\cX} \sqrt{\log(m)}}{\epsilon} \right)$ iterations. Furthermore an iteration of SP-MP has a computational complexity of order of a step of mirror descent in $\cX$ on the function $x \mapsto \sum_{i=1}^m y(i) f_i(x)$ (plus $O(m)$ for the update in the $\cY$-space).

Thus by using the structure of $f$ we were able to obtain a much better rate than black-box procedures (which would have required $\Omega(1/\epsilon^2)$ iterations as $f$ is potentially non-smooth).

\subsubsection{Matrix games} \label{sec:spex2}
Let $A \in \R^{n \times m}$, we denote $\|A\|_{\mathrm{max}}$ for the maximal entry (in absolute value) of $A$, and $A_i \in \R^n$ for the $i^{th}$ column of $A$. We consider the problem of computing a Nash equilibrium for the zero-sum game corresponding to the loss matrix $A$, that is we want to solve
$$\min_{x \in \Delta_n} \max_{y \in \Delta_m} x^{\top} A y .$$
Here we equip both $\Delta_n$ and $\Delta_m$ with $\|\cdot\|_1$. Let $\phi(x,y) = x^{\top} A y$. Using that $\nabla_x \phi(x,y) = Ay$ and $\nabla_y \phi(x,y) = A^{\top} x$ one immediately obtains $\beta_{11} = \beta_{22} = 0$. Furthermore since
$$\|A(y - y') \|_{\infty} = \|\sum_{i=1}^m (y(i) - y'(i)) A_i \|_{\infty} \leq \|A\|_{\mathrm{max}} \|y - y'\|_1 ,$$
one also has $\beta_{12} = \beta_{21} = \|A\|_{\mathrm{max}}$. Thus SP-MP with the negentropy on both $\Delta_n$ and $\Delta_m$ attains an $\epsilon$-optimal pair of mixed strategies with $O\left(\|A\|_{\mathrm{max}} \sqrt{\log(n) \log(m)} / \epsilon \right)$ iterations. Furthermore the computational complexity of a step of SP-MP is dominated by the matrix-vector multiplications which are $O(n m)$. Thus overall the complexity of getting an $\epsilon$-optimal Nash equilibrium with SP-MP is $O\left(\|A\|_{\mathrm{max}} n m \sqrt{\log(n) \log(m)} / \epsilon  \right)$.

\subsubsection{Linear classification} \label{sec:spex3}
Let $(\ell_i, A_i) \in \{-1,1\} \times \R^n$, $i \in [m]$, be a data set that one wishes to separate with a linear classifier. That is one is looking for $x \in \mB_{2,n}$ such that for all $i \in [m]$, $\mathrm{sign}(x^{\top} A_i) = \mathrm{sign}(\ell_i)$, or equivalently $\ell_i x^{\top} A_i > 0$. Clearly without loss of generality one can assume $\ell_i = 1$ for all $i \in [m]$ (simply replace $A_i$ by $\ell_i A_i$). Let $A \in \R^{n \times m}$ be the matrix where the $i^{th}$ column is $A_i$. The problem of finding $x$ with maximal margin can be written as
\begin{equation} \label{eq:linearclassif}
\max_{x \in \mB_{2,n}} \min_{1 \leq i \leq m} A_i^{\top} x = \max_{x \in \mB_{2,n}} \min_{y \in \Delta_m} x^{\top} A y .
\end{equation}
Assuming that $\|A_i\|_2 \leq B$, and using the calculations we did in Section \ref{sec:spex1}, it is clear that $\phi(x,y) = x^{\top} A y$ is $(0, B, 0, B)$-smooth with respect to $\|\cdot\|_2$ on $\mB_{2,n}$ and $\|\cdot\|_1$ on $\Delta_m$. This implies in particular that SP-MP with the Euclidean norm squared on $\mB_{2,n}$ and the negentropy on $\Delta_m$ will solve \eqref{eq:linearclassif} in $O(B \sqrt{\log(m)} / \epsilon)$ iterations. Again the cost of an iteration is dominated by the matrix-vector multiplications, which results in an overall complexity of $O(B n m \sqrt{\log(m)} / \epsilon)$ to find an $\epsilon$-optimal solution to \eqref{eq:linearclassif}.

\section{Interior point methods} \label{sec:IPM}
We describe here interior point methods (IPM), a class of algorithms fundamentally different from what we have seen so far. The first algorithm of this type was described in \cite{Kar84}, but the theory we shall present was developed in \cite{NN94}. We follow closely the presentation given in [Chapter 4, \cite{Nes04}]. Other useful references (in particular for the primal-dual IPM, which are the ones used in practice) include \cite{Ren01, Nem04b, NW06}.
\newline

IPM are designed to solve convex optimization problems of the form
\begin{align*}
& \mathrm{min.} \; c^{\top} x \\
& \text{s.t.} \; x \in \cX ,
\end{align*}
with $c \in \R^n$, and $\cX \subset \R^n$ convex and compact. 
Note that, at this point, the linearity of the objective is without loss of generality as minimizing a convex function $f$ over $\cX$ is equivalent to minimizing a linear objective over the epigraph of $f$ (which is also a convex set). The structural assumption on $\cX$ that one makes in IPM is that there exists a {\em self-concordant barrier} for $\cX$ with an easily computable gradient and Hessian. The meaning of the previous sentence will be made precise in the next subsections. The importance of IPM stems from the fact that LPs and SDPs (see Section \ref{sec:structured}) satisfy this structural assumption.

\subsection{The barrier method} \label{sec:barriermethod}
We say that $F : \inte(\cX) \rightarrow \R$ is a {\em barrier} for $\cX$ if 
$$F(x) \xrightarrow[x \to \partial \cX]{} +\infty .$$
We will only consider strictly convex barriers. We extend the domain of definition of $F$ to $\R^n$ with $F(x) = +\infty$ for $x \not\in \inte(\cX)$. For $t \in \R_+$ let
$$x^*(t) \in \argmin_{x \in \R^n} t c^{\top} x + F(x) .$$
In the following we denote $F_t(x) := t c^{\top} x + F(x)$.
In IPM the path $(x^*(t))_{t \in \R_+}$ is referred to as the {\em central path}. It seems clear that the central path eventually leads to the minimum $x^*$ of the objective function $c^{\top} x$ on $\cX$, precisely we will have
$$x^*(t) \xrightarrow[t \to +\infty]{} x^* .$$
The idea of the {\em barrier method} is to move along the central path by ``boosting" a fast locally convergent algorithm, which we denote for the moment by $\cA$, using the following scheme: Assume that one has computed $x^*(t)$, then one uses $\cA$ initialized at $x^*(t)$ to compute $x^*(t')$ for some $t'>t$. There is a clear tension for the choice of $t'$, on the one hand $t'$ should be large in order to make as much progress as possible on the central path, but on the other hand $x^*(t)$ needs to be close enough to $x^*(t')$ so that it is in the basin of fast convergence for $\cA$ when run on $F_{t'}$. 

IPM follows the above methodology with $\cA$ being {\em Newton's method}. Indeed as we will see in the next subsection, Newton's method has a quadratic convergence rate, in the sense that if initialized close enough to the optimum it attains an $\epsilon$-optimal point in $\log\log(1/\epsilon)$ iterations! Thus we now have a clear plan to make these ideas formal and analyze the iteration complexity of IPM:
\begin{enumerate}
\item First we need to describe precisely the region of fast convergence for Newton's method. This will lead us to define self-concordant functions, which are ``natural" functions for Newton's method.
\item Then we need to evaluate precisely how much larger $t'$ can be compared to $t$, so that $x^*(t)$ is still in the region of fast convergence of Newton's method when optimizing the function $F_{t'}$ with $t'>t$. This will lead us to define $\nu$-self concordant barriers.
\item How do we get close to the central path in the first place? Is it possible to compute $x^*(0) = \argmin_{x \in \R^n} F(x)$ (the so-called analytical center of $\mathcal{X}$)?
\end{enumerate}

\subsection{Traditional analysis of Newton's method} \label{sec:tradanalysisNM}
We start by describing Newton's method together with its standard analysis showing the quadratic convergence rate when initialized close enough to the optimum. In this subsection we denote $\|\cdot\|$ for both the Euclidean norm on $\R^n$ and the operator norm on matrices (in particular $\|A x\| \leq \|A\| \cdot \|x\|$).

Let $f: \R^n \rightarrow \R$ be a $C^2$ function. 
Using a Taylor's expansion of $f$ around $x$ one obtains
$$f(x+h) = f(x) + h^{\top} \nabla f(x) + \frac12 h^{\top} \nabla^2 f(x) h + o(\|h\|^2) .$$
Thus, starting at $x$, in order to minimize $f$ it seems natural to move in the direction $h$ that minimizes 
$$h^{\top} \nabla f(x) + \frac12 h^{\top} \nabla f^2(x) h .$$
If $\nabla^2 f(x)$ is positive definite then the solution to this problem is given by $h = - [\nabla^2 f(x)]^{-1} \nabla f(x)$. Newton's method simply iterates this idea: starting at some point $x_0 \in \R^n$, it iterates for $k \geq 0$ the following equation:
$$x_{k+1} = x_k  - [\nabla^2 f(x_k)]^{-1} \nabla f(x_k) .$$
While this method can have an arbitrarily bad behavior in general, if started close enough to a strict local minimum of $f$, it can have a very fast convergence:

\begin{theorem}
\label{th:NM}
Assume that $f$ has a Lipschitz Hessian, that is $\| \nabla^2 f(x) - \nabla^2 f(y) \| \leq M \|x - y\|$. Let $x^*$ be local minimum of $f$ with strictly positive Hessian, that is $\nabla^2 f(x^*) \succeq \mu \mI_n$, $\mu > 0$. Suppose that the initial starting point $x_0$ of Newton's method is such that
$$\|x_0 - x^*\| \leq \frac{\mu}{2 M} .$$
Then Newton's method is well-defined and converges to $x^*$ at a quadratic rate:
$$\|x_{k+1} - x^*\| \leq \frac{M}{\mu} \|x_k - x^*\|^2.$$
\end{theorem}

\begin{proof}
We use the following simple formula, for $x, h \in \R^n$,
$$\int_0^1 \nabla^2 f(x + s h) \ h \ ds = \nabla f(x+h) - \nabla f(x) .$$
Now note that $\nabla f(x^*) = 0$, and thus with the above formula one obtains
$$\nabla f(x_k) = \int_0^1 \nabla^2 f(x^* + s (x_k - x^*)) \ (x_k - x^*) \ ds ,$$
which allows us to write:
\begin{align*}
& x_{k+1} - x^* \\
& = x_k - x^* - [\nabla^2 f(x_k)]^{-1} \nabla f(x_k) \\
& = x_k - x^* - [\nabla^2 f(x_k)]^{-1} \int_0^1 \nabla^2 f(x^* + s (x_k - x^*)) \ (x_k - x^*) \ ds \\
& = [\nabla^2 f(x_k)]^{-1} \int_0^1 [\nabla^2 f (x_k) - \nabla^2 f(x^* + s (x_k - x^*)) ] \ (x_k - x^*) \ ds .
\end{align*}
In particular one has
\begin{align*}
& \|x_{k+1} - x^*\| \\
& \leq \|[\nabla^2 f(x_k)]^{-1}\| \\
& \times \left( \int_0^1 \| \nabla^2 f (x_k) - \nabla^2 f(x^* + s (x_k - x^*)) \| \ ds \right) \|x_k - x^* \|.
\end{align*}
Using the Lipschitz property of the Hessian one immediately obtains that 
$$\left( \int_0^1 \| \nabla^2 f (x_k) - \nabla^2 f(x^* + s (x_k - x^*)) \| \ ds \right) \leq \frac{M}{2} \|x_k - x^*\| .$$
Using again the Lipschitz property of the Hessian (note that $\|A - B\| \leq s \Leftrightarrow s \mI_n \succeq A - B \succeq - s \mI_n$), the hypothesis on $x^*$, and an induction hypothesis that $\|x_k - x^*\| \leq \frac{\mu}{2M}$, one has
$$\nabla^2 f(x_k) \succeq \nabla^2 f(x^*) - M \|x_k - x^*\| \mI_n \succeq (\mu - M \|x_k - x^*\|) \mI_n \succeq \frac{\mu}{2} \mI_n ,$$
which concludes the proof.
\end{proof}

\subsection{Self-concordant functions}
Before giving the definition of self-concordant functions let us try to get some insight into the ``geometry" of Newton's method. Let $A$ be a $n \times n$ non-singular matrix. We look at a Newton step on the functions $f: x \mapsto f(x)$ and $\phi: y \mapsto f(A^{-1} y)$, starting respectively from $x$ and $y= A x$, that is:
$$x^+ = x  - [\nabla^2 f(x)]^{-1} \nabla f(x) , \; \text{and} \; y^+ = y  - [\nabla^2 \phi(y)]^{-1} \nabla \phi(y) .$$
By using the following simple formulas
$$\nabla (x \mapsto f(A x) ) =A^{\top} \nabla f(A x) , \; \text{and} \; \nabla^2 (x \mapsto f(A x) ) =A^{\top} \nabla^2 f(A x) A .$$
it is easy to show that
$$y^+ = A x^+ .$$
In other words Newton's method will follow the same trajectory in the ``$x$-space" and in the ``$y$-space" (the image through $A$ of the $x$-space), that is Newton's method is {\em affine invariant}. Observe that this property is not shared by the methods described in Chapter \ref{dimfree} (except for the conditional gradient descent).

The affine invariance of Newton's method casts some concerns on the assumptions of the analysis in Section \ref{sec:tradanalysisNM}. Indeed the assumptions are all in terms of the canonical inner product in $\R^n$. However we just showed that the method itself does not depend on the choice of the inner product (again this is not true for first order methods). Thus one would like to derive a result similar to Theorem \ref{th:NM} without any reference to a prespecified inner product. The idea of self-concordance is to modify the Lipschitz assumption on the Hessian to achieve this goal.

Assume from now on that $f$ is $C^3$, and let $\nabla^3 f(x) : \R^n \times \R^n \times \R^n \rightarrow \R$ be the third order differential operator. The Lipschitz assumption on the Hessian in Theorem \ref{th:NM} can be written as:
$$\nabla^3 f(x) [h,h,h] \leq M \|h\|_2^3 .$$
The issue is that this inequality depends on the choice of an inner product. More importantly it is easy to see that a convex function which goes to infinity on a compact set simply cannot satisfy the above inequality. A natural idea to try fix these issues is to replace the Euclidean metric on the right hand side by the metric given by the function $f$ itself at $x$, that is:
$$\|h\|_x = \sqrt{ h^{\top} \nabla^2 f(x) h }.$$
Observe that to be clear one should rather use the notation $\|\cdot\|_{x, f}$, but since $f$ will always be clear from the context we stick to $\|\cdot\|_x$.
\begin{definition}
Let $\mathcal{X}$ be a convex set with non-empty interior, and $f$ a $C^3$ convex function defined on $\inte(\mathcal{X})$. Then $f$ is self-concordant (with constant $M$) if for all $x \in \inte(\mathcal{X}), h \in \R^n$,
$$\nabla^3 f(x) [h,h,h] \leq M \|h\|_x^3 .$$
We say that $f$ is standard self-concordant if $f$ is self-concordant with constant $M=2$.
\end{definition}

An easy consequence of the definition is that a self-concordant function is a barrier for the set $\mathcal{X}$, see [Theorem 4.1.4, \cite{Nes04}]. The main example to keep in mind of a standard self-concordant function is $f(x) = - \log x$ for $x > 0$. The next definition will be key in order to describe the region of quadratic convergence for Newton's method on self-concordant functions. 

\begin{definition}
Let $f$ be a standard self-concordant function on $\mathcal{X}$. For $x \in \mathrm{int}(\mathcal{X})$, we say that $\lambda_f(x) = \|\nabla f(x)\|_x^*$ is the {\em Newton decrement} of $f$ at $x$.
\end{definition}
An important inequality is that for $x$ such that $\lambda_f(x) < 1$, and $x^* = \argmin f(x)$, one has
\begin{equation} \label{eq:trucipm3}
\|x - x^*\|_x \leq \frac{\lambda_f(x)}{1 - \lambda_f(x)} ,
\end{equation}
see [Equation 4.1.18, \cite{Nes04}]. We state the next theorem without a proof, see also [Theorem 4.1.14, \cite{Nes04}].
\begin{theorem} \label{th:NMsc}
Let $f$ be a standard self-concordant function on $\mathcal{X}$, and $x \in \mathrm{int}(\mathcal{X})$ such that $\lambda_f(x) \leq 1/4$, then
$$\lambda_f\Big(x - [\nabla^2 f(x)]^{-1} \nabla f(x)\Big) \leq 2 \lambda_f(x)^2 .$$
\end{theorem}
In other words the above theorem states that, if initialized at a point $x_0$ such that $\lambda_f(x_0) \leq 1/4$, then Newton's iterates satisfy $\lambda_f(x_{k+1}) \leq 2 \lambda_f(x_k)^2$. Thus, Newton's region of quadratic convergence for self-concordant functions can be described as a ``Newton decrement ball" $\{x : \lambda_f(x) \leq 1/4\}$. In particular by taking the barrier to be a self-concordant function we have now resolved Step (1) of the plan described in Section \ref{sec:barriermethod}. 

\subsection{$\nu$-self-concordant barriers}
We deal here with Step (2) of the plan described in Section \ref{sec:barriermethod}. Given Theorem \ref{th:NMsc} we want $t'$ to be as large as possible and such that
\begin{equation} \label{eq:trucipm1}
\lambda_{F_{t'}}(x^*(t) ) \leq 1/4 .
\end{equation}
Since the Hessian of $F_{t'}$ is the Hessian of $F$, one has
$$\lambda_{F_{t'}}(x^*(t) ) = \|t' c + \nabla F(x^*(t)) \|_{x^*(t)}^* .$$
Observe that, by first order optimality, one has 
$t c + \nabla F(x^*(t))  = 0,$
which yields
\begin{equation} \label{eq:trucipm11}
\lambda_{F_{t'}}(x^*(t) ) = (t'-t) \|c\|^*_{x^*(t)} .
\end{equation}
Thus taking 
\begin{equation} \label{eq:trucipm2}
t' = t + \frac{1}{4 \|c\|^*_{x^*(t)}}
\end{equation} 
immediately yields \eqref{eq:trucipm1}. In particular with the value of $t'$ given in \eqref{eq:trucipm2} the Newton's method on $F_{t'}$ initialized at $x^*(t)$ will converge quadratically fast to $x^*(t')$.

It remains to verify that by iterating \eqref{eq:trucipm2} one obtains a sequence diverging to infinity, and to estimate the rate of growth. Thus one needs to control $\|c\|^*_{x^*(t)} = \frac1{t} \|\nabla F(x^*(t))\|_{x^*(t)}^*$. Luckily there is a natural class of functions for which one can control $\|\nabla F(x)\|_x^*$ uniformly over $x$. This is the set of functions such that
\begin{equation} \label{eq:nu}
\nabla^2 F(x) \succeq \frac1{\nu} \nabla F(x) [\nabla F(x) ]^{\top} .
\end{equation}
Indeed in that case one has:
\begin{eqnarray*}
\|\nabla F(x)\|_x^* & = & \sup_{h : h^{\top} \nabla F^2(x) h \leq 1} \nabla F(x)^{\top} h \\
& \leq & \sup_{h : h^{\top} \left( \frac1{\nu} \nabla F(x) [\nabla F(x) ]^{\top} \right) h \leq 1} \nabla F(x)^{\top} h \\
& = & \sqrt{\nu} .
\end{eqnarray*}
Thus a safe choice to increase the penalization parameter is $t' = \left(1 + \frac1{4\sqrt{\nu}}\right) t$. Note that the condition \eqref{eq:nu} can also be written as the fact that the function $F$ is $\frac1{\nu}$-exp-concave, that is $x \mapsto \exp(- \frac1{\nu} F(x))$ is concave. We arrive at the following definition.

\begin{definition}
$F$ is a $\nu$-self-concordant barrier if it is a standard self-concordant function, and it is $\frac1{\nu}$-exp-concave.
\end{definition}
Again the canonical example is the logarithmic function, $x \mapsto - \log x$, which is a $1$-self-concordant barrier for the set $\R_{+}$. We state the next theorem without a proof (see \cite{BE14} for more on this result).

\begin{theorem}
Let $\mathcal{X} \subset \R^n$ be a closed convex set with non-empty interior. There exists $F$ which is a $(c \ n)$-self-concordant barrier for $\mathcal{X}$ (where $c$ is some universal constant).
\end{theorem}
A key property of $\nu$-self-concordant barriers is the following inequality:
\begin{equation} \label{eq:key}
c^{\top} x^*(t) - \min_{x \in \mathcal{X}} c^{\top} x \leq \frac{\nu}{t} ,
\end{equation}
see [Equation (4.2.17), \cite{Nes04}]. More generally using \eqref{eq:key} together with \eqref{eq:trucipm3} one obtains
\begin{eqnarray}
c^{\top} y- \min_{x \in \mathcal{X}} c^{\top} x & \leq & \frac{\nu}{t} + c^{\top} (y - x^*(t)) \notag \\
& = & \frac{\nu}{t} + \frac{1}{t} (\nabla F_t(y) - \nabla F(y))^{\top} (y - x^*(t)) \notag \\ 
& \leq & \frac{\nu}{t} + \frac{1}{t} \|\nabla F_t(y) - \nabla F(y)\|_y^* \cdot \|y - x^*(t)\|_y \notag \\
& \leq & \frac{\nu}{t} + \frac{1}{t} (\lambda_{F_t}(y) + \sqrt{\nu})\frac{\lambda_{F_t} (y)}{1 - \lambda_{F_t}(y)} \label{eq:trucipm4}
\end{eqnarray}
In the next section we describe a precise algorithm based on the ideas we developed above. As we will see one cannot ensure to be exactly on the central path, and thus it is useful to generalize the identity \eqref{eq:trucipm11} for a point $x$ close to the central path. We do this as follows:
\begin{eqnarray}
\lambda_{F_{t'}}(x) & = & \|t' c + \nabla F(x)\|_x^* \notag \\
& = &  \|(t' / t) (t c + \nabla F(x)) + (1- t'/t) \nabla F(x)\|_x^* \notag \\
& \leq & \frac{t'}{t} \lambda_{F_t}(x) + \left(\frac{t'}{t} - 1\right) \sqrt{\nu} .\label{eq:trucipm12}
\end{eqnarray}

\subsection{Path-following scheme}
We can now formally describe and analyze the most basic IPM called the {\em path-following scheme}. Let $F$ be $\nu$-self-concordant barrier for $\cX$. Assume that one can find $x_0$ such that $\lambda_{F_{t_0}}(x_0) \leq 1/4$ for some small value $t_0 >0$ (we describe a method to find $x_0$ at the end of this subsection).
Then for $k \geq 0$, let
\begin{eqnarray*}
& & t_{k+1} = \left(1 + \frac1{13\sqrt{\nu}}\right) t_k ,\\
& & x_{k+1} = x_k - [\nabla^2 F(x_k)]^{-1} (t_{k+1} c + \nabla F(x_k) ) .
\end{eqnarray*}
The next theorem shows that after $O\left( \sqrt{\nu} \log \frac{\nu}{t_0 \epsilon} \right)$ iterations of the path-following scheme one obtains an $\epsilon$-optimal point.
\begin{theorem}
The path-following scheme described above satisfies
$$c^{\top} x_k - \min_{x \in \mathcal{X}} c^{\top} x \leq \frac{2 \nu}{t_0} \exp\left( - \frac{k}{1+13\sqrt{\nu}} \right) .$$
\end{theorem}
\begin{proof}
We show that the iterates $(x_k)_{k \geq 0}$ remain close to the central path $(x^*(t_k))_{k \geq 0}$. Precisely one can easily prove by induction that 
$$\lambda_{F_{t_k}}(x_k) \leq 1/4 .$$
Indeed using Theorem \ref{th:NMsc} and equation \eqref{eq:trucipm12} one immediately obtains
\begin{eqnarray*}
\lambda_{F_{t_{k+1}}}(x_{k+1}) & \leq & 2 \lambda_{F_{t_{k+1}}}(x_k)^2 \\
& \leq & 2 \left(\frac{t_{k+1}}{t_k} \lambda_{F_{t_k}}(x_k) + \left(\frac{t_{k+1}}{t_k} - 1\right) \sqrt{\nu}\right)^2  \\
& \leq & 1/4 ,
\end{eqnarray*}
where we used in the last inequality that $t_{k+1} / t_k = 1 + \frac1{13\sqrt{\nu}}$ and $\nu \geq 1$.

Thus using \eqref{eq:trucipm4} one obtains
$$c^{\top} x_k - \min_{x \in \mathcal{X}} c^{\top} x \leq \frac{\nu + \sqrt{\nu} / 3 + 1/12}{t_k} \leq \frac{2 \nu}{t_k} .$$
Observe that $t_{k} = \left(1 + \frac1{13\sqrt{\nu}}\right)^{k} t_0$, which finally yields
$$c^{\top} x_k - \min_{x \in \mathcal{X}} c^{\top} x \leq \frac{2 \nu}{t_0} \left(1 + \frac1{13\sqrt{\nu}}\right)^{- k}.$$
\end{proof}

At this point we still need to explain how one can get close to an intial point $x^*(t_0)$ of the central path. This can be done with the following rather clever trick. Assume that one has some point $y_0 \in \cX$. The observation is that $y_0$ is on the central path at $t=1$ for the problem where $c$ is replaced by $- \nabla F(y_0)$. Now instead of following this central path as $t \to +\infty$, one follows it as $t \to 0$. Indeed for $t$ small enough the central paths for $c$ and for $- \nabla F(y_0)$ will be very close. Thus we iterate the following equations, starting with $t_0' = 1$,
\begin{eqnarray*}
& & t_{k+1}' = \left(1 - \frac1{13\sqrt{\nu}}\right) t_k' ,\\
& & y_{k+1} = y_k - [\nabla^2 F(y_k)]^{-1} (- t_{k+1}' \nabla F(y_0) + \nabla F(y_k) ) .
\end{eqnarray*}
A straightforward analysis shows that for $k = O(\sqrt{\nu} \log \nu)$, which corresponds to $t_k'=1/\nu^{O(1)}$, one obtains a point $y_k$ such that $\lambda_{F_{t_k'}}(y_k) \leq 1/4$. In other words one can initialize the path-following scheme with $t_0 = t_k'$ and $x_0 = y_k$.

\subsection{IPMs for LPs and SDPs}
We have seen that, roughly, the complexity of interior point methods with a $\nu$-self-concordant barrier is $O\left(M \sqrt{\nu} \log \frac{\nu}{\epsilon} \right)$, where $M$ is the complexity of computing a Newton direction (which can be done by computing and inverting the Hessian of the barrier). Thus the efficiency of the method is directly related to the {\em form} of the self-concordant barrier that one can construct for $\mathcal{X}$. It turns out that for LPs and SDPs one has particularly nice self-concordant barriers. Indeed one can show that $F(x) = - \sum_{i=1}^n \log x_i$ is an $n$-self-concordant barrier on $\R_{+}^n$, and $F(x) = - \log \mathrm{det}(X)$ is an $n$-self-concordant barrier on $\mathbb{S}_{+}^n$. See also \cite{LS13} for a recent improvement of the basic logarithmic barrier for LPs.

There is one important issue that we overlooked so far. In most interesting cases LPs and SDPs come with {\em equality constraints}, resulting in a set of constraints $\cX$ with empty interior. From a theoretical point of view there is an easy fix, which is to reparametrize the problem as to enforce the variables to live in the subspace spanned by $\cX$. This modification also has algorithmic consequences, as the evaluation of the Newton direction will now be different. In fact, rather than doing a reparametrization, one can simply search for Newton directions such that the updated point will stay in $\cX$. In other words one has now to solve a convex quadratic optimization problem under linear equality constraints. Luckily using Lagrange multipliers one can find a closed form solution to this problem, and we refer to previous references for more details.

\chapter{Convex optimization and randomness}
\label{rand}
In this chapter we explore the interplay between optimization and randomness. A key insight, going back to \cite{RM51}, is that first order methods are quite robust: the gradients do not have to be computed exactly to ensure progress towards the optimum. Indeed since these methods usually do many small steps, as long as the gradients are correct {\em on average}, the error introduced by the gradient approximations will eventually vanish. As we will see below this intuition is correct for non-smooth optimization (since the steps are indeed small) but the picture is more subtle in the case of smooth optimization (recall from Chapter \ref{dimfree} that in this case we take long steps).

We introduce now the main object of this chapter: a (first order) {\em stochastic} oracle for a convex function $f : \cX \rightarrow \R$ takes as input a point $x \in \cX$ and outputs a random variable $\tg(x)$ such that $\E \ \tg(x) \in \partial f(x)$. In the case where the query point $x$ is a random variable (possibly obtained from previous queries to the oracle), one assumes that $\E \ (\tg(x) | x) \in \partial f(x)$.

The unbiasedness assumption by itself is not enough to obtain rates of convergence, one also needs to make assumptions about the fluctuations of $\tg(x)$.  Essentially in the non-smooth case we will assume that there exists $B >0$ such that $\E \|\tg(x)\|_*^2 \leq B^2$ for all $x \in \cX$, while in the smooth case we assume that there exists $\sigma > 0$ such that  $\E \|\tg(x) - \nabla f(x)\|_*^2 \leq \sigma^2$ for all $x \in \cX$.

We also note that the situation with a {\em biased} oracle is quite different, and we refer to \cite{Asp08, SLRB11} for some works in this direction.

The two canonical examples of a stochastic oracle in machine learning are as follows. 

Let $f(x) = \E_{\xi} \ell(x, \xi)$ where $\ell(x, \xi)$ should be interpreted as the loss of predictor $x$ on the example $\xi$. We assume that $\ell(\cdot, \xi)$ is a (differentiable\footnote{We assume differentiability only for sake of notation here.}) convex function for any $\xi$. The goal is to find a predictor with minimal expected loss, that is to minimize $f$. When queried at $x$ the stochastic oracle can draw $\xi$ from the unknown distribution and report $\nabla_x \ell(x, \xi)$. One obviously has $\E_{\xi} \nabla_x \ell(x, \xi) \in \partial f(x)$. 

The second example is the one described in Section \ref{sec:mlapps}, where one wants to minimize $f(x) = \frac{1}{m} \sum_{i=1}^m f_i(x)$. In this situation a stochastic oracle can be obtained by selecting uniformly at random $I \in [m]$ and reporting $\nabla f_I(x)$.

Observe that the stochastic oracles in the two above cases are quite different. Consider the standard situation where one has access to a data set of i.i.d. samples $\xi_1, \hdots, \xi_m$. Thus in the first case, where one wants to minimize the {\em expected loss}, one is limited to $m$ queries to the oracle, that is to a {\em single pass} over the data (indeed one cannot ensure that the conditional expectations are correct if one uses twice a data point). On the contrary for the {\em empirical loss} where $f_i(x) = \ell(x, \xi_i)$ one can do as many passes as one wishes.

\section{Non-smooth stochastic optimization} \label{sec:smd}
We initiate our study with stochastic mirror descent (S-MD) which is defined as follows: $x_1 \in \argmin_{\cX \cap \cD} \Phi(x)$, and
$$x_{t+1} = \argmin_{x \in \mathcal{X} \cap \mathcal{D}} \ \eta \tilde{g}(x_t)^{\top} x + D_{\Phi}(x,x_t) .$$
In this case equation \eqref{eq:vfMD} rewrites
$$\sum_{s=1}^t \tg(x_s)^{\top} (x_s - x) \leq \frac{R^2}{\eta} + \frac{\eta}{2 \rho} \sum_{s=1}^t \|\tg(x_s)\|_*^2 .$$
This immediately yields a rate of convergence thanks to the following simple observation based on the tower rule:
\begin{eqnarray*}
\E f\bigg(\frac{1}{t} \sum_{s=1}^t x_s \bigg) - f(x) & \leq & \frac{1}{t} \E \sum_{s=1}^t (f(x_s) - f(x)) \\
& \leq & \frac{1}{t} \E \sum_{s=1}^t \E(\tg(x_s) | x_s)^{\top} (x_s - x) \\
& = & \frac{1}{t} \E \sum_{s=1}^t \tg(x_s)^{\top} (x_s - x) .
\end{eqnarray*}
We just proved the following theorem.
\begin{theorem} \label{th:SMD}
Let $\Phi$ be a mirror map $1$-strongly convex on $\mathcal{X} \cap \mathcal{D}$ with respect to $\|\cdot\|$, and
let $R^2 = \sup_{x \in \mathcal{X} \cap \mathcal{D}} \Phi(x) - \Phi(x_1)$. Let $f$ be convex. Furthermore assume that the stochastic oracle is such that $\E \|\tg(x)\|_*^2 \leq B^2$. Then S-MD with $\eta = \frac{R}{B} \sqrt{\frac{2}{t}}$ satisfies
$$\E f\bigg(\frac{1}{t} \sum_{s=1}^t x_s \bigg) - \min_{x \in \mathcal{X}} f(x) \leq R B \sqrt{\frac{2}{t}} .$$
\end{theorem}

Similarly, in the Euclidean and strongly convex case, one can directly generalize Theorem \ref{th:LJSB12}. Precisely we consider stochastic gradient descent (SGD), that is S-MD with $\Phi(x) = \frac12 \|x\|_2^2$, with time-varying step size $(\eta_t)_{t \geq 1}$, that is
$$x_{t+1} = \Pi_{\cX}(x_t - \eta_t \tg(x_t)) .$$
\begin{theorem} \label{th:sgdstrong}
Let $f$ be $\alpha$-strongly convex, and assume that the stochastic oracle is such that $\E \|\tg(x)\|_*^2 \leq B^2$. Then SGD with $\eta_s = \frac{2}{\alpha (s+1)}$ satisfies
$$f \left(\sum_{s=1}^t \frac{2 s}{t(t+1)} x_s \right) - f(x^*) \leq \frac{2 B^2}{\alpha (t+1)} .$$
\end{theorem}

\section{Smooth stochastic optimization and mini-batch SGD}
In the previous section we showed that, for non-smooth optimization, there is basically no cost for having a stochastic oracle instead of an exact oracle. Unfortunately one can show (see e.g. \cite{Tsy03}) that smoothness does not bring any acceleration for a general stochastic oracle\footnote{While being true in general this statement does not say anything about specific functions/oracles. For example it was shown in \cite{BM13} that acceleration can be obtained for the square loss and the logistic loss.}. This is in sharp contrast with the exact oracle case where we showed that gradient descent attains a $1/t$ rate (instead of $1/\sqrt{t}$ for non-smooth), and this could even be improved to $1/t^2$ thanks to Nesterov's accelerated gradient descent. 

The next result interpolates between the $1/\sqrt{t}$ for stochastic smooth optimization, and the $1/t$ for deterministic smooth optimization. We will use it to propose a useful modification of SGD in the smooth case. The proof is extracted from \cite{DGBSX12}.

\begin{theorem} \label{th:SMDsmooth}
Let $\Phi$ be a mirror map $1$-strongly convex on $\mathcal{X} \cap \mathcal{D}$ w.r.t. $\|\cdot\|$, and let $R^2 = \sup_{x \in \mathcal{X} \cap \mathcal{D}} \Phi(x) - \Phi(x_1)$. Let $f$ be convex and $\beta$-smooth w.r.t. $\|\cdot\|$. Furthermore assume that the stochastic oracle is such that $\E \|\nabla f(x) - \tg(x)\|_*^2 \leq \sigma^2$. Then S-MD with stepsize $\frac{1}{\beta + 1/\eta}$ and $\eta = \frac{R}{\sigma} \sqrt{\frac{2}{t}}$ satisfies
$$\E f\bigg(\frac{1}{t} \sum_{s=1}^t x_{s+1} \bigg) - f(x^*) \leq R \sigma \sqrt{\frac{2}{t}} + \frac{\beta R^2}{t} .$$
\end{theorem}

\begin{proof}
Using $\beta$-smoothness, Cauchy-Schwarz (with $2 ab \leq x a^2+ b^2 / x$ for any $x >0$), and the 1-strong convexity of $\Phi$, one obtains
\begin{align*}
& f(x_{s+1}) - f(x_s) \\
& \leq \nabla f(x_s)^{\top} (x_{s+1} - x_s) + \frac{\beta}{2} \|x_{s+1} - x_s\|^2 \\
& = \tg_s^{\top} (x_{s+1} - x_s) + (\nabla f(x_s) - \tg_s)^{\top} (x_{s+1} - x_s) + \frac{\beta}{2} \|x_{s+1} - x_s\|^2 \\
& \leq \tg_s^{\top} (x_{s+1} - x_s) + \frac{\eta}{2} \|\nabla f(x_s) - \tg_s\|_*^2 + \frac12 (\beta + 1/\eta) \|x_{s+1} - x_s\|^2 \\
& \leq \tg_s^{\top} (x_{s+1} - x_s) + \frac{\eta}{2} \|\nabla f(x_s) - \tg_s\|_*^2 + (\beta + 1/\eta) D_{\Phi}(x_{s+1}, x_s) .
\end{align*}
Observe that, using the same argument as to derive \eqref{eq:pourplustard1}, one has
$$\frac{1}{\beta + 1/\eta} \tg_s^{\top} (x_{s+1} - x^*) \leq D_{\Phi} (x^*, x_s) - D_{\Phi}(x^*, x_{s+1}) - D_{\Phi}(x_{s+1}, x_s) .$$
Thus
\begin{align*}
& f(x_{s+1}) \\
 & \leq f(x_s) + \tg_s^{\top}(x^* - x_s) + (\beta + 1/\eta) \left(D_{\Phi} (x^*, x_s) - D_{\Phi}(x^*, x_{s+1})\right) \\
& \qquad + \frac{\eta}{2} \|\nabla f(x_s) - \tg_s\|_*^2 \\
& \leq f(x^*) + (\tg_s-\nabla f(x_s))^{\top}(x^* - x_s) \\
& \qquad + (\beta + 1/\eta) \left(D_{\Phi} (x^*, x_s) - D_{\Phi}(x^*, x_{s+1})\right) + \frac{\eta}{2} \|\nabla f(x_s) - \tg_s\|_*^2 .
\end{align*}
In particular this yields
$$\E f(x_{s+1}) - f(x^*) \leq (\beta + 1/\eta) \E \left(D_{\Phi} (x^*, x_s) - D_{\Phi}(x^*, x_{s+1})\right) + \frac{\eta \sigma^2}{2} .$$
By summing this inequality from $s=1$ to $s=t$ one can easily conclude with the standard argument.
\end{proof}

We can now propose the following modification of SGD based on the idea of {\em mini-batches}. Let $m \in \N$, then mini-batch SGD iterates the following equation:
$$x_{t+1} = \Pi_{\cX}\left(x_t - \frac{\eta}{m} \sum_{i=1}^m \tg_i(x_t)\right).$$
where $\tg_i(x_t), i=1,\hdots,m$ are independent random variables (conditionally on $x_t$) obtained from repeated queries to the stochastic oracle. Assuming that $f$ is $\beta$-smooth and that the stochastic oracle is such that $\|\tg(x)\|_2 \leq B$, one can obtain a rate of convergence for mini-batch SGD with Theorem \ref{th:SMDsmooth}. Indeed one can apply this result with the modified stochastic oracle that returns $\frac{1}{m} \sum_{i=1}^m \tg_i(x)$, it satisfies
$$\E \| \frac1{m} \sum_{i=1}^m \tg_i(x) - \nabla f(x) \|_2^2 = \frac{1}{m}\E \| \tg_1(x) - \nabla f(x) \|_2^2 \leq \frac{2 B^2}{m} .$$
Thus one obtains that with $t$ calls to the (original) stochastic oracle, that is $t/m$ iterations of the mini-batch SGD, one has a suboptimality gap bounded by
$$R \sqrt{\frac{2 B^2}{m}} \sqrt{\frac{2}{t/m}} + \frac{\beta R^2}{t/m} = 2 \frac{R B}{\sqrt{t}} + \frac{m \beta R^2}{t} .$$
Thus as long as $m \leq \frac{B}{R \beta} \sqrt{t}$ one obtains, with mini-batch SGD and $t$ calls to the oracle, a point which is $3\frac{R B}{\sqrt{t}}$-optimal.

Mini-batch SGD can be a better option than basic SGD in at least two situations: (i) When the computation for an iteration of mini-batch SGD can be distributed between multiple processors. Indeed a central unit can send the message to the processors that estimates of the gradient at point $x_s$ have to be computed, then each processor can work independently and send back the estimate they obtained. (ii) Even in a serial setting mini-batch SGD can sometimes be advantageous, in particular if some calculations can be re-used to compute several estimated gradients at the same point.

\section{Sum of smooth and strongly convex functions}
Let us examine in more details the main example from Section \ref{sec:mlapps}. That is one is interested in the unconstrained minimization of 
$$f(x) = \frac1{m} \sum_{i=1}^m f_i(x) ,$$
where $f_1, \hdots, f_m$ are $\beta$-smooth and convex functions, and $f$ is $\alpha$-strongly convex. Typically in machine learning $\alpha$ can be as small as $1/m$, while $\beta$ is of order of a constant. In other words the condition number $\kappa= \beta / \alpha$ can be as large as $\Omega(m)$. Let us now compare the basic gradient descent, that is
$$x_{t+1} = x_t - \frac{\eta}{m} \sum_{i=1}^m \nabla f_i(x) ,$$
to SGD
$$x_{t+1} = x_t - \eta \nabla f_{i_t}(x) ,$$
where $i_t$ is drawn uniformly at random in $[m]$ (independently of everything else). Theorem \ref{th:gdssc} shows that gradient descent requires $O(m \kappa \log(1/\epsilon))$ gradient computations (which can be improved to $O(m \sqrt{\kappa} \log(1/\epsilon))$ with Nesterov's accelerated gradient descent), while Theorem \ref{th:sgdstrong} shows that SGD (with appropriate averaging) requires $O(1/ (\alpha \epsilon))$ gradient computations. Thus one can obtain a low accuracy solution reasonably fast with SGD, but for high accuracy the basic gradient descent is more suitable. Can we get the best of both worlds? This question was answered positively in \cite{LRSB12} with SAG (Stochastic Averaged Gradient) and in \cite{SSZ13} with SDCA (Stochastic Dual Coordinate Ascent). These methods require only $O((m+\kappa) \log(1/\epsilon))$ gradient computations. We describe below the SVRG (Stochastic Variance Reduced Gradient descent) algorithm from \cite{JZ13} which makes the main ideas of SAG and SDCA more transparent (see also \cite{DBLJ14} for more on the relation between these different methods). We also observe that a natural question is whether one can obtain a Nesterov's accelerated version of these algorithms that would need only $O((m + \sqrt{m \kappa}) \log(1/\epsilon))$, see \cite{SSZ13b, ZX14, AB14} for recent works on this question.

To obtain a linear rate of convergence one needs to make ``big steps", that is the step-size should be of order of a constant. In SGD the step-size is typically of order $1/\sqrt{t}$ because of the variance introduced by the stochastic oracle. The idea of SVRG is to ``center" the output of the stochastic oracle in order to reduce the variance. Precisely instead of feeding $\nabla f_{i}(x)$ into the gradient descent one would use $\nabla f_i(x) - \nabla f_i(y) + \nabla f(y)$ where $y$ is a centering sequence. This is a sensible idea since, when $x$ and $y$ are close to the optimum, one should have that $\nabla f_i(x) - \nabla f_i(y)$ will have a small variance, and of course $\nabla f(y)$ will also be small (note that $\nabla f_i(x)$ by itself is not necessarily small). This intuition is made formal with the following lemma.
\begin{lemma} \label{lem:SVRG}
Let $f_1, \hdots f_m$ be $\beta$-smooth convex functions on $\R^n$, and $i$ be a random variable uniformly distributed in $[m]$. Then
$$\E \| \nabla f_i(x) - \nabla f_i(x^*) \|_2^2 \leq 2 \beta (f(x) - f(x^*)) .$$
\end{lemma}

\begin{proof}
Let $g_i(x) = f_i(x) - f_i(x^*) - \nabla f_i(x^*)^{\top} (x - x^*)$. By convexity of $f_i$ one has $g_i(x) \geq 0$ for any $x$ and in particular using \eqref{eq:onestepofgd} this yields $- g_i(x) \leq - \frac{1}{2\beta} \|\nabla g_i(x)\|_2^2$ which can be equivalently written as
$$\| \nabla f_i(x) - \nabla f_i(x^*) \|_2^2 \leq 2 \beta (f_i(x) - f_i(x^*) - \nabla f_i(x^*)^{\top} (x - x^*)) .$$
Taking expectation with respect to $i$ and observing that $\E \nabla f_i(x^*) = \nabla f(x^*) = 0$ yields the claimed bound.
\end{proof}
On the other hand the computation of $\nabla f(y)$ is expensive (it requires $m$ gradient computations), and thus the centering sequence should be updated more rarely than the main sequence. These ideas lead to the following epoch-based algorithm.

Let $y^{(1)} \in \R^n$ be an arbitrary initial point. For $s=1, 2 \ldots$, let $x_1^{(s)}=y^{(s)}$. For $t=1, \hdots, k$ let 
$$x_{t+1}^{(s)} = x_t^{(s)} - \eta \left( \nabla f_{i_t^{(s)}}(x_t^{(s)}) - \nabla f_{i_t^{(s)}} (y^{(s)}) + \nabla f(y^{(s)}) \right) ,$$
where $i_t^{(s)}$ is drawn uniformly at random (and independently of everything else) in $[m]$. Also let
$$y^{(s+1)} = \frac1{k} \sum_{t=1}^k x_t^{(s)} .$$

\begin{theorem} \label{th:SVRG}
Let $f_1, \hdots f_m$ be $\beta$-smooth convex functions on $\R^n$ and $f$ be $\alpha$-strongly convex. Then SVRG with $\eta = \frac{1}{10\beta}$ and $k = 20 \kappa$ satisfies
$$\E f(y^{(s+1)}) - f(x^*) \leq 0.9^s (f(y^{(1)}) - f(x^*)) .$$
\end{theorem}

\begin{proof}
We fix a phase $s \geq 1$ and we denote by $\E$ the expectation taken with respect to $i_1^{(s)}, \hdots, i_k^{(s)}$. We show below that
$$\E f(y^{(s+1)}) - f(x^*) =  \E f\left(\frac1{k} \sum_{t=1}^k x_t^{(s)}\right) - f(x^*)  \leq 0.9 (f(y^{(s)}) - f(x^*)) ,$$
which clearly implies the theorem. To simplify the notation in the following we drop the dependency on $s$, that is we want to show that
\begin{equation} \label{eq:SVRG0}
\E f\left(\frac1{k} \sum_{t=1}^k x_t\right) - f(x^*)  \leq 0.9 (f(y) - f(x^*)) .
\end{equation}
We start as for the proof of Theorem \ref{th:gdssc} (analysis of gradient descent for smooth and strongly convex functions) with
\begin{equation} \label{eq:SVRG1}
\|x_{t+1} - x^*\|_2^2 = \|x_t - x^*\|_2^2 - 2 \eta v_t^{\top}(x_t - x^*) + \eta^2 \|v_t\|_2^2 ,
\end{equation}
where
$$v_t = \nabla f_{i_t}(x_t) - \nabla f_{i_t} (y) + \nabla f(y) .$$
Using Lemma \ref{lem:SVRG}, we upper bound $\E_{i_t} \|v_t\|_2^2$ as follows (also recall that $\E\|X-\E(X)\|_2^2 \leq \E\|X\|_2^2$, and $\E_{i_t} \nabla f_{i_t}(x^*) = 0$):
\begin{align}
& \E_{i_t} \|v_t\|_2^2 \notag \\
& \leq 2 \E_{i_t} \|\nabla f_{i_t}(x_t) - \nabla f_{i_t}(x^*) \|_2^2 + 2 \E_{i_t} \|\nabla f_{i_t}(y) - \nabla f_{i_t}(x^*) - \nabla f(y) \|_2^2 \notag \\
& \leq 2 \E_{i_t} \|\nabla f_{i_t}(x_t) - \nabla f_{i_t}(x^*) \|_2^2 + 2 \E_{i_t} \|\nabla f_{i_t}(y) - \nabla f_{i_t}(x^*) \|_2^2 \notag \\
& \leq 4 \beta (f(x_t) - f(x^*) + f(y) - f(x^*)) . \label{eq:SVRG2}
\end{align}
Also observe that
$$\E_{i_t} v_t^{\top}(x_t - x^*) = \nabla f(x_t)^{\top} (x_t - x^*) \geq f(x_t) - f(x^*) ,$$
and thus plugging this into \eqref{eq:SVRG1} together with \eqref{eq:SVRG2} one obtains
\begin{eqnarray*}
\E_{i_t} \|x_{t+1} - x^*\|_2^2 & \leq & \|x_t - x^*\|_2^2 - 2 \eta (1 - 2 \beta \eta) (f(x_t) - f(x^*)) \\
& & + 4 \beta \eta^2 (f(y) - f(x^*)) .
\end{eqnarray*}
Summing the above inequality over $t=1, \hdots, k$ yields
\begin{eqnarray*} 
\E \|x_{k+1} - x^*\|_2^2 & \leq & \|x_1 - x^*\|_2^2 - 2 \eta (1 - 2 \beta \eta) \E \sum_{t=1}^k (f(x_t) - f(x^*)) \\
& & + 4 \beta \eta^2 k (f(y) - f(x^*)) .
\end{eqnarray*}
Noting that $x_1 = y$ and that by $\alpha$-strong convexity one has $f(x) - f(x^*) \geq \frac{\alpha}{2} \|x - x^*\|_2^2$, one can rearrange the above display to obtain
$$\E f\left(\frac1{k} \sum_{t=1}^k x_t\right) - f(x^*)  \leq \left(\frac{1}{\alpha \eta (1 - 2 \beta \eta) k} + \frac{2 \beta \eta}{1- 2\beta \eta} \right) (f(y) - f(x^*)) .$$
Using that $\eta = \frac{1}{10\beta}$ and $k = 20 \kappa$ finally yields \eqref{eq:SVRG0} which itself concludes the proof.
\end{proof}

\section{Random coordinate descent}
We assume throughout this section that $f$ is a convex and differentiable function on $\R^n$, with a unique\footnote{Uniqueness is only assumed for sake of notation.} minimizer $x^*$. We investigate one of the simplest possible scheme to optimize $f$, the random coordinate descent (RCD) method. In the following we denote $\nabla_i f(x) = \frac{\partial f}{\partial x_i} (x)$. RCD is defined as follows, with an arbitrary initial point $x_1 \in \R^n$,
$$x_{s+1} = x_s - \eta \nabla_{i_s} f(x) e_{i_s} ,$$
where $i_s$ is drawn uniformly at random from $[n]$ (and independently of everything else). 

One can view RCD as SGD with the specific oracle $\tg(x) = n \nabla_{I} f(x) e_I$ where $I$ is drawn uniformly at random from $[n]$. Clearly $\E \tg(x) = \nabla f(x)$, and furthermore
$$\E \|\tg(x)\|_2^2 = \frac{1}{n}\sum_{i=1}^n \|n \nabla_{i} f(x) e_i\|_2^2 = n \|\nabla f(x)\|_2^2 .$$
Thus using Theorem \ref{th:SMD} (with $\Phi(x) = \frac12 \|x\|_2^2$, that is S-MD being SGD) one immediately obtains the following result. 
\begin{theorem}
Let $f$ be convex and $L$-Lipschitz on $\R^n$, then RCD with $\eta = \frac{R}{L} \sqrt{\frac{2}{n t}}$ satisfies
$$\E f\bigg(\frac{1}{t} \sum_{s=1}^t x_s \bigg) - \min_{x \in \mathcal{X}} f(x) \leq R L \sqrt{\frac{2 n}{t}} .$$
\end{theorem}
Somewhat unsurprisingly RCD requires $n$ times more iterations than gradient descent to obtain the same accuracy. In the next section, we will see that this statement can be greatly improved by taking into account directional smoothness.

\subsection{RCD for coordinate-smooth optimization}
We assume now directional smoothness for $f$, that is there exists $\beta_1, \hdots, \beta_n$ such that for any $i \in [n], x \in \R^n$ and $u \in \R$,
$$| \nabla_i f(x+u e_i) - \nabla_i f(x) | \leq \beta_i |u| .$$
If $f$ is twice differentiable then this is equivalent to $(\nabla^2 f(x))_{i,i} \leq \beta_i$. In particular, since the maximal eigenvalue of a matrix is upper bounded by its trace, one can see that the directional smoothness implies that $f$ is $\beta$-smooth with $\beta \leq \sum_{i=1}^n \beta_i$. We now study the following ``aggressive" RCD, where the step-sizes are of order of the inverse smoothness:
$$x_{s+1} = x_s - \frac{1}{\beta_{i_s}} \nabla_{i_s} f(x) e_{i_s} .$$
Furthermore we study a more general sampling distribution than uniform, precisely for $\gamma \geq 0$ we assume that $i_s$ is drawn (independently) from the distribution $p_{\gamma}$ defined by
$$p_{\gamma}(i) = \frac{\beta_i^{\gamma}}{\sum_{j=1}^n \beta_j^{\gamma}}, i \in [n] .$$
This algorithm was proposed in \cite{Nes12}, and we denote it by RCD($\gamma$). Observe that, up to a preprocessing step of complexity $O(n)$, one can sample from $p_{\gamma}$ in time $O(\log(n))$. 

The following rate of convergence is derived in \cite{Nes12}, using the dual norms $\|\cdot\|_{[\gamma]}, \|\cdot\|_{[\gamma]}^*$ defined by
$$\|x\|_{[\gamma]} = \sqrt{\sum_{i=1}^n \beta_i^{\gamma} x_i^2} , \;\; \text{and} \;\; \|x\|_{[\gamma]}^* = \sqrt{\sum_{i=1}^n \frac1{\beta_i^{\gamma}} x_i^2} .$$

\begin{theorem} \label{th:rcdgamma}
Let $f$ be convex and such that $u \in \R \mapsto f(x + u e_i)$ is $\beta_i$-smooth for any $i \in [n], x \in \R^n$. Then RCD($\gamma$) satisfies for $t \geq 2$,
$$\E f(x_{t}) - f(x^*) \leq \frac{2 R_{1 - \gamma}^2(x_1) \sum_{i=1}^n \beta_i^{\gamma}}{t-1} ,$$
where
$$R_{1-\gamma}(x_1) = \sup_{x \in \R^n : f(x) \leq f(x_1)} \|x - x^*\|_{[1-\gamma]} .$$
\end{theorem}
Recall from Theorem \ref{th:gdsmooth} that in this context the basic gradient descent attains a rate of $\beta \|x_1 - x^*\|_2^2 / t$ where $\beta \leq \sum_{i=1}^n \beta_i$ (see the discussion above). Thus we see that RCD($1$) greatly improves upon gradient descent for functions where $\beta$ is of order of $\sum_{i=1}^n \beta_i$. Indeed in this case both methods attain the same accuracy after a fixed number of iterations, but the iterations of coordinate descent are potentially much cheaper than the iterations of gradient descent. 
\begin{proof}
By applying \eqref{eq:onestepofgd} to the $\beta_i$-smooth function $u \in \R \mapsto f(x + u e_i)$ one obtains
$$f\left(x - \frac{1}{\beta_i} \nabla_i f(x) e_i\right) - f(x) \leq - \frac{1}{2 \beta_i} (\nabla_i f(x))^2 .$$
We use this as follows:
\begin{eqnarray*}
\E_{i_s} f(x_{s+1}) - f(x_s)
& = & \sum_{i=1}^n p_{\gamma}(i) \left(f\left(x_s - \frac{1}{\beta_i} \nabla_i f(x_s) e_i\right) - f(x_s) \right) \\
& \leq & - \sum_{i=1}^n \frac{p_{\gamma}(i)}{2 \beta_i} (\nabla_i f(x_s))^2 \\
& = & - \frac{1}{2 \sum_{i=1}^n \beta_i^{\gamma}} \left(\|\nabla f(x_s)\|_{[1-\gamma]}^*\right)^2 .
\end{eqnarray*}
Denote $\delta_s = \E f(x_s) - f(x^*)$. Observe that the above calculation can be used to show that $f(x_{s+1}) \leq f(x_s)$ and thus one has, by definition of $R_{1-\gamma}(x_1)$,
\begin{eqnarray*} 
\delta_s & \leq & \nabla f(x_s)^{\top} (x_s - x^*) \\
& \leq & \|x_s - x^*\|_{[1-\gamma]} \|\nabla f(x_s)\|_{[1-\gamma]}^* \\
& \leq & R_{1-\gamma}(x_1) \|\nabla f(x_s)\|_{[1-\gamma]}^* .
\end{eqnarray*}
Thus putting together the above calculations one obtains
$$\delta_{s+1} \leq \delta_s - \frac{1}{2 R_{1 - \gamma}^2(x_1) \sum_{i=1}^n \beta_i^{\gamma} } \delta_s^2 .$$
The proof can be concluded with similar computations than for Theorem \ref{th:gdsmooth}.
\end{proof}

We discussed above the specific case of $\gamma = 1$. Both $\gamma=0$ and $\gamma=1/2$ also have an interesting behavior, and we refer to \cite{Nes12} for more details. The latter paper also contains a discussion of high probability results and potential acceleration \`a la Nesterov. We also refer to \cite{RT12} for a discussion of RCD in a distributed setting.

\subsection{RCD for smooth and strongly convex optimization}
If in addition to directional smoothness one also assumes strong convexity, then RCD attains in fact a linear rate.
\begin{theorem} \label{th:linearratercd}
Let $\gamma \geq 0$. Let $f$ be $\alpha$-strongly convex w.r.t. $\|\cdot\|_{[1-\gamma]}$, and such that $u \in \R \mapsto f(x + u e_i)$ is $\beta_i$-smooth for any $i \in [n], x \in \R^n$. Let $\kappa_{\gamma} = \frac{\sum_{i=1}^n \beta_i^{\gamma}}{\alpha}$, then RCD($\gamma$) satisfies
$$\E f(x_{t+1}) - f(x^*) \leq \left(1 - \frac1{\kappa_{\gamma}}\right)^t (f(x_1) - f(x^*)) .$$
\end{theorem}
We use the following elementary lemma.
\begin{lemma} \label{lem:tittrucnes}
Let $f$ be $\alpha$-strongly convex w.r.t. $\| \cdot\|$ on $\R^n$, then
$$f(x) - f(x^*) \leq \frac1{2\alpha} \|\nabla f(x)\|_*^2 .$$
\end{lemma}
\begin{proof}
By strong convexity, H{\"o}lder's inequality, and an elementary calculation,
\begin{eqnarray*}
f(x) - f(y) & \leq & \nabla f(x)^{\top} (x-y) - \frac{\alpha}{2} \|x-y\|_2^2 \\
& \leq & \|\nabla f(x)\|_* \|x-y\| - \frac{\alpha}{2} \|x-y\|_2^2 \\
& \leq & \frac1{2\alpha} \|\nabla f(x)\|_*^2 ,
\end{eqnarray*}
which concludes the proof by taking $y = x^*$.
\end{proof}
We can now prove Theorem \ref{th:linearratercd}.
\begin{proof}
In the proof of Theorem \ref{th:rcdgamma} we showed that 
$$\delta_{s+1} \leq \delta_s - \frac{1}{2 \sum_{i=1}^n \beta_i^{\gamma}} \left(\|\nabla f(x_s)\|_{[1-\gamma]}^*\right)^2 .$$
On the other hand Lemma \ref{lem:tittrucnes} shows that 
$$\left(\|\nabla f(x_s)\|_{[1-\gamma]}^*\right)^2 \geq 2 \alpha \delta_s .$$
The proof is concluded with straightforward calculations.
\end{proof}

\section{Acceleration by randomization for saddle points}
We explore now the use of randomness for saddle point computations. That is we consider the context of Section \ref{sec:sp} with a stochastic oracle of the following form: given $z=(x,y) \in \cX \times \cY$ it outputs $\tg(z) = (\tg_{\cX}(x,y), \tg_{\cY}(x,y))$ where $\E \ (\tg_{\cX}(x,y) | x,y) \in \partial_x \phi(x,y)$, and $\E \ (\tg_{\cY}(x,y) | x,y) \in \partial_y (-\phi(x,y))$. Instead of using true subgradients as in SP-MD (see Section \ref{sec:spmd}) we use here the outputs of the stochastic oracle. We refer to the resulting algorithm as S-SP-MD (Stochastic Saddle Point Mirror Descent). Using the same reasoning than in Section \ref{sec:smd} and Section \ref{sec:spmd} one can derive the following theorem.
\begin{theorem} \label{th:sspmd}
Assume that the stochastic oracle is such that $\E \left(\|\tg_{\cX}(x,y)\|_{\cX}^* \right)^2 \leq B_{\cX}^2$, and $\E \left(\|\tg_{\cY}(x,y)\|_{\cY}^* \right)^2 \leq B_{\cY}^2$. Then S-SP-MD with $a= \frac{B_{\cX}}{R_{\cX}}$, $b=\frac{B_{\cY}}{R_{\cY}}$, and $\eta=\sqrt{\frac{2}{t}}$ satisfies
$$\E \left( \max_{y \in \mathcal{Y}} \phi\left( \frac1{t} \sum_{s=1}^t x_s,y \right) - \min_{x \in \mathcal{X}} \phi\left(x, \frac1{t} \sum_{s=1}^t y_s \right) \right) \leq (R_{\cX} B_{\cX} + R_{\cY} B_{\cY}) \sqrt{\frac{2}{t}}.$$
\end{theorem}
Using S-SP-MD we revisit the examples of Section \ref{sec:spex2} and Section \ref{sec:spex3}. In both cases one has $\phi(x,y) = x^{\top} A y$ (with $A_i$ being the $i^{th}$ column of $A$), and thus $\nabla_x \phi(x,y) = Ay$ and $\nabla_y \phi(x,y) = A^{\top} x$.
\newline

\noindent
\textbf{Matrix games.} Here $x \in \Delta_n$ and $y \in \Delta_m$. Thus there is a quite natural stochastic oracle:
\begin{equation} \label{eq:oraclematrixgame}
\tg_{\cX}(x,y) = A_I, \; \text{where} \; I \in [m] \; \text{is drawn according to} \; y \in \Delta_m ,
\end{equation}
and $\forall i \in [m]$,
\begin{equation} \label{eq:oraclematrixgame2}
\tg_{\cY}(x,y)(i) = A_i(J), \; \text{where} \; J \in [n] \; \text{is drawn according to} \; x \in \Delta_n .
\end{equation}
Clearly $\|\tg_{\cX}(x,y)\|_{\infty} \leq \|A\|_{\mathrm{max}}$ and $\|\tg_{\cX}(x,y)\|_{\infty} \leq \|A\|_{\mathrm{max}}$, which implies that S-SP-MD attains an $\epsilon$-optimal pair of points with $O\left(\|A\|_{\mathrm{max}}^2 \log(n+m) / \epsilon^2 \right)$ iterations. Furthermore the computational complexity of a step of S-SP-MD is dominated by drawing the indices $I$ and $J$ which takes $O(n + m)$. Thus overall the complexity of getting an $\epsilon$-optimal Nash equilibrium with S-SP-MD is $O\left(\|A\|_{\mathrm{max}}^2 (n + m) \log(n+m) / \epsilon^2  \right)$. While the dependency on $\epsilon$ is worse than for SP-MP (see Section \ref{sec:spex2}), the dependencies on the dimensions is $\tilde{O}(n+m)$ instead of $\tilde{O}(nm)$. In particular, quite astonishingly, this is {\em sublinear} in the size of the matrix $A$. The possibility of sublinear algorithms for this problem was first observed in \cite{GK95}.
\newline

\noindent
\textbf{Linear classification.} Here $x \in \mB_{2,n}$ and $y \in \Delta_m$. Thus the stochastic oracle for the $x$-subgradient can be taken as in \eqref{eq:oraclematrixgame} but for the $y$-subgradient we modify \eqref{eq:oraclematrixgame2} as follows. For a vector $x$ we denote by $x^2$ the vector such that $x^2(i) = x(i)^2$. For all $i \in [m]$,
$$\tg_{\cY}(x,y)(i) = \frac{\|x\|^2}{x(j)} A_i(J), \; \text{where} \; J \in [n] \; \text{is drawn according to} \; \frac{x^2}{\|x\|_2^2} \in \Delta_n .$$ 
Note that one indeed has $\E (\tg_{\cY}(x,y)(i) | x,y) = \sum_{j=1}^n x(j) A_i(j) = (A^{\top} x)(i)$.
Furthermore $\|\tg_{\cX}(x,y)\|_2 \leq B$, and
$$\E (\|\tg_{\cY}(x,y)\|_{\infty}^2 | x,y) = \sum_{j=1}^n \frac{x(j)^2}{\|x\|_2^2} \max_{i \in [m]} \left(\frac{\|x\|^2}{x(j)} A_i(j)\right)^2 \leq \sum_{j=1}^n \max_{i \in [m]} A_i(j)^2 .$$
Unfortunately this last term can be $O(n)$. However it turns out that one can do a more careful analysis of mirror descent in terms of local norms, which allows to prove that the ``local variance" is dimension-free. We refer to \cite{BC12} for more details on these local norms, and to \cite{CHW12} for the specific details in the linear classification situation.

\section{Convex relaxation and randomized rounding} \label{sec:convexrelaxation}
In this section we briefly discuss the concept of convex relaxation, and the use of randomization to find approximate solutions. By now there is an enormous literature on these topics, and we refer to \cite{Bar14} for further pointers. 

We study here the seminal example of $\mathrm{MAXCUT}$. This problem can be described as follows. Let $A \in \R_+^{n \times n}$ be a symmetric matrix of non-negative weights. The entry $A_{i,j}$ is interpreted as a measure of the ``dissimilarity" between point $i$ and point $j$. The goal is to find a partition of $[n]$ into two sets, $S \subset [n]$ and $S^c$, so as to maximize the total dissimilarity between the two groups: $\sum_{i \in S, j \in S^c} A_{i,j}$. Equivalently $\mathrm{MAXCUT}$ corresponds to the following optimization problem:
\begin{equation} \label{eq:maxcut1}
\max_{x \in \{-1,1\}^n} \frac12 \sum_{i,j =1}^n A_{i,j} (x_i - x_j)^2 .
\end{equation}
Viewing $A$ as the (weighted) adjacency matrix of a graph, one can rewrite \eqref{eq:maxcut1} as follows, using the graph Laplacian $L=D-A$ where $D$ is the diagonal matrix with entries $(\sum_{j=1}^n A_{i,j})_{i \in [n]}$,
\begin{equation} \label{eq:maxcut2}
\max_{x \in \{-1,1\}^n} x^{\top} L x .
\end{equation}
It turns out that this optimization problem is $\mathbf{NP}$-hard, that is the existence of a polynomial time algorithm to solve \eqref{eq:maxcut2} would prove that $\mathbf{P} = \mathbf{NP}$. The combinatorial difficulty of this problem stems from the hypercube constraint. Indeed if one replaces $\{-1,1\}^n$ by the Euclidean sphere, then one obtains an efficiently solvable problem (it is the problem of computing the maximal eigenvalue of $L$).

We show now that, while \eqref{eq:maxcut2} is a difficult optimization problem, it is in fact possible to find relatively good {\em approximate} solutions by using the power of randomization. 
Let $\zeta$ be uniformly drawn on the hypercube $\{-1,1\}^n$, then clearly
$$\E \ \zeta^{\top} L \zeta = \sum_{i,j=1, i \neq j}^n A_{i,j} \geq \frac{1}{2} \max_{x \in \{-1,1\}^n} x^{\top} L x .$$
This means that, on average, $\zeta$ is a $1/2$-approximate solution to \eqref{eq:maxcut2}. Furthermore it is immediate that the above expectation bound implies that, with probability at least $\epsilon$, $\zeta$ is a $(1/2-\epsilon)$-approximate solution. Thus by repeatedly sampling uniformly from the hypercube one can get arbitrarily close (with probability approaching $1$) to a $1/2$-approximation of $\mathrm{MAXCUT}$.

Next we show that one can obtain an even better approximation ratio by combining the power of convex optimization and randomization. This approach was pioneered by \cite{GW95}. The Goemans-Williamson algorithm is based on the following inequality
$$\max_{x \in \{-1,1\}^n} x^{\top} L x = \max_{x \in \{-1,1\}^n} \langle L, xx^{\top} \rangle \leq \max_{X \in \mathbb{S}_+^n, X_{i,i}=1, i \in [n]} \langle L, X \rangle .$$ 
The right hand side in the above display is known as the {\em convex (or SDP) relaxation} of $\mathrm{MAXCUT}$. The convex relaxation is an SDP and thus one can find its solution efficiently with Interior Point Methods (see Section \ref{sec:IPM}). The following result states both the Goemans-Williamson strategy and the corresponding approximation ratio.

\begin{theorem} \label{th:GW}
Let $\Sigma$ be the solution to the SDP relaxation of $\mathrm{MAXCUT}$. Let $\xi \sim \cN(0, \Sigma)$ and $\zeta = \mathrm{sign}(\xi) \in \{-1,1\}^n$. Then
$$\E \ \zeta^{\top} L \zeta \geq 0.878 \max_{x \in \{-1,1\}^n} x^{\top} L x .$$
\end{theorem}

The proof of this result is based on the following elementary geometric lemma.

\begin{lemma} \label{lem:GW}
Let $\xi \sim \mathcal{N}(0,\Sigma)$ with $\Sigma_{i,i}=1$ for $i \in [n]$, and $\zeta = \mathrm{sign}(\xi)$. Then
$$\E \ \zeta_i \zeta_j = \frac{2}{\pi} \mathrm{arcsin} \left(\Sigma_{i,j}\right) .$$
\end{lemma}

\begin{proof}
Let $V \in \R^{n \times n}$ (with $i^{th}$ row $V_i^{\top}$) be such that $\Sigma = V V^{\top}$. Note that since $\Sigma_{i,i}=1$ one has $\|V_i\|_2 = 1$ (remark also that necessarily $|\Sigma_{i,j}| \leq 1$, which will be important in the proof of Theorem \ref{th:GW}). Let $\epsilon \sim \mathcal{N}(0,\mI_n)$ be such that $\xi = V \epsilon$. Then $\zeta_i = \mathrm{sign}(V_i^{\top} \epsilon)$, and in particular
\begin{eqnarray*}
\E \ \zeta_i \zeta_j & = & \P(V_i^{\top} \epsilon \geq 0 \ \text{and} \ V_j^{\top} \epsilon \geq 0) + \P(V_i^{\top} \epsilon \leq 0 \ \text{and} \ V_j^{\top} \epsilon \leq 0 \\
& & - \P(V_i^{\top} \epsilon \geq 0 \ \text{and} \ V_j^{\top} \epsilon < 0) - \P(V_i^{\top} \epsilon < 0 \ \text{and} \ V_j^{\top} \epsilon \geq 0) \\
& = & 2 \P(V_i^{\top} \epsilon \geq 0 \ \text{and} \ V_j^{\top} \epsilon \geq 0) - 2 \P(V_i^{\top} \epsilon \geq 0 \ \text{and} \ V_j^{\top} \epsilon < 0) \\
& = & \P(V_j^{\top} \epsilon \geq 0 | V_i^{\top} \epsilon \geq 0) - \P(V_j^{\top} \epsilon < 0 | V_i^{\top} \epsilon \geq 0) \\
& = & 1 - 2 \P(V_j^{\top} \epsilon < 0 | V_i^{\top} \epsilon \geq 0).
\end{eqnarray*}
Now a quick picture shows that $\P(V_j^{\top} \epsilon < 0 | V_i^{\top} \epsilon \geq 0) = \frac{1}{\pi} \mathrm{arccos}(V_i^{\top} V_j)$ (recall that $\epsilon / \|\epsilon\|_2$ is uniform on the Euclidean sphere). Using the fact that $V_i^{\top} V_j = \Sigma_{i,j}$ and $\mathrm{arccos}(x) = \frac{\pi}{2} - \mathrm{arcsin}(x)$ conclude the proof.
\end{proof}

We can now get to the proof of Theorem \ref{th:GW}.

\begin{proof}
We shall use the following inequality:
\begin{equation} \label{eq:dependsonL}
1 - \frac{2}{\pi} \mathrm{arcsin}(t) \geq 0.878 (1-t), \ \forall t \in [-1,1] .
\end{equation}
Also remark that for $X \in \R^{n \times n}$ such that $X_{i,i}=1$, one has
$$\langle L, X \rangle = \sum_{i,j=1}^n A_{i,j} (1 - X_{i,j}) ,$$
and in particular for $x \in \{-1,1\}^n$, $x^{\top} L x = \sum_{i,j=1}^n A_{i,j} (1 - x_i x_j)$.
Thus, using Lemma \ref{lem:GW}, and the facts that $A_{i,j} \geq 0$ and $|\Sigma_{i,j}| \leq 1$ (see the proof of Lemma \ref{lem:GW}), one has
\begin{eqnarray*}
\E \ \zeta^{\top} L \zeta
& = & \sum_{i,j=1}^n A_{i,j} \left(1- \frac{2}{\pi} \mathrm{arcsin} \left(\Sigma_{i,j}\right)\right)  \\
& \geq & 0.878 \sum_{i,j=1}^n A_{i,j} \left(1- \Sigma_{i,j}\right) \\
& = & 0.878 \ \max_{X \in \mathbb{S}_+^n, X_{i,i}=1, i \in [n]} \langle L, X \rangle \\
& \geq & 0.878 \max_{x \in \{-1,1\}^n} x^{\top} L x .
\end{eqnarray*}
\end{proof}

Theorem \ref{th:GW} depends on the form of the Laplacian $L$ (insofar as \eqref{eq:dependsonL} was used). We show next a result from \cite{Nes97} that applies to any positive semi-definite matrix, at the expense of the constant of approximation. Precisely we are now interested in the following optimization problem:
\begin{equation} \label{eq:quad}
\max_{x \in \{-1,1\}^n} x^{\top} B x .
\end{equation}
The corresponding SDP relaxation is
$$\max_{X \in \mathbb{S}_+^n, X_{i,i}=1, i \in [n]} \langle B, X \rangle .$$

\begin{theorem}
Let $\Sigma$ be the solution to the SDP relaxation of \eqref{eq:quad}. Let $\xi \sim \cN(0, \Sigma)$ and $\zeta = \mathrm{sign}(\xi) \in \{-1,1\}^n$. Then
$$\E \ \zeta^{\top} B \zeta \geq \frac{2}{\pi} \max_{x \in \{-1,1\}^n} x^{\top} B x .$$
\end{theorem}

\begin{proof}
Lemma \ref{lem:GW} shows that
$$\E \ \zeta^{\top} B \zeta = \sum_{i,j=1}^n B_{i,j} \frac{2}{\pi} \mathrm{arcsin} \left(X_{i,j}\right) = \frac{2}{\pi} \langle B, \mathrm{arcsin}(X) \rangle .$$
Thus to prove the result it is enough to show that $\langle B, \mathrm{arcsin}(\Sigma) \rangle \geq \langle B, \Sigma \rangle$, which is itself implied by $\mathrm{arcsin}(\Sigma) \succeq \Sigma$ (the implication is true since $B$ is positive semi-definite, just write the eigendecomposition). Now we prove the latter inequality via a Taylor expansion. Indeed recall that $|\Sigma_{i,j}| \leq 1$ and thus denoting by $A^{\circ \alpha}$ the matrix where the entries are raised to the power $\alpha$ one has
$$\mathrm{arcsin}(\Sigma) = \sum_{k=0}^{+\infty} \frac{{2k \choose k}}{4^k (2k +1)} \Sigma^{\circ (2k+1)} = \Sigma + \sum_{k=1}^{+\infty} \frac{{2k \choose k}}{4^k (2k +1)} \Sigma^{\circ (2k+1)}.$$
Finally one can conclude using the fact if $A,B \succeq 0$ then $A \circ B \succeq 0$. This can be seen by writing $A= V V^{\top}$, $B=U U^{\top}$, and thus 
$$(A \circ B)_{i,j} = V_i^{\top} V_j U_i^{\top} U_j = \mathrm{Tr}(U_j V_j^{\top} V_i U_i^{\top}) = \langle V_i U_i^{\top}, V_j U_j^{\top} \rangle .$$ In other words $A \circ B$ is a Gram-matrix and, thus it is positive semi-definite.
\end{proof}

\section{Random walk based methods} \label{sec:rwmethod}
Randomization naturally suggests itself in the center of gravity method (see Section \ref{sec:gravity}), as a way to circumvent the exact calculation of the center of gravity. This idea was proposed and developed in \cite{BerVem04}. We give below a condensed version of the main ideas of this paper.

Assuming that one can draw independent points $X_1, \hdots, X_N$ uniformly at random from the current set $\cS_t$, one could replace $c_t$ by $\hat{c}_t = \frac{1}{N} \sum_{i=1}^N X_i$. \cite{BerVem04} proved the following generalization of Lemma \ref{lem:Gru60} for the situation where one cuts a convex set through a point close the center of gravity. Recall that a convex set $\cK$ is in isotropic position if $\E X = 0$ and $\E X X^{\top} = \mI_n$, where $X$ is a random variable drawn uniformly at random from $\cK$. Note in particular that this implies $\E \|X\|_2^2 = n$. We also say that $\cK$ is in near-isotropic position if $\frac{1}{2} \mI_n \preceq \E X X^{\top} \preceq \frac3{2} \mI_n$.
\begin{lemma} \label{lem:BerVem04}
Let $\cK$ be a convex set in isotropic position. Then for any $w \in \R^n, w \neq 0$, $z \in \R^n$, one has
$$\mathrm{Vol} \left( \cK \cap \{x \in \R^n : (x-z)^{\top} w \geq 0\} \right) \geq \left(\frac{1}{e} - \|z\|_2\right) \mathrm{Vol} (\cK) .$$
\end{lemma}
Thus if one can ensure that $\cS_t$ is in (near) isotropic position, and $\|c_t - \hat{c}_t\|_2$ is small (say smaller than $0.1$), then the randomized center of gravity method (which replaces $c_t$ by $\hat{c}_t$) will converge at the same speed than the original center of gravity method. 

Assuming that $\cS_t$ is in isotropic position one immediately obtains $\E \|c_t - \hat{c}_t\|_2^2 = \frac{n}{N}$, and thus by Chebyshev's inequality one has $\P(\|c_t - \hat{c}_t\|_2 > 0.1) \leq 100 \frac{n}{N}$. In other words with $N = O(n)$ one can ensure that the randomized center of gravity method makes progress on a constant fraction of the iterations (to ensure progress at every step one would need a larger value of $N$ because of an union bound, but this is unnecessary).

Let us now consider the issue of putting $\cS_t$ in near-isotropic position. Let $\hat{\Sigma}_t = \frac1{N} \sum_{i=1}^N (X_i-\hat{c}_t) (X_i-\hat{c}_t)^{\top}$. \cite{Rud99} showed that as long as $N= \tilde{\Omega}(n)$, one has with high probability (say at least probability $1-1/n^2$) that the set $\hat{\Sigma}_t^{-1/2} (\cS_t - \hat{c}_t)$ is in near-isotropic position.

Thus it only remains to explain how to sample from a near-isotropic convex set $\cK$. This is where random walk ideas come into the picture. The hit-and-run walk\footnote{Other random walks are known for this problem but hit-and-run is the one with the sharpest theoretical guarantees. Curiously we note that one of those walks is closely connected to projected gradient descent, see \cite{BEL15}.} is described as follows: at a point $x \in \cK$, let $\cL$ be a line that goes through $x$ in a direction taken uniformly at random, then move to a point chosen uniformly at random in $\cL \cap \cK$. \cite{Lov98} showed that if the starting point of the hit-and-run walk is chosen from a distribution ``close enough" to the uniform distribution on $\cK$, then after $O(n^3)$ steps the distribution of the last point is $\epsilon$ away (in total variation) from the uniform distribution on $\cK$. In the randomized center of gravity method one can obtain a good initial distribution for $\cS_t$ by using the distribution that was obtained for $\cS_{t-1}$. In order to initialize the entire process correctly we start here with $\cS_1 = [-L, L]^n \supset \cX$ (in Section \ref{sec:gravity} we used $\cS_1 = \cX$), and thus we also have to use a {\em separation oracle} at iterations where $\hat{c}_t \not\in \cX$, just like we did for the ellipsoid method (see Section \ref{sec:ellipsoid}).

Wrapping up the above discussion, we showed (informally) that to attain an $\epsilon$-optimal point with the randomized center of gravity method one needs: $\tilde{O}(n)$ iterations, each iterations requires $\tilde{O}(n)$ random samples from $\cS_t$ (in order to put it in isotropic position) as well as a call to either the separation oracle or the first order oracle, and each sample costs $\tilde{O}(n^3)$ steps of the random walk. Thus overall one needs $\tilde{O}(n)$ calls to the separation oracle and the first order oracle, as well as $\tilde{O}(n^5)$ steps of the random walk.

\addcontentsline{toc}{chapter}{Acknowledgements}
\begin{acknowledgements}
This text grew out of lectures given at Princeton University in 2013 and 2014. I would like to thank Mike Jordan for his support in this project.
My gratitude goes to the four reviewers, and especially the non-anonymous referee Francis Bach, whose comments have greatly helped
to situate this monograph in the vast optimization literature. Finally I am thankful to Philippe Rigollet for suggesting the new title (a previous version of the manuscript was titled ``Theory of Convex Optimization for Machine Learning"), and to Yin-Tat Lee for many insightful discussions about cutting-plane methods.
\end{acknowledgements}

\bibliographystyle{plainnat}
\bibliography{newbib}

\end{document}